\documentclass[11pt]{amsart}
\usepackage{graphicx} %

\usepackage[margin=1in]{geometry}

\usepackage{amsfonts}
\usepackage{amsthm}
\usepackage{amsmath}
\usepackage{amssymb}
\usepackage{amscd} 
\usepackage{enumitem}
\usepackage{algorithm2e}
\RestyleAlgo{ruled}
\usepackage{mathtools}
\usepackage[normalem]{ulem}
\usepackage[dvipsnames]{xcolor}

\usepackage{blkarray}
\usepackage{graphicx}
\usepackage{xcolor}
\usepackage{hyperref}
\usepackage{tikz}

\usepackage{multirow}

\usepackage{verbatim}
\usepackage{url}

\newcommand{\Q}{\mathbb Q}
\newcommand{\Z}{\mathbb Z}
\newcommand{\R}{\mathbb R}

\newcommand{\vw}{\mathbf w}
\newcommand{\vv}{\mathbf v}
\newcommand{\ve}{\mathbf e}
\newcommand{\vr}{\mathbf r}
\newcommand{\vO}{\mathbf O}

\newcommand{\vC}{\mathbf C}
\newcommand{\vD}{\mathbf D}
\newcommand{\vF}{\mathbf F}
\newcommand{\vG}{\mathbf G}

\newcommand{\J}{\mathrm{Jac}}

\newcommand{\FC}{\mathrm{FC}}
\newcommand{\hZ}{\Z_{\frac{1}{2}}}

\newcommand{\cB}{\mathcal B}
\newcommand{\cC}{\mathcal C}
\newcommand{\cD}{\mathcal D}
\newcommand{\cO}{\mathcal O}
\newcommand{\cQ}{\mathcal Q}
\newcommand{\cR}{\mathcal R}
\newcommand{\cT}{\mathcal T}

\newcommand{\bJ}[1]{[#1]_\mathrm{J}}
\newcommand{\bR}[1]{[#1]_\mathrm{R}}

\newcommand{\pcc}{\langle \pi(\cC)\rangle}

\newtheorem{theorem}{Theorem}[section]
\newtheorem{lemma}[theorem]{Lemma}

\newtheorem{corollary}[theorem]{Corollary}
\newtheorem{proposition}[theorem]{Proposition}

\newtheorem{thmA}{Theorem} %
\newtheorem*{thm*}{Theorem}

\theoremstyle{definition}
\newtheorem{definition}[theorem]{Definition}
\newtheorem{example}[theorem]{Example}

\theoremstyle{remark}
\newtheorem{remark}[theorem]{Remark}

\title[Torsor structures on spanning quasi-trees of ribbon graphs]{The Jacobian of a regular orthogonal matroid and torsor structures on spanning quasi-trees of ribbon graphs}
\author{Matthew Baker}
\email{mbaker@math.gatech.edu}
\address{School of Mathematics, Georgia Institute of Technology, Atlanta, USA}

\author{Changxin Ding}
\email{cding66@gatech.edu}
\address{School of Mathematics, Georgia Institute of Technology, Atlanta, USA}

\author{Donggyu Kim}
\email{donggyu@gatech.edu}
\address{School of Mathematics, Georgia Institute of Technology, Atlanta, USA}

\date{\today}

\begin{document}

\begin{abstract}
Previous work of Chan--Church--Grochow and Baker--Wang shows that the set of spanning trees in a plane graph $G$ is naturally a torsor for the Jacobian group of $G$. Informally, this means that the set of spanning trees of $G$ naturally forms a group, except that there is no distinguished identity element. We generalize this fact to graphs cellularly embedded on orientable surfaces of arbitrary genus, which can be identified with {\em ribbon graphs}. In this generalization, the set of spanning trees of $G$ is replaced by the set of spanning quasi-trees of the ribbon graph, and the Jacobian group of $G$ is replaced by the Jacobian group of the associated regular orthogonal matroid $M$ (which coincides with the critical group of the ribbon graph, in the sense of Merino, Moffatt, and Noble).
Our proof shows, more generally, that the family of ``BBY torsors'' constructed by Backman--Baker--Yuen and later generalized by Ding admit natural generalizations to (regular representations of) regular orthogonal matroids. In addition to shedding light on the role of planarity in the earlier work mentioned above, our results represent one of the first substantial applications of orthogonal matroids (also called ``even delta-matroids'' or ``Lagrangian orthogonal matroids'') to a natural combinatorial problem about graphs. 
\end{abstract}

\maketitle
\section{Introduction}

\subsection{Torsor structures on spanning trees of plane graphs}
\label{sec:torsor structures on spanning trees}
Given a connected finite graph $G$, there is --- in general --- no natural way to make the set $\cT(G)$ of spanning trees of $G$ into a group.

For example, if $G=C_3$ is a 3-cycle, then there is clearly no canonical choice for an ``identity tree''. On the other hand, if $C_3$ is embedded in the plane as an equilateral triangle and we orient the plane counterclockwise, then there is a natural simply transitive action of the group $\Z / 3\Z$ on $\cT(C_3)$ given by counterclockwise rotation. Moreover, the group $\Z/3\Z$ is isomorphic to the {\em Jacobian group}\footnote{Also referred to in the literature as the {\em critical group}, {\em sandpile group}, or {\em Picard group} of $G$.} of $C_3$.

For any connected graph $G$, the cardinality of its Jacobian group $\J(G)$ is equal to the number of spanning trees of $G$ (this is a version of {\em Kirchhoff's Matrix-Tree Theorem}). So a better question is whether there is a natural way to define a simply transitive action of $\J(G)$ on $\cT(G)$.

In general, the answer seems to be no (depending on precisely what one means by ``natural''). However, for graphs embedded in the plane, the answer is {\bf yes} (at least once we fix an orientation of the plane). There are at least three different ways to describe such an action:

\begin{enumerate}
    \item In \cite{CCG15}, Chan, Church, and Grochow construct, for any connected graph $G$, a family of simply transitive actions of $\J(G)$ on $\cT(G)$ using the {\em rotor-routing process}, and they show that for plane graphs the resulting action is independent of all choices (up to isomorphism).
    \item In \cite{BW2018}, the first author and Wang define a different family of simply transitive actions of $\J(G)$ on $\cT(G)$ using the {\em Bernardi process} and show that, for plane graphs, the resulting action is independent of all choices and coincides with the action defined in \cite{CCG15}. We refer collectively to these actions as ``Bernardi torsors''.\footnote{Intuitively, when a group $G$ acts simply and transitively on a set $X$, we can think of $X$ as being akin to a ``group without a distinguished identity element''. This is made precise by the theory of {\em torsors}. More formally, given a group $G$, a {\em $G$-torsor} is a set $X$ on which $G$ acts simply and transitively; for example, $G$ is itself a $G$-torsor via left multiplication. Once we fix a reference point in a $G$-torsor $X$, there is a natural isomorphism of $G$-torsors between $X$ and $G$ itself.}

    \item In \cite{BBY2019}, Backman, the first author, and Yuen define the Jacobian of a regular matroid $M$ and construct a family of simply transitive actions of $\J(M)$ on the set $\cB(M)$ of bases of $M$, which generalizes Yuen's work \cite{Yuen}. We refer collectively to these actions as ``BBY torsors''. When $M = M(G)$ is the cycle matroid associated to a plane graph $G$, using \cite[Theorem 20]{Yuen}, the authors of \cite{BBY2019} show that there is a canonical choice for the action of $\J(M(G))$ on $\cB(M(G)) = \cT(G)$, and that this action coincides with the one defined in \cite{CCG15} and \cite{BW2018}, as described above.

\end{enumerate}

\subsection{An extension to the non-planar case}

For a while after the publication of \cite{CCG15,BW2018,BBY2019}, it seemed that the considerations in \S\ref{sec:torsor structures on spanning trees} might be the end of the story: for a plane graph $G$ there is a natural way to define a simply transitive group action on $\cT(G)$, and for a non-planar graph there is not.
However, when we came across the paper \cite{MMN2023} by Merino, Moffatt, and Noble, we realized that it should be possible to generalize the results of \cite{CCG15,BW2018,BBY2019} to the non-planar case from their new perspective. 

The key idea is to consider connected graphs $G$ embedded\footnote{Technically speaking, we need $G$ to be {\em cellularly embedded}, see \S\ref{subsection: ribbon} for the definition.} on an arbitrary compact and connected oriented surface $\Sigma$, replacing the Jacobian group of $G$ with a group that depends on the embedding and the set of spanning trees of $G$ with a larger set that also depends on the embedding. We refer to such a pair $\vG = (G,\Sigma)$ as a {\em ribbon graph}\footnote{Ribbon graphs are also known in the literature as {\em combinatorial maps} or {\em rotation systems}.}. 
The reason for the name ``ribbon graph'' is that one often works instead with an $\epsilon$-thickening of the graph $G$ in $\Sigma$ for a suitably small $\epsilon > 0$, which is a 2-dimensional manifold-with-boundary that resembles a ribbon. We use the two viewpoints interchangeably in this paper.

When $\Sigma$ is a sphere, the ribbon graph $\vG$ encodes the same data as $G$ together with a plane embedding and an orientation of the plane.
However, {\em every} graph $G$ can be realized as a ribbon graph $\vG = (G,\Sigma)$ for a suitable surface $\Sigma$. 

For a ribbon graph $\vG$, one defines a {\em spanning quasi-tree} to be a connected subgraph $Q$ of $G$ which contains every vertex of $G$, and such that an $\epsilon$-thickening of $Q$ in $\Sigma$ has exactly one boundary component; see Example~\ref{ex:ribbon basis}.
When $\Sigma$ is a sphere, this is the same thing as a spanning tree of $G$, but for surfaces of higher genus the two notions diverge; in general, every spanning tree is also a spanning quasi-tree, but not vice-versa. 

The main result of this paper can be stated as follows:

\begin{thmA} \label{thm:main theorem 1}
If $\vG$ is a ribbon graph, there is a naturally associated group $\J(\vG)$ admitting a natural simply transitive action on the set $\cQ(\vG)$ of spanning quasi-trees of $\vG$.
\end{thmA}

\subsection{The Jacobian of a ribbon graph and regular orthogonal matroids}

The group $\J(\vG)$ acting simply and transitively on $\cQ(\vG)$ in Theorem~\ref{thm:main theorem 1} is isomorphic to the {\em critical group} of $\vG$ as defined by Merino, Moffatt, and Noble \cite{MMN2023}. We call it the {\em Jacobian group} of the ribbon graph $\vG$.
Among other things, the authors prove in \cite{MMN2023} that the cardinality of $\J(\vG)$ is the number of spanning quasi-trees in $\vG$, and this is one of the key observations which motivated the present paper.

In \cite{Bouchet1987b,Bouchet1987c}, Bouchet showed that there is a natural way to associate to a ribbon graph $\vG$ an {\em orthogonal matroid}\footnote{Bouchet works with the equivalent concept of {\em even delta-matroids}, rather than orthogonal matroids, and many authors (e.g. Vince and White in \cite{VinceWhite01}) call them {\em Lagrangian orthogonal matroids}; we omit the adjective ``Lagrangian'' since we will not consider the non-Lagrangian case. Recall also that we are assuming the ambient surface $\Sigma$ is orientable; without this assumption Bouchet's construction still yields a delta-matroid, but it will no longer be even.} $M(\vG)$. 
Orthogonal matroids are natural generalizations of matroids, and in particular there is a special class of orthogonal matroids, called {\em regular} orthogonal matroids, which generalize regular matroids. Roughly speaking, a regular orthogonal matroid $M$ is one which is representable by a {\em principally unimodular skew-symmetric matrix}, and we call the choice of such a matrix a {\em representation} of $M$.
The work of Bouchet shows that not only is the orthogonal matroid $M(\vG)$ associated to a ribbon graph $\vG$ regular, it also comes equipped with a natural equivalence class of regular representations.

We generalize the results of \cite{MMN2023} by showing that there is a natural way to associate a {\em Jacobian group} $\J(M, [\cC])$ to a pair $(M,[\cC])$ consisting of a regular orthogonal matroid $M$ and an equivalence class $[\cC]$ of regular representations of $M$. 
The cardinality of $\J(M,[\cC])$ is equal to the number of {\em bases}
of $M$.
Moreover, when $\cC$ arises from a ribbon graph $\vG$ via the construction of Bouchet, $\J(M,[\cC])$ coincides with the critical group of Merino et~al.\footnote{The dependence of $\J(M,[\cC])$ on $[\cC]$ clarifies the observation in \cite{MMN2023} that $\J(\vG)$ does {\bf not} depend only on the even delta-matroid associated to $\vG$. When $M$ arises from a regular {\em matroid}, there is a {\em unique} equivalence class of regular representations, but for more general regular orthogonal matroids this is no longer true.}

As we now explain, our construction of a simply transitive action of $\J(\vG)$ on $\cQ(\vG)$ for a ribbon graph $\vG$ is a special case of a more general family of simply transitive actions of $\J(M,[\cC])$ on the set of bases of $M$. 

\subsection{Generalized BBY torsors for regular orthogonal matroids}

In order to explain our generalization of the theory of ``BBY torsors'' from regular matroids to regular orthogonal matroids, we first briefly recall the setup from Backman et~al.~\cite{BBY2019}. We assume in this discussion that the reader is familiar with the basic terminology of matroid theory.

If $M$ is a regular matroid on the ground set $E$ then $M$ is uniquely orientable (up to a certain natural notion of equivalence of orientations). In particular, each circuit $C$ of $M$ has two distinguished orientations $\pm \vec{C}$. A choice of one of these two orientations, call it $\sigma(C)$, for each circuit $C$ of $M$ is called a {\em circuit signature} for $M$. Such a signature is called {\em acyclic} if $\sum_C a_C \sigma(C) = 0$ in $\R^E$ with all $a_C \geq 0$  implies $a_C=0$ for all circuits $C$. Applying the same definitions to the dual matroid yields the notion of (acyclic) {\em cocircuit signature}.

If $\cO,\cO'$ are orientations of $E$ (meaning we choose one of two possible states for each element of $E$), we call $\cO$ and $\cO'$ {\em equivalent}, and write $\cO \sim \cO'$, if we can get from $\cO$ to $\cO'$ by a sequence of circuit and cocircuit reversals (meaning that whenever we have an orientation consistent with a signed circuit or cocircuit $\vec{C}$, we can replace $\vec{C}$ in the orientation by $-\vec{C}$). The set $\cR(M)$ of equivalence classes of orientations of $E$ is called the {\em circuit-cocircuit reversal system} for $M$.\footnote{The notion of circuit-cocircuit reversal system for graphs (resp. regular matroids) is due to Gioan \cite{Gioan07} (resp. \cite{Gioan08}). The circuit-cocircuit reversal system for ribbon graphs (resp. regular orthogonal matroids) is new to this paper.} 

The authors of \cite{BBY2019} prove:

\begin{thm*}
Let $M$ be a regular matroid.
\begin{enumerate}
    \item There is a canonical simply transitive action of $\J(M)$ on $\cR(M)$. 
    \item Let $\sigma$ (resp. $\sigma^*$) be an acyclic circuit (resp. cocircuit) signature of $M$. For each basis $B$ of $M$, orient elements $e \not\in B$ according to the orientation of $e$ in $\sigma(C(B,e))$, and orient elements $e \in B$ according to the orientation of $e$ in $\sigma^*(C^*(B,e))$, where $C(B,e)$ (resp. $C^*(B,e)$) denotes the fundamental circuit (resp. cocircuit) associated to $B$ and $e$. Denote the resulting orientation class of $E$ by $\bar{\beta}(B) := \bar{\beta}_{\sigma,\sigma^*}(B)$. Then $\bar{\beta} : \cB(M) \to \cR(M)$ is a bijection. Combined with (1), this gives a simply transitive action of $\J(M)$ on $\cB(M)$. 
\end{enumerate}
\end{thm*}

If $G$ is a (connected) plane graph and $M=M(G)$, then $\cB(M) = \cT(G)$ and circuits of $M$ coincide with (simple) cycles of $G$, and there is a canonical choice for $\sigma$ given by orienting all cycles of $G$ counterclockwise. 
Applying the same reasoning to a plane dual of $G$ allows us to similarly define an acyclic cocircuit signature $\sigma^*$. It turns out that different plane duals give rise to the same group actions, from which one deduces that there is in fact a canonical simply transitive action of $\J(G)$ on $\cT(G)$. 
This is the same action discussed in \S\ref{sec:torsor structures on spanning trees} above, which had previously been constructed in different ways in \cite{CCG15} and \cite{BW2018}.

\medskip

We apply similar ideas in this paper to the case of regular orthogonal matroids. In order to explain, we first recall from \cite{JK2023} that an orthogonal matroid $M$ on the ground set $E \cup E^* = \{ 1,\ldots,n,1^*,\ldots,n^* \}$ can be characterized either in terms of its {\em bases}, which are certain ``transversal'' subsets of $E \cup E^*$ not containing any {\em skew pairs} $\{ i,i^* \}$, or in terms of its {\em circuits}, which are also subsets of $E \cup E^*$. Given a matroid $N$ on $E$, there is an associated orthogonal matroid on $E \cup E^*$, called the {\em lift} of $N$, whose set of bases is $\{ B \cup (E-B)^* \; | \; B \in \cB(N) \}$ and whose circuits are either circuits of $N$ or stars of cocircuits of $N$ (i.e., 
$\{ i^* \; | \; i \in D\}$ for some cocircuit $D$ of $N$).
In particular, one can think of the circuit set of an orthogonal matroid as generalizing both circuits and cocircuits of a matroid. 
A matroid is regular if and only if its corresponding orthogonal matroid is regular in the sense of \cite{JK2023}, which coincides with the underlying even delta-matroid being regular in the sense of \cite{BCG1996}.
Given a regular orthogonal matroid $M$ on the ground set $E \cup E^*$ and a regular representation $\cC$ of $M$, we will define an equivalence relation on the set of orientations of $E$ through a suitable notion of {\em circuit reversals}. We will then define a corresponding {\em circuit reversal system} $\cR(M,\cC)$, cf.~Definition~\ref{def:circuit reversal class}.
When $M$ arises from a regular matroid $\tilde{M}$, $\cR(M,\cC)$ coincides with the circuit-cocircuit reversal system of $\tilde{M}$ as defined in \cite{Gioan08}.

We will also define a notion of {\em acyclic circuit signature} for $(M,\cC)$, which also generalizes the corresponding notion from \cite{BBY2019}.
With these concepts in hand, we have:

\begin{thmA} \label{thm:main theorem 2}
Let $M$ be a regular orthogonal matroid on the ground set $E \cup E^*$, and fix a regular representation $\cC$ of $M$.
\begin{enumerate}
    \item There is a canonical simply transitive action of $\J(M,\cC)$ on $\cR(M,\cC)$. 
    \item There is a bijection $\bar{\beta}_\sigma : \cB(M) \to \cR(M,\cC)$ naturally associated to each acyclic circuit signature $\sigma$ of $(M,\cC)$. Combined with (1), this gives a simply transitive action of $\J(M, \cC)$ on $\cB(M)$ for each choice of $\sigma$. 
    \item Both $\J(M,\cC)$ and the action of $\J(M, \cC)$ on $\cB(M)$ are invariant under {\em reorientations} of $\cC$ (cf.~ Definition~\ref{def:reorientation}). %
\end{enumerate}
\end{thmA}

If $\vG = (G,\Sigma)$ is a ribbon graph and $(M,\cC)$ is the corresponding pair consisting of a regular orthogonal matroid $M$ and a regular representation $\cC$ 
then $\cB(M) = \cQ(\vG)$, and choosing a generic point $p$ of the surface $\Sigma$ allows us to define an acyclic circuit signature $\sigma_p$ using the orientation of $\Sigma$. We show that different choices of $p$ lead to isomorphic simply transitive group actions, and then Theorem~\ref{thm:main theorem 1} follows from Theorem~\ref{thm:main theorem 2}.

\subsection{A few words on the proof of Theorem~\ref{thm:main theorem 2}}

Although the statement of Theorem~\ref{thm:main theorem 2} is similar to the analogous statement in \cite{BBY2019} (in the case of regular matroids),
there are some complications which arise in the orthogonal case.

First of all, the proof in \cite{BBY2019} that $\bar{\beta} : \cB(M) \to \cR(M)$ is a bijection involves geometric reasoning related to tilings of a certain zonotope. It is not clear how to define the corresponding geometric objects in the more general setting of regular orthogonal matroids, so we employ the machinery introduced in the second author's paper \cite{Ding2024}, in which a purely combinatorial proof of the bijectivity of $\bar{\beta}$ is given. This involves, among other things, a discussion of {\em fourientations} in the sense of \cite{BH17} and
an analogue for orthogonal matroids of {\em Farkas's Lemma} from oriented matroid theory. 

Second, we need to prove that the definition we give for $\J(M,\cC)$ coincides with the definition in \cite{MMN2023} when $(M,\cC)$ comes from a ribbon graph $\vG$. 

Third, the descriptions in the existing literature of the circuits of the orthogonal matroid associated to a ribbon graph $\vG$ are hard to work with in the present context, since they are not very geometric in nature. We provide an alternate description which should be of independent interest. We must also connect this to the oriented circuits of a regular orthogonal matroid and relate our description to the regular representation defined by Bouchet.

\subsection{Bernardi torsors for ribbon graphs}\label{subsec:Bernardi-intro}

As mentioned above, BBY torsors generalize Bernardi torsors in the case of graphs. We show that a similar phenomenon holds in the context of ribbon graphs for our generalized BBY torsors.

The idea is as follows (see \S\ref{subsec:ribbon bernardi} for details). 
Let $\vG = (G,\Sigma)$ be a ribbon graph.
Let $(M,\cC)$ be the corresponding pair consisting of an orthogonal matroid $M$ and an orthogonal representation $\cC$ of $M$, as defined by Bouchet.
Let $\vG^* = (G^*,\Sigma)$ be a dual ribbon graph embedded in $\Sigma$, whose edges we call {\em coedges} of $G$.

Let $Q$ be a spanning quasi-tree of $\vG$, and thicken $Q$ to a ``ribbon'' $N$. 
Starting at an arbitrary point of the boundary $\partial N$ of $N$, travel along $\partial N$ according to the orientation of $\Sigma$ (as seen from the interior of $N$). 
For any edge $e$, you will either traverse $e$ twice, if $e\notin Q$, or traverse $e^*$ twice, if $e\in Q$. The first time you traverse $e$, orient $e$ inward toward $N$, and the first time you traverse $e^*$, orient $e$ in the direction you're traveling.
Since $Q$ has a unique boundary component, you will eventually get back to your starting point, and when you do, you will have oriented each edge $e$ of $G$, cf.~Figure~\ref{fig: bernardi}.
This defines an orientation $\cO$ on the edges of $G$, which has a corresponding circuit-reversal class $[\cO]$ in $\cR(M,\cC)$. 

We show that this procedure gives a {\em bijection} from $\cQ(\vG)$ to $\cR(M,\cC)$. Moreover, while different choices of starting point along $\partial N$ will in general give different bijections, it follows from our results that if $\beta_1,\beta_2$ are any two such bijections, there exists $g \in \J(M,\cC)$ such that $\beta_2(Q) = g \cdot \beta_1(Q)$ for all $Q \in \cQ$. Our results also show that the resulting ``Bernardi torsor'' coincides with the one defined in \cite{BW2018} when $\Sigma$ is a sphere (i.e., when $\vG$ is a plane ribbon graph). In addition (and this is how our proof actually proceeds), we show that this torsor is in fact a BBY torsor with respect to a suitable acyclic circuit signature $\sigma$ of $(M,\cC)$. 
%
%

\subsection{Structure of the paper}

In \S\ref{sec:preliminaries}, we provide background information on orthogonal matroids, oriented and regular orthogonal matroids, and ribbon graphs, and we prove the above-mentioned generalization of Farkas's Lemma.

In \S\ref{sec: cycle reversal}, we define the Jacobian group $\J(\cC)$ of a pair $(M,\cC)$, where $M$ is a regular orthogonal matroid and $\cC$ is a regular representation. We also define the circuit reversal system $\cR(\cC)$ in this context, and prove that there is a canonical group action of $\J(\cC)$ on the set $\cR(\cC)$. Lastly, we prove that the group action is invariant under reorientations of $\cC$. 

In \S\ref{sec: BBY}, we formulate and prove Theorem~\ref{thm:main theorem 2}, generalizing the theory of BBY bijections from regular matroids to regular orthogonal matroids.

In \S\ref{sec:ribbon torsor}, we apply Theorem~\ref{thm:main theorem 2} to the case of ribbon graphs and deduce Theorem~\ref{thm:main theorem 1}. We also relate our generalized BBY bijections to the family of Bernardi torsors associated to a ribbon graph $\vG$.

In \S\ref{sec: complement}, we examine how Theorem~\ref{thm:main theorem 1} interacts with dual ribbon graphs and explain how our definitions and results generalize the corresponding ones for regular matroids.

In Appendix~\ref{subsec:bouchet}, we discuss the compatibility between certain definitions in the existing literature on ribbon graphs and the definitions employed in the main body of this paper.

\section{Preliminaries}
\label{sec:preliminaries}

\subsection{Orthogonal matroids}

Let $E$ be a finite set of cardinality $n$ and let $E^* := \{e^* : e\in E\}$ be a disjoint copy of $E$. 
For $e\in E$, we write $(e^*)^* = e$.

A \emph{skew pair} is a $2$-element subset $\{e,e^*\}$.
A \emph{transversal} is an $n$-element subset of $E\cup E^*$ which does not contain any skew pair.
A \emph{subtransversal}\footnote{Subtransversals are also known in the literature as {\em admissible sets}.} is a subset of a transversal.

\begin{definition}
    An \emph{orthogonal matroid} $M$ is a pair $(E\cup E^*, \cB)$ such that $\cB$ is a nonempty set of transversals satisfying the \emph{symmetric exchange property}:
    \begin{itemize}[label=\rm(B)]
        \item For all $B,B'\in \cB$ and $e\in B\setminus B'$, there is an element $f\ne e$ in $B\setminus B'$ such that $B\triangle\{e,e^*,f,f^*\} \in \cB$.
    \end{itemize}
    
    An element in $\cB$ is called a \emph{basis}. We sometimes write $\cB(M)$ instead of $\cB$ to emphasize the dependence on $M$.
    
    A \emph{circuit} is a minimal subtransversal not contained in any basis.
\end{definition}

Recall that a \emph{matroid} is a pair $(E,\cB)$ such that $\cB$ is a nonempty collection of subsets of $E$ satisfying:
\begin{itemize}[label=\rm(B')]
    \item For all $B,B'\in \cB$ and $e\in B\setminus B'$, there is an element $f\in B'\setminus B$ such that $(B\setminus \{e\}) \cup \{f\} \in \cB$.
\end{itemize}

Orthogonal matroids generalize matroids: given a matroid $N = (E,\cB)$, the pair $\mathrm{lift}(N):=(E\cup E^*, \cB')$ is an orthogonal matroid, where $\cB' = \{B\cup (E\setminus B)^* : B\in \cB\}$.
The set of circuits of $\mathrm{lift}(N)$ is the union of $\{C\subseteq E: \text{$C$ is a circuit of $N$}\}$ and $\{D^*\subseteq E^*: \text{$D$ is a cocircuit of $N$}\}$.

Orthogonal matroids have a ``cryptomorphic'' characterization in terms of circuits which generalizes both the circuit elimination axiom for matroids and the fact that circuits and cocircuits of a matroid are orthogonal.
The following theorem is due to Bouchet~\cite{Bouchet1997,Bouchet2001} and is a special case of a more general characterization of multimatroids in terms of their circuits. In particular, orthogonal matroids coincide with \emph{tight 2-matroids} in~\cite{Bouchet1997,Bouchet2001}. For convenience, we refer the reader to~{\cite[Theorem~12]{BMP2003}}.

\begin{theorem}[\cite{Bouchet1997,Bouchet2001}]
    Let $\cC$ be a set of subtransversals of $E\cup E^*$.
    Then $\cC$ is the set of circuits of an orthogonal matroid if and only if it satisfies the following:
    \begin{enumerate}[label=\rm(C\arabic*)]
        \item $\emptyset\notin \cC$,
        \item $C_1,C_2\in\cC$ and $C_1\subseteq C_2$ imply $C_1 = C_2$,
        \item\label{item:c3} for any $C_1, C_2\in \cC$ and $e\in C_1 \cap C_2$, if $C_1\cup C_2$ is a subtransversal, then there exists $C_3\in \cC$ such that $C_3 \subseteq (C_1 \cup C_2) \setminus \{e\}$,
        \item\label{item:c4} $|C_1\cap C_2^*| \ne 1$ for all $C_1,C_2\in\cC$, and 
        \item\label{item:c5} for every transversal $T$ and every element $e\in T^*$, there exists $C\in\cC$ such that $C\subseteq T\cup\{e\}$.
    \end{enumerate}
\end{theorem}

The \emph{fundamental circuit} of an orthogonal matroid $M$ with respect to a basis $B$ and an element $e\in B^*$ is a circuit contained in $B\triangle \{e,e^*\}$, which exists and is unique by~\ref{item:c3} and~\ref{item:c5}.
We denote it by $\FC(B,e)$. 

We will need the following lemma later, which strengthens \ref{item:c5}.
\begin{lemma}[{\cite[Lemma~3.4]{Kim2025b}}]\label{lem: maximality-stronger}
    If $\cC$ is the set of circuits of an orthogonal matroid, then for every transversal $T$ and every element $e\in T^*$, there exists $C\in\cC$ such that $C\subseteq T\cup\{e\}$ and $C\cap \{e,e^*\} \ne \emptyset$.
\end{lemma}

\subsection{Oriented and regular orthogonal matroids}\label{subsec: regular om}

Wenzel~\cite{Wenzel1993b,Wenzel1996b} introduced oriented orthogonal matroids as a generalization of oriented matroids.
Wenzel defined this concept by generalizing the chirotope axiom for oriented matroids, and later Jin and the third author~\cite{JK2023} presented an equivalent definition in terms of signed circuits.

For a vector $\vC\in \mathbb{Z}^{E\cup E^*}$, let $\vC^*$ be the vector in $\mathbb{Z}^{E\cup E^*}$ such that $\vC^*(e) = \vC(e^*)$ for all $e\in E\cup E^*$.

\begin{definition}\label{def: oriented OM}
    An \emph{oriented orthogonal matroid} $\cC$ on $E\cup E^*$ is a set of $(0,\pm 1)$-vectors in $\mathbb{Z}^{E\cup E^*}$ satisfying the following:
    \begin{enumerate}[label=\rm(\roman*)]
        \item $\cC=-\cC$,
        \item $\{\mathrm{supp}(\vC) : \vC\in \cC\}$ is the set of circuits of an orthogonal matroid $M_{\cC}$ on $E\cup E^*$, and 
        \item for all $\vC,\vD \in \cC$, we have either 
        \begin{itemize}
            \item 
            $\vC(e)\vD^*(e) = 0$ for all $e\in E\cup E^*$, or 
            \item 
            $\vC(e)\vD^*(e) = 1$ and $\vC(e')\vD^*(e') = -1$ for some $e,e' \in E\cup E^*$.
        \end{itemize}
    \end{enumerate}
    We call the elements of $\cC$ \emph{signed circuits}, and we call $M_{\cC}$ the \emph{underlying orthogonal matroid} of $\cC$.
\end{definition}

By~{\cite[Proposition~3.4]{JK2023}}, an oriented matroid is the same thing as an oriented orthogonal matroid whose underlying orthogonal matroid is equal to $\mathrm{lift}(N)$ for some matroid $N$.

\begin{definition}\label{def: regular representation}
    A \emph{regular orthogonal representation} of an orthogonal matroid ${M}$ is an oriented orthogonal matroid $\mathcal{C}$ such that 
    $M = M_\cC$ 
    and $\cC$ satisfies the following stronger version of (iii) in Definition~\ref{def: oriented OM}:
    \begin{enumerate}[label=\rm(\roman*')]
        \setcounter{enumi}{2}
        \item\label{item:ror} 
        $\langle \vC, \vD\rangle := \sum_{e\in E\cup E^*} \vC(e)\vD^*(e) = 0$ in $\Z$ for all $\vC,\vD \in \cC$.
    \end{enumerate}
    An orthogonal matroid is \emph{regular} if it admits a regular orthogonal representation.
\end{definition}

\begin{lemma}\label{lem: sum of fund circuit}
Let $\cC$ be a regular orthogonal representation.
For any $\vC\in\cC$ and any basis B, we have\[\vC = \sum_{e\in \mathrm{supp}(\vC)\setminus B} \vC_e,\]where $\vC_e$ is the signed circuit such that $\mathrm{supp}(\vC_e)=\FC(B,e)$ and $\vC_e(e) = \vC(e)$.
\end{lemma}

\begin{proof}
    By the characterization of strong $F$-circuit sets of orthogonal matroids in~{\cite[Definition~3.10]{JK2023}}, applied with $F$ being the regular partial field,
    $\vC = \sum_{e\in \mathrm{supp}(C)\setminus B} \vC_e$, where each $\vC_e$ is a signed circuit such that $\mathrm{supp}(\vC_e) = \FC(B,e)$.
    Then for each $e\in \mathrm{supp}(\vC)\setminus B$, we have $\vC_e(e) = \vC(e)$.
\end{proof}

Each regular orthogonal representation induces a principally unimodular skew-symmetric matrix as follows.
A square matrix is \emph{principally unimodular} (in short, \emph{PU}) if it is defined over the rational numbers~$\mathbb{Q}$ and all its principal minors are $0$ or $\pm 1$.
Note that any principal minor of a PU skew-symmetric matrix is either $0$ or $1$, 
and a skew-symmetric matrix is PU if and only if every principal submatrix has Pfaffian $0$ or $\pm 1$.
These facts follow from Cayley's theorem that $\mathrm{pf}(A)^2 = \det(A)$ for any skew-symmetric matrix~$A$.

Given a regular orthogonal representation~$\cC$ of an orthogonal matroid~$M$, we fix a basis $B$ of $M$ and
let $\Lambda(\cC,B)$ be the $B \times (E\cup E^*)$ matrix defined by taking its $e$-th row as $\vC_e \in \cC$ such that $\mathrm{supp}(\vC_e) = \FC(B,e^*)$ and $\vC_e(e^*) = 1$.
By rearranging the columns, we have
\begin{align} \label{eq:A-matrix}
    \Lambda(\cC,B) = 
    \begin{blockarray}{ccc}
    \scalebox{0.7}{\text{$B^*$}} & \scalebox{0.7}{\text{$B$}} & \\
    \begin{block}{(cc)c}
      I & A & \scalebox{0.7}{\text{$B$}} \\
    \end{block}
    \end{blockarray}
\end{align}
where $I$ is the identity matrix. %
Then $A$ is skew-symmetric by~\ref{item:ror}.
On the other hand, the regular orthogonal representation $\cC$ induces a Wick function over the regular partial field $\mathbb{U}_0$ by~{\cite[Theorem~1.2]{JK2023}}; that is, a function $\varphi$ from the set of transversals to $\{0,\pm 1\}$ that satisfies the Wick equations; see~{\cite[Definition~3.1]{JK2023}} for details. This function is equivalent, up to normalization so that $\varphi(B^*) = 1$, to the Pfaffian map introduced by Wenzel~\cite{Wenzel1993b}.  Therefore, by Wenzel's result~\cite{Wenzel1993b}, the Wick function $\varphi$ encodes precisely the Pfaffians of $A$, which implies that $A$ is PU.

The $\mathbb{Q}$-span of $\cC$ has dimension at most $n$, since it is an isotropic subspace of the $2n$-dimensional space $\mathbb{Q}^{E\cup E^*}$ with respect to the non-degenerate bilinear form~$\langle \cdot, \cdot \rangle$ in~\ref{item:ror}; see~{\cite[Theorem~6.4]{Lang1987}}. 
The row span of $\Lambda(\cC,B)$ over $\mathbb{Q}$ is equal to the $\mathbb{Q}$-span of $\cC$, as the $n$ rows $\vC_e$'s of $\Lambda(\cC,B)$ are linearly independent.
In fact, the same statement holds for $\mathbb{Z}$-spans:

\begin{lemma}\label{lem: equal row span}
The $\mathbb{Z}$-span of the rows of $\Lambda(\cC,B)$ equals the $\mathbb{Z}$-span of the signed circuits in $\cC$. 
\end{lemma}
\begin{proof}
    By Lemma~\ref{lem: sum of fund circuit}, the $\Z$-span of the rows of $\Lambda(\cC,B)$ contains all signed circuits in $\cC$, and therefore it is equal to the $\Z$-span of $\cC$.
\end{proof}

The following is a consequence of (4.2) and (4.4) in~\cite{Bouchet1988}.

\begin{theorem}[\cite{Bouchet1988}] 
    Let $\cC$ be a regular orthogonal representation, let $B$ be a basis of the underlying orthogonal matroid $M_\cC$, and let $\Lambda := \Lambda(\cC,B) = \begin{pmatrix}
        I & A 
    \end{pmatrix}$.
    Then for every transversal $T$, the following are equivalent:
    \begin{enumerate}[label=\rm(\roman*)]
        \item $T$ is a basis of $M_{\cC}$. 
        \item $\Lambda[B,T^*]$ is nonsingular.
        \item $A[B \cap T^*, B \cap T^*]$ is nonsingular. 
    \end{enumerate}
\end{theorem}

As the determinant of nonsingular $A[B \cap T^*, B \cap T^*]$ equals $1$, we deduce the following corollary.
This result generalizes the Matrix–Quasi-tree Theorem for ribbon graphs {\cite[Theorem~6.3]{MMN2023}}; the connection between orthogonal matroids and ribbon graphs is explained in \S\ref{subsection: ribbon}.

\begin{corollary}
\label{cor:counting bases}
    The number of bases of $M_{\cC}$ is equal to $\det(I+A)$. 
    \qed
\end{corollary}
We remark that the matrix $I+A$ has entries in $\{0,\pm 1\}$ because the supports of the rows of  $\begin{pmatrix}
        I & A
    \end{pmatrix}$ are subtransversals.

Given an oriented orthogonal matroid, there is a way to obtain other oriented orthogonal matroids with the same underlying orthogonal matroid.

\begin{definition} \label{def:reorientation}
    Let $\cC \subseteq \mathbb{Z}^{E\cup E^*}$ be an oriented orthogonal matroid.
    The \emph{reorientation} of $\cC$ on a set $X\subseteq E$ is 
    $\cC^{X} := \{\vC^{X} \in \mathbb{Z}^{E\cup E^*} : \vC\in\cC\}$, where
    \begin{align*}
        \vC^{X}(i) = \begin{cases}
            -\vC(i) & \text{if } i\in X\cup X^*, \\
            \vC(i) & \text{otherwise}.
        \end{cases}
    \end{align*}
\end{definition}

A reorientation $\cC^X$ is an oriented orthogonal matroid.
Furthermore, if $\cC$ is a regular orthogonal representation, then $\cC^{X}$ is also a regular orthogonal representation.

We briefly review the definition of regular matroids and their connection to regular orthogonal matroids.
Regular matroids are usually defined in terms of totally unimodular matrices, but here we recall an equivalent definition in terms of chain group representations in the sense of Tutte~\cite{Tutte1958} or weak $\mathbb{U}_0$-circuit sets in the sense of \cite{BB2019}.

\begin{definition}\label{def:regular matroid}
    A \emph{regular representation} of a matroid $N$ on $E$ is a set $\cC$ of $(0,\pm 1)$-vectors in $\mathbb{Z}^E$ satisfying the following conditions:
    \begin{enumerate}[label=\rm(\roman*)]
        \item $\cC = -\cC$,
        \item $\{\mathrm{supp}(\vC) : \vC \in \cC\}$ is the set of circuits of $N$,  
        \item if two vectors $\vC_1,\vC_2 \in \cC$ have the same support, then $\vC_1 = \pm \vC_2$, and 
        \item for any $\vC_1,\vC_2\in \cC$ and $e\in \mathrm{supp}(\vC_1) \cap \mathrm{supp}(\vC_2)$, if $\vC_1(e) = -\vC_2(e)$ and the rank of $\mathrm{supp}(\vC_1) \cup \mathrm{supp}(\vC_2)$ is $|\mathrm{supp}(\vC_1) \cup \mathrm{supp}(\vC_2)|-2$,
        then there exists $\vC_3\in \cC$ such that $\vC_1 + \vC_2 = \vC_3$.
    \end{enumerate}
    A matroid is \emph{regular} if it admits a regular representation.
\end{definition}

\begin{proposition}
    \label{prop: restriction}
    Let $\cC$ be a regular orthogonal representation on $E\cup E^*$ and let $S$ be a subtransversal in $E\cup E^*$.
    Then the set of circuits $C$ of $M_{\cC}$ with $C\subseteq S$ is the circuit set of a matroid $N$ on $S$, 
    and $\cC |_{S} := \{\vC\in \cC : \mathrm{supp}(\vC) \subseteq S\}$ is a regular representation of $N$. 
\end{proposition}

\begin{proof}
    By~\ref{item:c3}, the set of circuits $C$ of $M_{\cC}$ with $C\subseteq S$ satisfies the circuit elimination axiom for matroids.
    Therefore, it forms the set of circuits of a matroid, say $N$, on $S$.

    Now, it suffices to check that $\cC|_{S}$ satisfies (iii) and (iv) of Definition~\ref{def:regular matroid}.
    The condition~(iii) follows from~{\cite[Lemma~3.7]{JK2023}}. %

    We claim that $\cC|_{S}$ satisfies (iv).
    Let $\vC_1, \vC_2\in \cC|_S$ and let $e\in S$ be such that $\vC_1(e) = -\vC_2(e)$.
    By definition of $\cC|_{S}$ and the circuit elimination axiom, there exists $\vC_3\in \cC|_{S}$ such that $C_3 \subseteq (C_1\cup C_2) \setminus\{e\}$, where $C_i := \mathrm{supp}(\vC_i)$.
    Suppose that the rank of $C_1 \cup C_2$ is equal to $|C_1\cup C_2|-2$.
    Then there is a basis $B$ of $N$ such that $|C_1 \setminus B| = |C_2 \setminus B| = 1$.
    Writing $C_i \setminus B = \{x_i\}$, we have $C_3 \setminus B = \{x_1,x_2\}$.
    The underlying orthogonal matroid $M_{\cC}$ has a basis $B'$ containing $B$, since $B$ does not contain any circuits of $M_{\cC}$.
    Since $B' \cap \{x_1,x_2\} = \emptyset$,
    we must have $\vC_3 = \alpha_1 \vC_1 + \alpha_2 \vC_2$ for some $\alpha_i\in \{\pm 1\}$ by Lemma~\ref{lem: sum of fund circuit}.
    Then $0 = \vC_3(e) = \alpha_1 \vC_1(e) + \alpha_2\vC_2(e)$.
    Because $\vC_1(e) = -\vC_2(e)$, we have $\alpha_1=\alpha_2$.
    Then $\vC_1+\vC_2 = \vC_3$ or $-\vC_3$, implying that $\cC|_{S}$ satisfies (iv).
\end{proof}

\subsection{Fourientations and a generalized Farkas' Lemma}\label{sec: farkas}

In this section we formulate an extension of Farkas' Lemma for oriented orthogonal matroids.
For a vector $\vC$, we write $\vC \ge 0$ if all entries of $\vC$ are nonnegative.

\begin{lemma}
    \label{lem: Farkas}
    Let $\cC$ be an oriented orthogonal matroid on $E\cup E^*$ and let $e\in E$.
    Then there is either a signed circuit $\vC$ such that $\vC \ge 0$ and $\vC(e) \neq 0$ or a signed circuit $\vC$ such that $\vC \ge 0$ and $\vC(e^*) \neq 0$, but not both.
\end{lemma}

Farkas' Lemma for oriented matroids~{\cite[Corollary~3.4.6]{BVSWZ1999}} is equivalent to the special case of Lemma~\ref{lem: Farkas} where the underlying orthogonal matroid of $\cC$ in Lemma~\ref{lem: Farkas} is equal to $\mathrm{lift}(N)$ for some (oriented) matroid~$N$.
We will in fact prove a more general version of Lemma~\ref{lem: Farkas}, which implies the ``$4$-painting lemma'' for oriented matroids {\cite[Theorem~3.4.4]{BVSWZ1999}}.
To present the result, we begin by introducing fourientations, which play a central role in Section~\ref{sec: BBY}.

\begin{definition}[Compare with Backman--Hopkins \cite{BH17}]
A \emph{fourientation} $\vF$ is a function from $E\cup E^*$ to $\{\{ 1,-1\},\{1\},\{-1\},\emptyset\}$. We frequently write $\{\pm 1\}$ as shorthand for 
$\{ 1,-1\}$. If $\vF(e)=\{\pm 1 \}$ we call $e$ {\em bi-oriented}, and if $\vF(e)=\{ \emptyset \}$ we call $e$ {\em unoriented}. Moreover:
\begin{enumerate}[label=\rm(\roman*)]
    \item Given a fourientation $\vF$, we define the fourientation $-\vF$ by
    \[(-\vF)(e)=-(\vF(e)),\]
    where $-\emptyset:=\emptyset$, $-\{\pm 1\}:=\{\pm 1\}$, $-\{1\}:=\{-1\}$, and $-\{-1\}:=\{1\}$.
    \item Given two fourientations $\vF_1$ and $\vF_2$, the fourientation $\vF_1\cap\vF_2$ is defined by\[(\vF_1\cap\vF_2)(e):=(\vF_1(e))\cap(\vF_2(e)).\]
    \item We write $\vF_1\subseteq\vF_2$ if and only if $\vF_1(e)\subseteq\vF_2(e)$ for every $e\in E\cup E^*$.
\end{enumerate}
\end{definition}

\begin{remark}
A signed circuit $\vC\in \{0,\pm 1\}^{E\cup E^*}$ can be viewed as a fourientation by replacing $0$, $1$, and $-1$ with $\emptyset$, $\{1\}$, and $\{-1\}$ respectively. Hence it makes sense to write $\vC\subseteq\vF$ or say that $\vF$ contains $\vC$, where $\vF$ is a fourientation. 
\end{remark}

\begin{definition}\label{def: positive fourientation}
    A fourientation $\vF$ is \emph{positive} if $\vF(e) = \{ 1,-1 \} \setminus (-\vF(e^*))$ for each $e\in E$.
    It is \emph{negative} if $\vF(e) = \{ 1,-1 \} \setminus \vF(e^*)$ for each $e\in E$.
\end{definition}

\begin{example} \label{ex:foutientations}
\leavevmode
\begin{enumerate}
    \item If $\vF_{+}$ is the fourientation with $\vF(e)=\{+1\}$ for all $e\in E\cup E^*$,
then $\vF_{+}$ is positive. 
    \item Fix $e\in E$ and let $\vF_e$ be the fourientation such that $\vF_e(e) = -\vF_e(e^*) = \{+1\}$, and such that $\vF_e(f)$ is bi-oriented and $\vF_e(f^*)$ is unoriented for each $f\in E\setminus \{e\}$. Then $\vF_e$ is negative.
    By~\ref{item:c5}, if $\cC$ is an oriented orthogonal matroid on $E \cup E^*$, there is always a signed circuit $\vC$ of $\cC$ such that $\vC \subseteq \vF_e$.
\end{enumerate}
\end{example}

The following observation generalizes Example~\ref{ex:foutientations}(2):

\begin{lemma}\label{lem: Farkas2}
    Let $\cC$ be an oriented orthogonal matroid on $E\cup E^*$ and let $e\in E$.
    Let $\vF$ be a positive or negative fourientation with ${|\vF(e)| = 1}$. 
    Then there is either a signed circuit $\vC$ such that $\vC \subseteq \vF$ and $\vC(e) \ne 0$ or a signed circuit $\vC$ such that $\vC \subseteq \vF$ and $\vC(e^*) \ne 0$, but not both.
\end{lemma}

\begin{proof}
    We first claim the existence of a signed circuit $\vC$ such that $\vC\subseteq \vF$ and $\vC(e)+\vC(e^*) \neq 0$.
    We proceed by induction on the number $t$ of elements $f\in E$ with $|\vF(f)|=1$.
    If $t=1$, then there is a signed circuit $\vC$ such that $\vC \subseteq \vF$ and $\vC(e) + \vC(e^*) \ne 0$ by Lemma~\ref{lem: maximality-stronger} and Definition~\ref{def: oriented OM}(ii). 
    
    We now assume that $t \ge 2$, and we choose $f\in E\setminus\{e\}$ with $|\vF(f)| = |\vF(f^*)| = 1$.
    Let $\vF_1$ and $\vF_2$ be fourientations such that $\vF_1(f) = \vF_2(f^*) = \{\pm 1\}$, $\vF_1(f^*) = \vF_2(f) = \emptyset$, and $\vF_1(g) = \vF_2(g) = \vF(g)$ for all $g\in (E\cup E^*) \setminus \{f,f^*\}$.
    By the induction hypothesis, there are signed circuits $\vC_i$ with $i=1,2$ such that $\vC_i \subseteq \vF_i$ and $\vC_i(e) + \vC_i(e^*) \ne 0$.
    
    Assume that neither $\vC_1$ nor $\vC_2$ is contained in $\vF$.
    Then $\vC_1(f) = -\vF(f)$ and $\vC_2(f^*) = -\vF(f^*)$.
    If $\vF$ is positive, then $\vC_1(f)$ and $\vC_2(f^*)$ have the same sign, and
    for any $g\in E\cup E^*$, $\vC_1(g)$ and $\vC_2(g^*)$ either have the same sign or one of them is $0$.
    However, this contradicts Definition~\ref{def: oriented OM}(iii).
    If $\vF$ is negative, we obtain a contradiction in the same way.
    Therefore either $\vC_1$ or $\vC_2$ is contained in $\vF$.

    Suppose that there are signed circuits $\vC_1$ and $\vC_2$ contained in $\vF$ such that 
    $\vC_1(e) \vC_2(e^*) \ne 0$.
    Then the signs of nonzero $\vC_1(f)\vC_2(f^*)$ with $f\in E\cup E^*$ are the same ($+1$ if $\vF$ is positive and $-1$ if $\vF$ is negative), 
    contradicting Definition~\ref{def: oriented OM}(iii). %
\end{proof}

Applying Lemma~\ref{lem: Farkas2} to the positive fourientation $\vF_{+}$, we obtain Lemma~\ref{lem: Farkas} as a special case.

\subsection{Ribbon graphs}\label{subsection: ribbon}

A \emph{ribbon graph} $\vG = (G,\Sigma)$ is a graph $G = (V,E)$ together with a cellular embedding of $G$ in a connected closed orientable\footnote{Ribbon graphs are generally defined on arbitrary closed surfaces, regardless of orientability. However, we assume in this paper that the ambient surface for every ribbon graph is orientable.} surface $\Sigma$, where a {\em cellular embedding} is an injective map the complement of whose image is homeomorphic to a disjoint union of open disks.

The \emph{vertices} and \emph{edges} of the ribbon graph $\vG$ are those of the graph $G$.

The ribbon graph $\vG$ associated with $G\hookrightarrow \Sigma$ naturally induces a \emph{dual} graph $G^* = (V^*,E^*)$.
The vertices of $G^*$ are points located in different components of the complement of $G$, and each edge $e^*$ of $G^*$ traverses the corresponding edge $e$ of the original graph $G$.
See Figure~\ref{fig:torus} for an example.
The \emph{dual} ribbon graph is $\vG^* := (G^*,\Sigma)$. 
We call the vertices and edges of $G^*$ the \emph{covertices} and \emph{coedges}, respectively, of the ribbon graph $\vG$.

A transversal $B$ of $E\cup E^*$ is a \emph{basis} of $\vG$ if $\Sigma \setminus B$ is connected, and we denote by $\cB(\vG)$ the set of bases.

\begin{figure}
    \centering
    \begin{tikzpicture}
        \node at (0,0) {\includegraphics[width=0.40\linewidth]{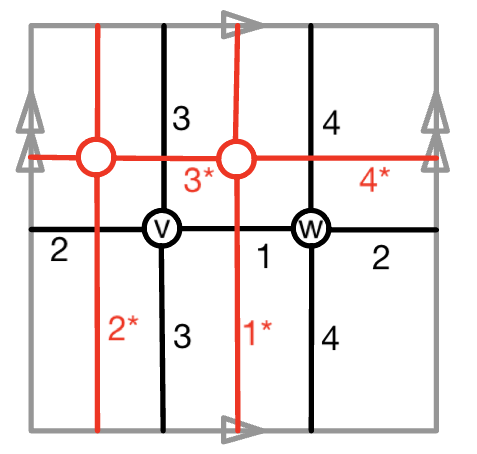}};
        \node at (7.8,0) {\includegraphics[width=0.50\linewidth]{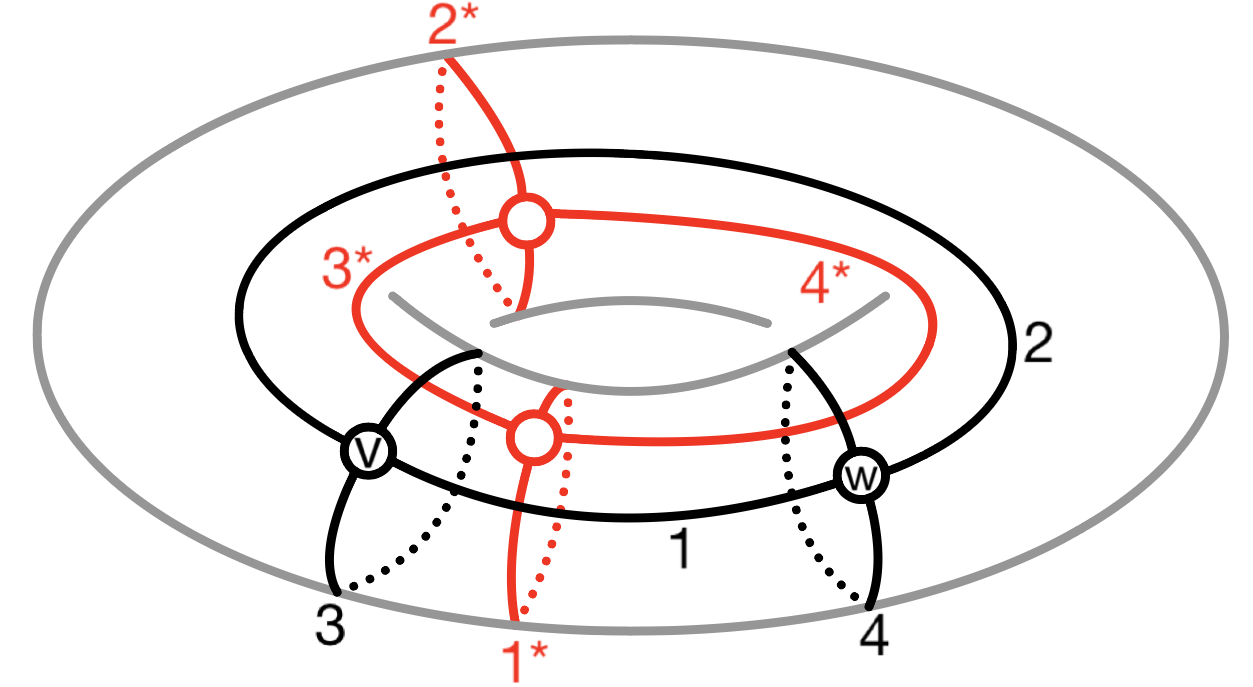}};
    \end{tikzpicture}
    \caption{A graph $G = (\{v,w\},\{1,2,3,4\})$ embedded on a torus.
    The graph $G$ is colored black and its dual $G^*$ is colored red.}
    \label{fig:torus}
\end{figure}

\begin{theorem}[Bouchet~\cite{Bouchet1987b}]
    $M(\vG) := (E\cup E^*,\cB(\vG))$ is an orthogonal matroid.
\end{theorem}

A subgraph $H$ of $G$ is \emph{spanning} if its vertex set is $V$, and is a \emph{quasi-tree} if an $\epsilon$-neighborhood of $H$ in $\Sigma$ has exactly one boundary component for a sufficiently small $\epsilon>0$.
We often identify an edge set $Q\subseteq E$ with the spanning subgraph $(V,Q)$, and let $\cQ(\vG) \subseteq 2^E$ be the set of spanning quasi-trees of $\vG$.

\begin{example}\label{ex:ribbon basis}
    The bases of the ribbon graph depicted in Figure~\ref{fig:torus} are $12^*3^*4^*$, $1^*23^*4^*$, $1234^*$, and $123^*4$.
    Its spanning quasi-trees are $1$, $2$, $123$, and $124$.
\end{example}

By the following lemma, we can identify the spanning quasi-trees of $\vG$ with the bases of $\vG$. 

\begin{lemma}\label{lem:ribbon basis}
    A subset $Q$ of $E$ is a spanning quasi-tree of $\vG$ if and only if $Q\cup (E\setminus Q)^*$ is a basis of $\vG$.
\end{lemma}
\begin{proof}
    The boundary of a small $\epsilon$-neighborhood of $(V,Q)$ in $\Sigma$ is homeomorphic to $H|\tau(Q,(E\setminus Q)^*)$ in the sense of~{\cite[Section~3]{Bouchet1987b}}, where $H$ is the ``medial graph'' of $(G,\Sigma)$.
    So the desired equivalence follows from~{\cite[Corollary~3.4]{Bouchet1987b}}.
\end{proof}

We call $M(\vG)$ a \emph{ribbon-graphic} orthogonal matroid.
By definition, we can characterize the circuits of $M(\vG)$ as follows:

\begin{lemma}\label{lem:ribbon cycle}
    A subset $C$ of $E\cup E^*$ is a circuit of $M(\vG)$ if and only if $C$ is a minimal subtransversal whose complement in $\Sigma$ is not connected. 
    \qed
\end{lemma}

\begin{example}
    The circuits of the ribbon-graphic orthogonal matroid corresponding to Figure~\ref{fig:torus} are $1^*2^*$, $34$, $1^*3$, $1^*4$, $2^*3$, $2^*4$, and $123^*4^*$.
\end{example}

Every ribbon-graphic orthogonal matroid is regular, and the regularity is witnessed by the following representation.

\medskip

First, by convention, fix the orientation on $\Sigma$ which is counter-clockwise with respect to (outward) normal vectors.
Every submanifold $\Sigma'$ of $\Sigma$ inherits an orientation from $\Sigma$, and its boundary $\partial \Sigma'$ is also endowed with an \emph{induced orientation}. A precise definition of the induced orientation on a submanifold can be found in~\cite[Chapter~15]{Lee}. Informally, the induced orientation on $\partial \Sigma'$ may be described as follows: if a person traverses $\partial \Sigma'$ with their head pointing in the direction of the normal vector to $\Sigma'$ and with $\Sigma'$ on their left, then the induced orientation on $\partial \Sigma'$ is the direction of travel.

Now choose a reference orientation~$\vec{G}$ of the set $E$ of edges of $G$, and assign the {\em dual reference orientation} to $G^*$ by the ``right-hand rule'', i.e., so that each coedge $e^*$ crosses $e$ from right to left (or, equivalently, $e$ crosses $e^*$ from left to right); see Figure~\ref{fig:orientation}.

\begin{definition}\label{def: ribbon to om}
    Let $\cC(\vG, \vec{G}) \subseteq \mathbb{Z}^{E\cup E^*}$ be the set of vectors such that 
    \begin{enumerate}[label=\rm(\roman*)]
        \item $\{\mathrm{supp}(\vC): \vC\in \cC(\vG,\vec{G})\}$ is the set of circuits of $M(\vG)$, and 
        \item for each $\vC \in \cC(\vG,\vec{G})$, there is a connected component $\Sigma'$ of $\Sigma\setminus \mathrm{supp}(\vC)$ with boundary $\mathrm{supp}(\vC)$ such that
        \begin{align*}
            \vC(e) = 
            \begin{cases}
                \vspace{-0.2cm}
                +1 & \text{if the reference orientation of $e$} \\ 
                   & \text{agrees with the induced orientation on $\partial \Sigma'$}, \\
                -1 & \text{otherwise}
            \end{cases}
        \end{align*}
        for all $e\in \mathrm{supp}(\vC)$. In this case, we say that the signed circuit $\vC$ is {\em oriented by $\Sigma'$}.
    \end{enumerate}
    We call each vector $\vC\in \cC(\vG,\vec{G})$ a \emph{signed ribbon cycle} of $\vG$.
    We often abbreviate $\cC(\vG,\vec{G})$ to $\cC(\vG)$ if the reference orientation $\vec{G}$ is clear from the context.
\end{definition}

\begin{example}\label{ex: circuits}
    The signed ribbon cycle oriented by $\partial \Sigma'$ in Figure~\ref{fig:orientation} is the $(0,\pm1)$-vector
    \[
        \vC = \begin{pmatrix}
            0 & 0 & 1 & -1 \\
            0 & 0 & 0 & 0 \\
        \end{pmatrix},
    \]
    where the upper four entries are $\vC(1),\ldots, \vC(4)$ and the lower four entries are $\vC(1^*),\ldots,\vC(4^*)$.
    The signed ribbon cycles of $\vG$, up to multiplication by $-1$, are:
    \begin{align*}
        &
        \begin{pmatrix}
            0 & 0 & 0 & 0 \\
            1 & -1 & 0 & 0 \\
        \end{pmatrix},
        \quad
        \begin{pmatrix}
            0 & 0 & 1 & -1 \\
            0 & 0 & 0 & 0 \\
        \end{pmatrix} ,
        \quad
        \begin{pmatrix}
            0 & 0 & 1 & 0 \\
            -1 & 0 & 0 & 0 \\
        \end{pmatrix},
        \quad
        \begin{pmatrix}
            0 & 0 & 0 & 1 \\
            -1 & 0 & 0 & 0 \\
        \end{pmatrix}, \\
        &
        \begin{pmatrix}
            0 & 0 & 1 & 0 \\
            0 & -1 & 0 & 0 \\
        \end{pmatrix},
        \quad
        \begin{pmatrix}
            0 & 0 & 0 & 1 \\
            0 & -1 & 0 & 0 \\
        \end{pmatrix} ,
        \quad
        \begin{pmatrix}
            1 & 1 & 0 & 0 \\
            0 & 0 & 1 & 1 \\
        \end{pmatrix}.
    \end{align*}
\end{example}

\begin{figure}
    \centering
    \includegraphics[width=0.5\linewidth]{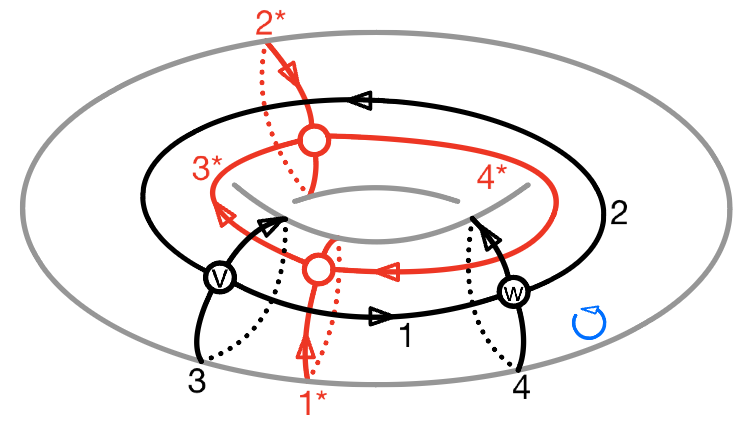}
    \includegraphics[width=0.4\linewidth]{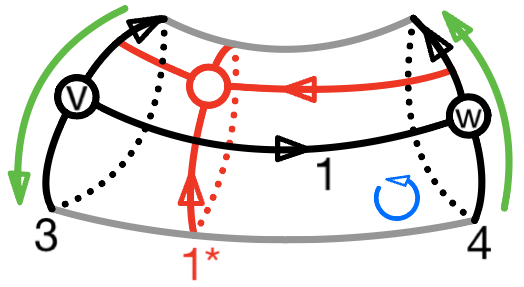}
    \caption{A reference orientation of a ribbon graph and the dual reference orientation (left), and a connected component $\Sigma'$ of $\Sigma$ minus edges $3$ and $4$ (right). The induced orientation on $\partial \Sigma'$ is colored green.}
    \label{fig:orientation}
\end{figure}

\begin{theorem}\label{thm:ribbon representation}
    $\cC(\vG)$ is a regular orthogonal representation of $M(\vG)$.
\end{theorem}

Theorem~\ref{thm:ribbon representation} can be derived from the results of Booth, Borovik, Gelfand, and Stone~\cite{BBGS2000}, as described below. It also follows from earlier work of Bouchet~\cite{Bouchet1987c}, who proved a corresponding result in terms of circle graphs. We refer the reader to Appendix~\ref{subsec:bouchet} for a comparison of his approach with that presented in this paper.

The \emph{intersection index} $(p,q) \in \mathbb{Z}$ of two curves $p$ and $q$ on $\Sigma$ is defined as the number of times $p$ traverses $q$ from left to right minus the number of times $p$ traverses $q$ from right to left.
For example, $(e,e^*)=1$ for each edge $e$.
For a cycle $c$ in the homology group $H_1(\Sigma\setminus (V\cup V^*)) := H_1(\Sigma\setminus (V\cup V^*),\mathbb{Q})$, let
\begin{align*}
    \iota(c) := \sum_{e\in E\cup E^*} (c,e) e \in \mathbb{Q}^{E\cup E^*}.
\end{align*} 

\begin{theorem}[\cite{BBGS2000}]\label{thm:BBGS}
    For all $X,Y\in \iota(H_1(\Sigma\setminus (V\cup V^*)))$, we have $\langle X,Y \rangle = 0$.
\end{theorem}

\begin{lemma}\label{lem:vC and fudamental circuit}
    Let $\vC\in \cC(\vG)$ be a signed ribbon cycle.
    Let $B$ be a basis of $\vG$ and $x$ be an element of $E$ such that $\mathrm{supp}(\vC) \subseteq B\triangle \{x,x^*\}$.
    Then 
    \begin{enumerate}[label=\rm(\roman*)]
        \item a small $\epsilon$-neighborhood of $(V, (E\cap B)\triangle \{x\})$ in $\Sigma$ has exactly two boundary components, say $c$ and $c'$, and
        \item $\{ \iota(c),\iota(c') \} = \{ \vC^*, -\vC^* \}$ as sets.
    \end{enumerate}
\end{lemma}

\begin{proof}
    (i)
    By Lemma~\ref{lem:ribbon basis}, a small $\epsilon$-neighborhood $N_\epsilon$ of $(V,(E\cap B)\triangle \{x\})$ in $\Sigma$ has exactly two boundary components $c$ and $c'$.

    (ii)
    We can regard $c$ and $c'$ as elements in $H_1(\Sigma \setminus (V\cup V^*))$.
    Orient $c$ and $c'$ according to the induced orientation on $\partial N_\epsilon$. 
    Let $\Sigma_1$ and $\Sigma_2$ be the two connected components of $\Sigma \setminus \mathrm{supp}(\vC)$.
    Because $\mathrm{supp}(\vC) \subseteq N_\epsilon$, $c$ and $c'$ are in the different components of $\Sigma \setminus \mathrm{supp}(\vC)$.
    We may assume that $c\subseteq\Sigma_1$ and $c'\subseteq \Sigma_2$.
    See Figure~\ref{fig: fundamental ribbon cycle} for an example.

    An edge or coedge $e$ in $B\triangle\{x,x^*\}$ does not intersect either $c$ or $c'$, while
    each $e\in (B\triangle \{x,x^*\})^*$ intersects the union of $c$ and $c'$ exactly twice.
    
    Suppose that $e\in (B\triangle \{x,x^*\})^* \setminus \mathrm{supp}(\vC)^*$.
    Then $e$ is entirely contained in either $\Sigma_1$ or $\Sigma_2$, and so it intersects only one of $c$ and $c'$.
    Hence, $(c,e) = 0 = (c',e)$.
    
    Suppose that $e\in \mathrm{supp}(\vC)^*$.
    Then $e$ intersects both $c$ and $c'$.
    Moreover, $(c,e) = -(c',e)$ by our choice of orientations on $c$ and $c'$.
    Therefore, $\iota(c) = -\iota(c')$ and the supports of $\iota(c)$ and $\iota(c')$ are $\mathrm{supp}(\vC)^*$.

    Let $e\in \mathrm{supp}(\vC)$.
    Then the (co)edge $e$ lies on the boundary of $\Sigma_1$.
    Recall that $c$ is one of the two boundary components of $N_\epsilon$, which is contained in $\Sigma_1$.

    Suppose first that $e\in E$.
    Then $e$ is contained in $N_\epsilon$.
    Hence, both $e$ and $c$ lie on $\partial_1 := \partial(\Sigma_1 \cap N_\epsilon)$, and $c$ agrees with the induced orientation on $\partial_1$.
    Let $\bar{e} = e$ if $e$ agrees with the induced orientation on $\partial_1$ and $\bar{e} = -e$ otherwise.
    Then $(\bar{e},e^*) = \vC(e)$.
    Since $\partial_1$ crosses $e^*$ exactly twice in opposite directions, we have 
    \begin{align*}
        0 
        = (\partial (\Sigma_1 \cap N_\epsilon), e^*) 
        = (c,e^*) + (\bar{e}, e^*)
        = (c,e^*) + \vC(e).
    \end{align*}
    
    Now suppose that $e\in E^*$.
    Then $e$ is contained in $\Sigma \setminus N_\epsilon$. 
    Hence, both $e$ and $c$ lie on $\partial_2 := \partial(\Sigma_1 \setminus N_\epsilon)$, and $c$ is oriented oppositely to 
    the induced orientation on $\partial_2$.
    Let $\bar{e} = e$ if $e$ agrees with the induced orientation on $\partial_2$ and $\bar{e} = -e$ otherwise.
    Then $(e^*,\bar{e}) = \vC(e)$. Since $\partial_2$ crosses $e^*$ exactly twice in opposite directions, we have
    \begin{align*}
        0 
        = (\partial (\Sigma_1 \setminus N_\epsilon), e^*) 
        = -(c,e^*) + (\bar{e}, e^*)
        = -(c,e^*) - \vC(e).
    \end{align*}
    Therefore, $\vC^* = -\iota(c)$.
\end{proof}

\begin{figure}
    \centering
    \includegraphics[width=0.5\linewidth]{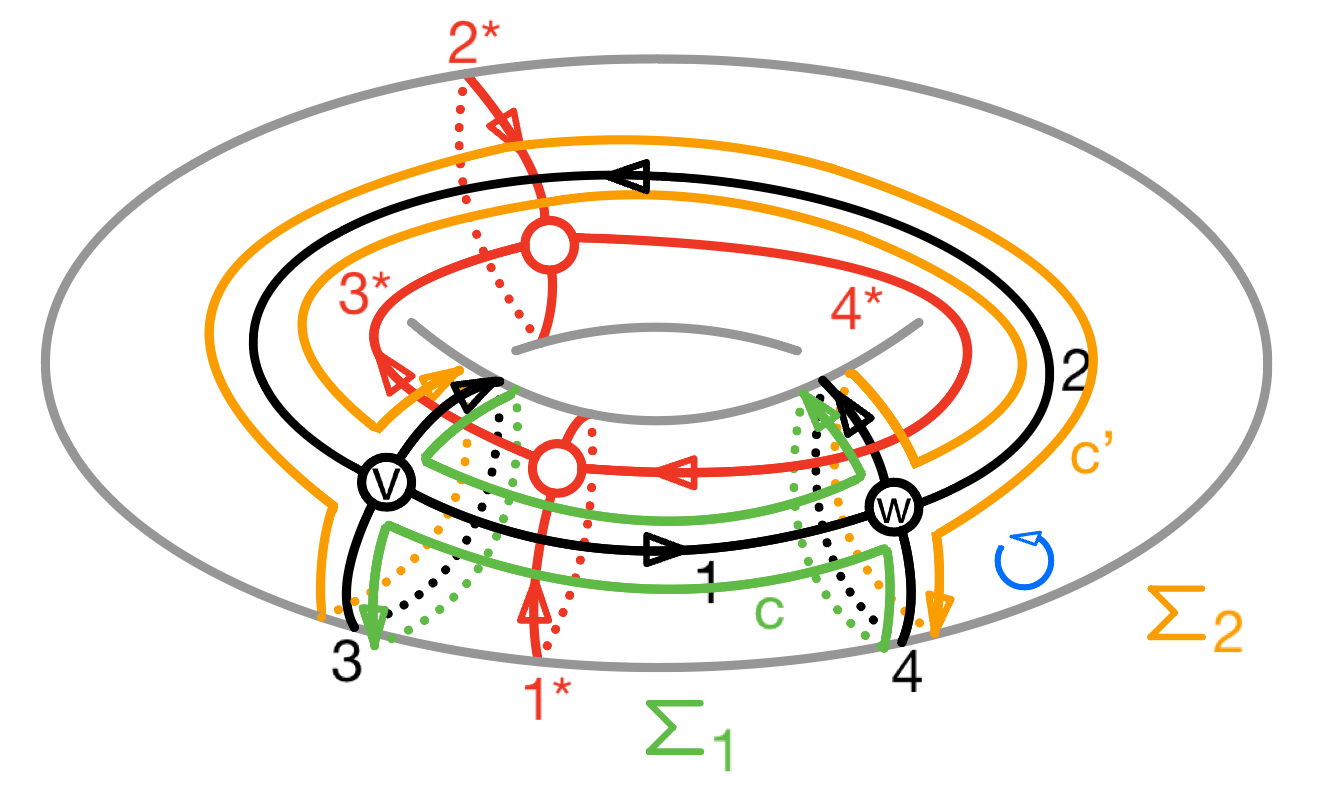}
    \caption{An example for the proof of Lemma~\ref{lem:vC and fudamental circuit}, applied to the ribbon graph in Figure~\ref{fig:orientation}, $B=1234^*$, and $x=4$.
    The two boundaries $c$ and $c'$ of a small $\epsilon$-neighborhood of $(V, (E\cap B)\triangle\{x\})$ are colored by green and orange, respectively.
    Note that $\FC(B,x) = 34$.
    }
    \label{fig: fundamental ribbon cycle}
\end{figure}

\begin{proof}[Proof of Theorem~\ref{thm:ribbon representation}]
    Let $\vC, \vD \in \cC(\vG)$.
    By Lemma~\ref{lem:vC and fudamental circuit}, there are cycles $c$ and $d$ in $H_1(\Sigma\setminus (V\cup V^*))$ such that $\vC = \iota(c)^*$ and $\vD = \iota(d)^*$.
    Then $\langle \vC, \vD \rangle = \langle \iota(c), \iota(d) \rangle = 0$ by Theorem~\ref{thm:BBGS}.
    Therefore, $\cC(\vG)$ is a regular orthogonal representation of $M(\vG)$.
\end{proof}

\begin{corollary}
    The span of $\cC(\vG)^*$ is equal to $\iota(H_1(\Sigma \setminus (V\cup V^*)))$.
\end{corollary}
\begin{proof}
    Let $B$ be a basis of $\vG$, and for each $e\in B^*$ let $\vC_e$ be a signed ribbon cycle of $\vG$ such that $\mathrm{supp}(\vC_e) = \FC(B,e)$.
    Then the $n$ vectors $\vC_e^*$ with $e\in B^*$ are in $\iota(H_1(\Sigma \setminus (V\cup V^*)))$ by Lemma~\ref{lem:vC and fudamental circuit}.
    Moreover, they are linearly independent, since for $f\in B^*$ we have $\vC_e(f) \ne 0$ if and only $e=f$.
    By Theorem~\ref{thm:BBGS}, $\iota(H_1(\Sigma \setminus (V\cup V^*)))$ has dimension at most $n$, and therefore it is spanned by $\cC(\vG)^*$.
\end{proof}

\begin{remark}\label{rem: reference orientation}
Different choices of a reference orientation of $G$ yield regular orthogonal representations of $M(\vG)$ which differ only by reorientations.
Choose two reference orientations~$\vec{G}_1$ and~$\vec{G}_2$ of $G$. 
Let $X$ be the set of edges on which $\vec{G}_1$ and $\vec{G}_2$ differ.
Then for any $e\in E\cup E^*$, the intersection indices $(c,e)$ are the same for $\vec{G}_1$ and $\vec{G}_2$ if and only if $e\notin X\cup X^*$.
Therefore, $\cC(\vG,\vec{G}_1) = \cC(\vG,\vec{G}_2)^X$. 
\end{remark}

\begin{remark}
We have two choices of ``consistent'' dual orientations of $G^*$.
One convention is that each coedge $e^*$ crosses its corresponding edge $e$ from right to left, as above, and the other
is that each coedge $e^*$ crosses $e$ from left to right.
These choices yield two regular orthogonal representations of $M(\vG)$ which differ from each other by a simple operation.
For a fixed reference orientation of $G$, consider the two regular orthogonal representations $\cC_1(\vG)$ and $\cC_2(\vG)$ induced by the two different choices for dual orientations of $G^*$.
Then for any $e\in E$, the intersection indices $(c,e)$ are the same, but $(c,e^*)$ are different in the two cases.
This implies $\cC_1(\vG) = \{\vC' : \vC\in \cC_2(\vG)\}$, where
\begin{align*}
    \vC'(x) = 
    \begin{cases}
        \vC(x) & \text{if } x\in E, \\
        -\vC(x) & \text{if } x\in E^*. \\
    \end{cases}
\end{align*}
\end{remark}

\section{The Jacobian group and the circuit reversal classes}\label{sec: cycle reversal}

In this section, we first define the Jacobian group $\J(\cC)$ and the circuit reversal classes $\cR(\cC)$ of a regular orthogonal representation $\cC$. Then we will prove that there is a canonical group action of $\J(\cC)$ on the set $\cR(\cC)$. 

\subsection{Definition of the Jacobian group}
Let $\cC$ be a regular orthogonal representation of an orthogonal matroid $M=(E\cup E^*, \cB)$. 

We define the function \[\pi\colon\Q^{E\cup E^*}\to\Q^E\] by \[\pi(\vv)(e)=\vv(e)+\vv(e^*).\]

Define \[\pcc:=\text{the subgroup of }\Z^E\text { generated by }\{\pi(\vC):\vC\in\cC\}.\]

We remark that $\pi(\vC)\in\{0, \pm 1\}^E$ for any signed circuit $\vC$, because
the support of any signed circuit $\vC$ is a subtransversal. 

We define an equivalence relation $\sim_\mathrm{J}$ on $\Q^E$ by the rule
\[\vv_1\sim_\mathrm{J}\vv_2\Leftrightarrow \vv_2-\vv_1\in\pcc.\]
Note that $\sim_\mathrm{J}$ induces an equivalence relation on $\Z^E$. We denote an equivalence class for this relation by $\bJ{\vv}$. 

\begin{definition}\label{def: Jac}
The \emph{Jacobian group} of the regular orthogonal representation $\cC$ is defined to be the quotient group \[\J (\cC):=\Z^{E}/\sim_\mathrm{J}.\] 
\end{definition}

\begin{proposition}\label{prop: the order of Jac}
The order of the Jacobian group $\J(\cC)$ equals the number of bases of $M$.
\end{proposition}
\begin{proof}
By Lemma~\ref{lem: equal row span}, the subgroup $\pcc$ of $\Z^E$ is spanned by the rows of $I+A$. Then by Corollary~\ref{cor:counting bases}, we get the desired formula. 
\end{proof}

\begin{example}
    Let $\cC$ be the set of signed circuits calculated in Example~\ref{ex: circuits}. It is easy to see that the subgroup $\pcc$ is generated by $(1,-1,0,0)$, $(1,0,-1,0)$, $(1,0,0,-1)$ and $(1,1,1,1)$. The Jacobian group is \[\J(\cC)=\{\bJ{(0,0,0,0)}, \bJ{(1,0,0,0)}, \bJ{(1,1,0,0)}, \bJ{(1,1,1,0)}\}\] which is isomorphic to $\Z/4\Z$. 
\end{example}

\subsection{The circuit reversal system}

\begin{definition}
\label{def:circuit reversal class}
We define the circuit reversal system for $\cC$ as follows. 
\begin{enumerate}[label=\rm(\roman*)]
    \item An \emph{orientation} (with respect to $\cC$) is an element of the set \[\cO:=\{\pm \frac{1}{2}\}^E.\] 

    \item A \emph{circuit reversal} turns an orientation $\vO$ into the orientation $\vO+\pi(\vC)$ for some $\vC\in\cC$. 

    \item Two orientations $\vO_1$ and $\vO_2$ are said to be \emph{equivalent}, denoted $\vO_1\sim_\mathrm{R}\vO_2$, if one orientation can be obtained from the other by a sequence of circuit reversals. This defines an equivalence relation on the set of orientations. We call an equivalence class for this relation a \emph{circuit reversal class}. The circuit reversal class of an orientation $\vO$ is denoted by $\bR{\vO}$, and the set of circuit reversal classes of $\cC$ is denoted by $\cR(\cC)$. We call $\cR(\cC)$ the \emph{circuit reversal system} for $\cC$.
\end{enumerate}
\end{definition}

\begin{example}\label{ex: circuit reversal classes}
     We continue with the example in Figure~\ref{fig:torus}. We denote its orientations by $\vO_{++++}=(\frac{1}{2},\frac{1}{2},\frac{1}{2},\frac{1}{2})$, $\vO_{+++-}=(\frac{1}{2},\frac{1}{2},\frac{1}{2},-\frac{1}{2})$, etc. We have $\vO_{++++}\sim_\mathrm{R} \vO_{----}$ because $\vO_{++++}-\vO_{----}=(1,1,1,1)=\pi(\vC)$, where $\vC=(1,1,0,0, 0,0,1,1)$ is the last signed circuit in Example~\ref{ex: circuits}. 
     The complete set of circuit reversal classes is as follows:
     \begin{align*}
        &\bR{\vO_{++++}}=\{\vO_{++++}, \vO_{----}\},\\
        &\bR{\vO_{+++-}}=\{\vO_{+++-}, \vO_{++-+}, \vO_{+-++}, \vO_{-+++}\},\\
        &\bR{\vO_{++--}}=\{\vO_{++--}, \vO_{+-+-}, \vO_{-++-}, \vO_{+--+}, \vO_{-+-+}, \vO_{--++}\},\\
        &\bR{\vO_{+---}}=\{\vO_{+---}, \vO_{-+--}, \vO_{--+-}, \vO_{---+}\}.
    \end{align*}
\end{example}

\begin{definition}\label{def: action}
    We define a group action of $\J(\cC)$ on $\cR(\cC)$ as follows. For $\bJ{\vv}\in\J(\cC)$ and $\bR{\vO_1}\in \cR(\cC)$, we define 
    \[\bJ{\vv}\cdot \bR{\vO_1}:=\bR{\vO_2},\]
    where $\vO_2$ is any orientation such that $\vv+\vO_1 \sim_\mathrm{J} \vO_2$.
\end{definition}

Our first main result of the paper is: 

\begin{theorem}\label{thm: action}
    The group action in Definition~\ref{def: action} is well defined and simply transitive.
\end{theorem}

\begin{corollary}\label{cor: count the circuit reversal classes}
We have\[|\J(\cC)|=|\cR(\cC)|=|\cB(M)|.\]
\end{corollary}
\begin{proof}
This is a direct consequence of Proposition~\ref{prop: the order of Jac} and Theorem~\ref{thm: action}. 
\end{proof}

In order to prove Theorem~\ref{thm: action}, we must first show that the orientation $\vO_2$ in Definition~\ref{def: action} always exists. The proof uses Farkas' Lemma for oriented orthogonal matroids (Lemma~\ref{lem: Farkas}), which for the present purpose we reformulate as follows:

\begin{lemma}\label{lem: sign}
For any function $f\colon E\to \{\pm 1\}$ and any $i\in E$, there exists a signed circuit $\vC\in\cC$ such that $\pi(\vC)(i)=f(i)$ and $\pi(\vC)(j)\in \{ 0, f(j) \}$ for every $j\in E$.

\end{lemma}
\begin{proof}
    Apply Lemma~\ref{lem: Farkas} to the reorientation $\cC^{f^{-1}(-1)}$.
\end{proof}

Define
\[
\hZ^E:=\text{the set of functions from }E\text{ to }\{n+\frac{1}{2}:n\in\Z\}
\]
and note that $\sim_\mathrm{J}$ induces an equivalence relation on $\hZ^E$. 
We again denote an equivalence class for this relation by $\bJ{\vv}$. 

\begin{lemma}\label{lem: cube represents all}
For any $\vv\in\hZ^E$, there exists an orientation $\vO$ such that $\bJ{\vv}=\bJ{\vO}$.
\end{lemma}
\begin{proof}
    We define the \emph{defect} of $\vv$ to be
    \[
    \delta(\vv) := \sum_{j\in E} (|\vv(j)|-\frac{1}{2}).
    \] 
    
    Note that $\delta(\vv)$ is a non-negative integer, and it is zero if and only if $\vv\in\cO$. 
    
    We may assume that $\delta(\vv)$ is positive. It is enough to show that there exists a signed circuit $\vC$ such that $\delta(\vv')<\delta(\vv)$, where $\vv':=\vv + \pi(\vC)$, because we can repeat this process until the defect becomes zero.  
    
    Because $\delta(\vv)>0$, there exists $i\in E$ such that $\vv(i) \notin \{\pm\frac{1}{2}\}$. Let $f : E\to \{\pm 1\}$ be the function such that for any $j\in E$, 
    \[f(j) = \begin{cases}
        1 & \text{if } \vv(j) < 0, \\
        -1 & \text{if } \vv(j) > 0. \\
    \end{cases}\]
    By Lemma~\ref{lem: sign}, there exists a signed circuit $\vC$ such that $\pi(\vC)(i)=f(i)$ and $\pi(\vC)(j)\in\{0,f(j)\}$ for every $j\in E$. Then $\vC$ is the desired signed circuit, because 
    \[|\vv(j)|\geq |\vv(j)+\pi(\vC)(j)|=|\vv'(j)|\]
    for every $j\in E$, and when $j=i$ we have strict inequality. 
\end{proof}

\begin{corollary}\label{cor: R'}
Define
\[\cR'(\cC):=\{\bJ{\vO}:\vO\in\cO\}.\] 
    \begin{enumerate}[label=\rm(\roman*)]
    \item We have \[\cR'(\cC)=\hZ^E/\sim_\mathrm{J}\]as sets. 
    \item The group action of $\J(\cC)$ on $\cR'(\cC)$ defined by\[\bJ{\vv}\cdot \bJ{\vO_1}=\bJ{\vO_2}\Leftrightarrow\bJ{\vv+\vO_1}=\bJ{\vO_2}\]is well defined and simply transitive.
    \end{enumerate}
\end{corollary}
\begin{proof}
The first part is a consequence of Lemma~\ref{lem: cube represents all}. 
The second part follows from the fact that the Jacobian group $\Z^E/\sim_\mathrm{J}$ acts on $\hZ^E/\sim_\mathrm{J}$ simply transitively. 
\end{proof}

The next step is to show that the combinatorial equivalence relation $\vO_1\sim_\mathrm{R}\vO_2$ coincides with the algebraic equivalence relation $\vO_2\sim_\mathrm{J}\vO_1$. 
To prove this, we will need two more lemmas. 

\begin{lemma}[{\cite[Lemma~4.1.1]{BBY2019}}]\label{lem: disjoint circuits}
Let $\cC$ be a regular representation of a \emph{matroid}. Then every $(0,\pm 1)$-vector in the span of $\cC$ is a sum of disjoint signed circuits in $\cC$.
\end{lemma}

\begin{lemma}\label{lem: subtransveral}
Let $\vv$ be a $\Z$-linear combination of elements in $\cC$ such that $\pi(\vv)\in\{0,\pm 1\}^E$. Then the support of $\vv$ is a subtransversal.   
\end{lemma}
\begin{proof}
By Definition~\ref{def: regular representation}, we have
    \begin{equation}\label{eqn: 1}
    \sum_{e\in E} \vv(e)\vv(e^*) = \frac{1}{2}\langle \vv, \vv\rangle = 0.
    \end{equation}
Assume for the sake of contradiction that $\vv(i)$ and $\vv(i^*)$ are both nonzero for some $i\in E$.
Because $\vv(i)+\vv(i^*)\in \{0,\pm 1\}$, we have $\vv(i) \vv(i^*) < 0$. Thus, by (\ref{eqn: 1}), there exists $j\in E$ such that $\vv(j) \vv(j^*) > 0$. 
It follows that $|\pi(\vv)(j)| = |\vv(j) + \vv(j^*)| \ge 2$, which contradicts our assumption that $\pi(\vv)\in\{0,\pm 1\}^E$. 
\end{proof}

\begin{proposition}\label{prop: R-equivalence}
Let $\vO_1$ and $\vO_2$ be two orientations. Then \[\vO_1\sim_\mathrm{R}\vO_2 \Leftrightarrow\vO_2\sim_\mathrm{J}\vO_1.\]
Moreover, when $\vO_1\sim_\mathrm{R}\vO_2$, there exist signed circuits $\vC_1,\cdots,\vC_m\in\cC$ (where $m$ is possibly zero) such that \begin{itemize}
    \item $\vO_2-\vO_1=\sum_{k=1}^{m}\pi(\vC_k)$, 
    \item the supports of these signed circuits are disjoint, and
    \item the union of these supports is a subtransversal.
\end{itemize}
\end{proposition}

\begin{proof}
    The forward direction $\vO_1\sim_\mathrm{R}\vO_2 \Rightarrow\vO_2\sim_\mathrm{J}\vO_1$ is trivial.
    To prove the backward direction, assume that $\vO_2-\vO_1\in\pcc$. 
    Let $\vv$ be a $\Z$-linear combination of elements in $\cC$ such that \[\vO_2-\vO_1=\pi(\vv).\] 
    Since $\vO_1,\vO_2\in \cO$, we have $\pi(\vv)\in\{0,\pm 1\}^E$. Hence, by Lemma~\ref{lem: subtransveral}, the support of $\vv$ is a subtransversal. 
    
    By Proposition~\ref{prop: restriction} and Lemma~\ref{lem: disjoint circuits}, we have
    \[\vv=\sum_{k=1}^m\vC_k,\]
    where the $\vC_k$'s are signed circuits in $\cC$ with disjoint supports contained in $\mathrm{supp}(\vv)$. Thus the $\pi(\vC_k)$'s also have disjoint supports (in $E$). This implies that the partial sum $\vO_1+\pi(\vC_1)+\cdots+\pi(\vC_{m'})$ is always an orientation for any $m'<m$. Therefore we get $\vO_1\sim_\mathrm{R}\vO_2$, together with the desired signed circuits. 
\end{proof}

\begin{proof}[Proof of Theorem~\ref{thm: action}]
By Proposition~\ref{prop: R-equivalence}, the set $\cR'(\cC)$ in Corollary~\ref{cor: R'} is the same as the set $\cR(\cC)$, and the group action in Corollary~\ref{cor: R'} is the same as the one in Definition~\ref{def: action}. 
\end{proof}

\subsection{Invariance under reorientation}\label{subsec: reorientation invariance}
As before, let $\cC$ be a regular orthogonal representation of an orthogonal matroid $M=(E\cup E^*, \cB)$. 
Fix a subset $X$ of $E$. Recall that the reorientation of $\cC$ on $X$ is defined as
    \[\cC^{X} := \{\vC^{X} \in \mathbb{Z}^{E\cup E^*} : \vC\in\cC\},\] where
    \begin{align*}
        \vC^{X}(i) = \begin{cases}
            -\vC(i) & \text{if } i\in X\cup X^*, \\
            \vC(i) & \text{otherwise}.
        \end{cases}
    \end{align*}
\begin{definition}
Given $\vv\in\Q^{E}$, we define a new element $\vv^X\in\Q^{E}$, called the \emph{reorientation} of $\vv$ on $X$, by
\[
        \vv^{X}(i) = \begin{cases}
            -\vv(i) & \text{if } i\in X, \\
            \vv(i) & \text{otherwise}.
        \end{cases}
\]
In particular, it makes sense to write $\vO^X$ for an orientation $\vO$. 
\end{definition}
Note also that $\pi(\vC^X)=(\pi(\vC))^X$. 

\begin{lemma}\label{lem: reorientation and Jac}
We have a natural group isomorphism
\begin{align*}
   \J(\cC)  &\to \J(\cC^X)\\
  \bJ{\vv} &\mapsto \bJ{\vv^X}.
\end{align*}
\end{lemma}
\begin{proof}
The map $\vv\mapsto\vv^X$ is a group isomorphism from $\Z^E$ to $\Z^E$ sending the subgroup $\pcc$ to $\langle \pi(\cC^X)\rangle$. Therefore it induces the desired isomorphism. 
\end{proof}

\begin{lemma}\label{lem: reorientation and circuit reversal}
We have a natural bijection
\begin{align*}
   \cR(\cC)  &\to \cR(\cC^X)\\
  \bR{\vO} &\mapsto \bR{\vO^X}.
\end{align*}
\end{lemma}
\begin{proof}
This follows from the easily verified fact that $\vO_1\sim_\mathrm{R}\vO_2$ in the definition of $\cR(\cC)$ if and only if $\vO_1^X\sim_\mathrm{R}\vO_2^X$ in the definition of $\cR(\cC^X)$. 
\end{proof}

    We will use the notation 
    $\bJ{\vv}^X:=\bJ{\vv^X}$ and $\bR{\vO}^X:=\bR{\vO^X}$ 
    to emphasize the fact that 
    $\bJ{\vv}^X\in\J(\cC^X)$ and $\bR{\vO}^X\in\cR(\cC^X)$ respectively. 

\begin{lemma}\label{lem: reorientation and action}
For any $\bJ{\vv}\in\J(\cC)$ and $\bR{\vO_1},\bR{\vO_2}\in \cR(\cC)$, we have  
    \[\bJ{\vv}\cdot \bR{\vO_1}=\bR{\vO_2}\Leftrightarrow\bJ{\vv}^X\cdot \bR{\vO_1}^X=\bR{\vO_2}^X.\]

\end{lemma}

\begin{proof}
By definition, the left hand side is equivalent to $\vv+\vO_1-\vO_2\in\pcc$, and the right hand side is equivalent to $\vv^X+\vO_1^X-\vO_2^X\in\langle \pi(\cC^X)\rangle$. It is clear that the two formulas are equivalent. 
\end{proof}

Putting together the various results proved in this section, we see that $\J(\cC)$, $\cR(\cC)$, and the action of $\J(\cC)$ on $\cR(\cC)$ all depend only on the regular orthogonal representation $\cC$ up to reorientation equivalence. 

\subsection{The lift of circuit reversals} \label{sec:liftsection}
The reader might have noticed that the signed circuits of a regular orthogonal representation $\cC$ are elements of $\Q^{E \cup E^*},$ but the circuit reversal system is defined in terms of an equivalence relation on orientations, which are elements of $\Q^E$. 
The map $\pi : \Q^{E \cup E^*} \to \Q^E$ is used to compare the two sides.
It is natural to wonder if one can ``lift'' orientations to $\Q^{E \cup E^*}$ instead of ``projecting'' circuits to $\Q^E$.
In this section, we show that this is indeed the case.

Although the results of this section will not be used elsewhere in the paper, it seems worth recording this alternative description of the equivalence relation $\sim_R$ generated by circuit reversals.

\begin{definition}\label{def: C and O}
Let $\vO$ be an orientation and let $\vC$ be a signed circuit. If $\pi(\vC)(f)\cdot\vO(f)>0$ for all $f\in \mathrm{supp}(\pi(\vC))$, we say that the signed circuit $\vC$ is \emph{compatible} with the orientation $\vO$.
\end{definition}

We now define a ``lifting'' map $\psi : \cO = \{\pm\frac{1}{2}\}^{E} \to \{0,\pm\frac{1}{2}\}^{E\cup E^*}$.

\begin{definition}\label{def: psi}
Let $\vO$ be an orientation. For any element $e\in E$, 
it follows from Lemma~\ref{lem: Farkas} that %
either: 
\begin{enumerate}[label=\rm(\roman*)]
\item there exists a signed circuit $\vC$ compatible with $\vO$ such that $\vC(e)\neq 0$, or
\item there exists a signed circuit $\vC$ compatible with $\vO$ such that $\vC(e^*)\neq 0$,
\end{enumerate}
but not both. Let $\psi(\vO)\in\Q^{E\cup E^*}$ be the unique element such that: 
\begin{itemize}
    \item $\mathrm{supp}(\psi(\vO))$ is a transversal of $E\cup E^*$,
    \item $\pi(\psi(\vO))=\vO$, and
    \item for any $e\in E$, $\psi(\vO)(e)=0$ iff we are in case (ii), i.e., $\vC(e)=0$ for every signed circuit $\vC$ compatible with $\vO$.
\end{itemize}
\end{definition}
Note that the first two properties of $\psi(\vO)$ imply that 
\[\psi(\vO)\in\{0,\pm\frac{1}{2}\}^{E\cup E^*}.\]

\begin{example}
     In our running example from Figure~\ref{fig:torus}, we have \[\psi(\vO_{++++})=\begin{pmatrix}
         \frac{1}{2} & \frac{1}{2} &0&0\\ 0&0&\frac{1}{2}&\frac{1}{2}\\
     \end{pmatrix}\]and\[\psi(\vO_{+++-})=\begin{pmatrix}
         0&0&\frac{1}{2}&-\frac{1}{2}\\ \frac{1}{2}&\frac{1}{2}&0&0\\
     \end{pmatrix}.\]
\end{example}

The next lemma follows immediately from Definition~\ref{def: psi}.

\begin{lemma}\label{lem: C and O 1}
Let $\vO$ be an orientation and let $\vC$ be a signed circuit compatible with $\vO$. Then $\mathrm{supp}(\vC)\subseteq\mathrm{supp}(\psi(\vO))$. Moreover, for all $f\in\mathrm{supp}(\vC)$ we have $\vC(f)=2\psi(\vO)(f)$. 
\end{lemma}

The following lemma gives an equivalent definition of $\psi(\vO)$. 
\begin{lemma}\label{lem: C and O 2}
Let $\vO$ be an orientation and let $\vv\in\Q^{E\cup E^*}$ be an element such that $\mathrm{supp}(\vv)$ is a transversal and $\pi(\vv)=\vO$. Then $\vv=\psi(\vO)$ if and only if for every $e\in\mathrm{supp}(\vv)$, there exists a signed circuit $\vC$ such that $e\in\mathrm{supp}(\vC)$ and $\vC(f)=2\vv(f)$ for all $f\in\mathrm{supp}(\vC)$. 
\end{lemma}
\begin{proof}
The ``if'' part follows from Definition~\ref{def: psi}.

For the ``only if'' part, suppose $\vv=\psi(\vO)$. Then by Definition~\ref{def: psi}, for every $e\in\mathrm{supp}(\vv)$, there exists a signed circuit $\vC$ compatible with $\vO$ such that $\vC(e)\neq 0$. 
We want to show that for every $f\in\mathrm{supp}(\vC)$, $\vC(f) = 2\vv(f)$. We apply Definition~\ref{def: psi} to $\vO$ and the unique element in $\{f,f^*\}\cap E$. 
Because $\vC(f)\neq 0$, we have $\psi(\vO)(f)\neq 0$.  
Since $\pi(\vC)(f)\cdot\vO(f)>0$ and $|\pi(\vC)(f)|=2|\vO(f)|$, we must have $\vC(f)=2\psi(\vO)(f)$. 
\end{proof}

We can now give the following alternative description of circuit reversals. 

\begin{proposition}\label{prop: pullback}
Let $\vO_1$ and $\vO_2$ be orientations, and let $\vC$ be a signed circuit. Then\[\vO_2=\vO_1+\pi(\vC)\Leftrightarrow\psi(\vO_2)=\psi(\vO_1)+\vC.\]
Moreover, if one of them holds, then \[\mathrm{supp}(\vC)\subseteq\mathrm{supp}(\psi(\vO_1))=\mathrm{supp}(\psi(\vO_2)).\]
\end{proposition}
\begin{proof}
The backward direction is trivial. To prove the forward direction, we assume $\vO_2=\vO_1+\pi(\vC)$. 

Clearly, the signed circuits $-\vC$ and $\vC$ are compatible with $\vO_1$ and $\vO_2$ respectively. By Lemma~\ref{lem: C and O 1}, we have\begin{itemize}
    \item $\mathrm{supp}(\vC)\subseteq\mathrm{supp}(\psi(\vO_1))$, and
    \item $\vC(f)=-2\psi(\vO_1)(f)=2\psi(\vO_2)(f)$ for all $f\in\mathrm{supp}(\vC)$. 
\end{itemize}
Thus the vector\[\vv:=\psi(\vO_1)+\vC\in\Q^{E\cup E^*}\] has the following properties: 
\begin{itemize}
    \item $\mathrm{supp}(\vv)=\mathrm{supp}(\psi(\vO_1))$. In particular, $\mathrm{supp}(\vv)$ is a transversal. 
    \item $\pi(\vv)=\pi(\psi(\vO_1))+\pi(\vC)=\vO_2$. 
\end{itemize}

By Lemma~\ref{lem: C and O 2}, to prove $\vv=\psi(\vO_2)$, we need to show that for every $e\in\mathrm{supp}(\vv)$, there exists $\vC_2\in\cC$ such that $e\in\mathrm{supp}(\vC_2)$ and $\vC_2(f)=2\vv(f)$ for all $f\in\mathrm{supp}(\vC_2)$. 

Fix $e\in\mathrm{supp}(\vv)=\mathrm{supp}(\psi(\vO_1))$. There are two cases to consider:
\begin{itemize}
    \item If $e\in\mathrm{supp}(\vC)$, then $\vC_2:=\vC$ is the desired vector.
    \item If $e\notin\mathrm{supp}(\vC)$, then by applying Lemma~\ref{lem: C and O 2} to $\vO_1$ we find $\vC_1\in\cC$ such that $e\in\mathrm{supp}(\vC_1)$ and $\vC_1(f)=2\psi(\vO_1)(f)$ for all $f\in\mathrm{supp}(\vC_1)$. Let $\vC_2':=\vC_1+\vC$. For each $g\in \mathrm{supp}(\vC)\cap\mathrm{supp}(\vC_1)$, we have $\vC_1(g)=2\psi(\vO_1)(g)=-\vC(g)$. Hence $\vC_2'$ is a $(0,\pm 1)$-vector in the span of $\cC$. By applying Proposition~\ref{prop: restriction} to the transversal $\mathrm{supp}(\vv)$ and then Lemma~\ref{lem: disjoint circuits} to $\vC_2'$, we can write $\vC_2'$ as a sum of disjoint signed circuits with supports contained in $\mathrm{supp}(\vv)$. The one containing $e$ is the desired vector $\vC_2$. 
\end{itemize}

Therefore $\vv=\psi(\vO_2)$, and hence $\psi(\vO_2)=\psi(\vO_1)+\vC$. It is clear that $\mathrm{supp}(\vC)\subseteq\mathrm{supp}(\psi(\vO_1))=\mathrm{supp}(\psi(\vO_2))$.
\end{proof}

As a direct consequence of Proposition~\ref{prop: pullback}, we obtain:
\begin{corollary}
Let $\vO_1, \vO_2, \vC_1, \cdots, \vC_m$ be as in Proposition~\ref{prop: R-equivalence}. Then we have\begin{enumerate}[label=\rm(\roman*)]
\item $\mathrm{supp}(\vC_k)\subseteq\mathrm{supp}(\psi(\vO_1))=\mathrm{supp}(\psi(\vO_2))$ for every $k$, and

\item $\psi(\vO_2)-\psi(\vO_1)=\sum_{k=1}^{m}\vC_k$.

\end{enumerate}
\end{corollary}

\section{The BBY bijections and torsors of regular orthogonal representations}\label{sec: BBY}
The BBY bijections between bases and circuit-cocircuit reversal classes of regular matroids are defined in \cite{BBY2019}. 
In this section we define analogous bijections for regular orthogonal representations. 

The proof of bijectivity in \cite{BBY2019} involves zonotope tilings, and it is not clear how to define similar geometric objects for orthogonal matroids. We will therefore adopt the machinery introduced in \cite{Ding2024}, where the BBY bijections are generalized and the proof of bijectivity is purely combinatorial. We state and prove our general BBY bijections for regular orthogonal representations in Section~\ref{subsec: BBY bijection}. 
In Section~\ref{subsec: signature}, we introduce the notion of ``acyclic signatures'' and connect this to the results of the previous section.
Finally, in Section~\ref{subsec: BBY torsor} we define and study the corresponding BBY torsors. %

Throughout this section we fix a regular orthogonal representation $\cC$ of a regular orthogonal matroid $M$ with ground set $E\cup E^*$.

\subsection{The BBY bijections}\label{subsec: BBY bijection}

For the definition of fourientations and related notions, see Section~\ref{sec: farkas}.

Recall that orientations are elements in $\cO=\{\pm\frac{1}{2}\}^{E}$. 
For an orientation $\vO \in \cO$ and an element $e \in E$, we define\[\text{sign}(\vO(e)) = \begin{cases}
        \{ 1\} & \text{if } \vO(e)>0, \\
        \{-1\} & \text{if } \vO(e)<0.\\
    \end{cases}\]
We will use the following notation many times.  
\begin{definition}\label{def: from map to fourientation}
For a basis $B\subseteq E\cup E^*$ and an orientation $\vO\in\cO$, we define the fourientation $\vF(B,\vO)$ whose value at each $e$ is determined by

\[\vF(B,\vO)(e) = \begin{cases}
        \{\pm 1\} & \text{if } e\in B, \\
        \text{sign}(\vO(\underline{e})) & \text{if } e\notin B, \\
    \end{cases}\]where $\underline{e}$ is the unique element in $E\cap\{e, e^*\}$. Intuitively, we bi-orient the edges in $B$ and orient the edges not in $B$ according to $\vO$.
\end{definition}

We now give the key definition in this section. 
\begin{definition}\label{def: triangulating}
Let $\beta\colon\cB(M)\to\cO$ be a map from the set of bases of $M$ to the set of orientations of $E$. 
\begin{enumerate}[label=\rm(\roman*)]
\item If the fourientation \[\vF(B_1,\beta(B_1))\cap (-\vF(B_2,\beta(B_2)))\] contains no signed circuit for any two bases $B_1$ and $B_2$, then we call the map $\beta$ \emph{triangulating}.

\item From the map $\beta$, we obtain the induced map
\begin{align*}
  \bar{\beta} \colon \cB(M) &\to \cR(\cC)\\
  B &\mapsto \bR{\beta(B)}.
\end{align*}

\end{enumerate}
\end{definition}

\begin{remark}\label{rem: atlas}
The notion of triangulating maps comes from triangulating atlases of regular matroids in \cite{Ding2024}.   
\end{remark}

The following is the main theorem of this section.
\begin{theorem}\label{thm: triangulating}
If $\beta\colon\cB(M)\to\cO$ is triangulating, then $\bar{\beta}$ is a bijection. 
\end{theorem}

\begin{proof}
    By Corollary~\ref{cor: count the circuit reversal classes}, it suffices to show that $\bar{\beta}$ is injective.

    Assume, for the sake of contradiction, that $\beta(B_1)\sim_\mathrm{R} \beta(B_2)$ for two different bases $B_1$ and $B_2$. By Proposition~\ref{prop: R-equivalence}, we have \[\beta(B_1) - \beta(B_2)=\sum_{k=1}^m \pi(\vC_k),\] where $\vC_1,\ldots,\vC_m$ are signed circuits 
    with the property that their supports are disjoint and the union of their supports is a subtransversal. Thus every $\vC_k$ is contained in $\vF_1:=\vF(B_1,\beta(B_1))$ and in $-\vF_2:=-\vF(B_2,\beta(B_2))$. If $m\neq 0$, then the fourientation \[\vF:=\vF_1 \cap (-\vF_2)\] contains a signed circuit, which contradicts the assumption that $\beta$ is triangulating. Thus $m=0$, which means that $\beta(B_1) = \beta(B_2)$.

Let $\vO:=\beta(B_1) = \beta(B_2)$. By a direct computation, for any element $e\in E$ we have:

\[\vF(e) = \begin{cases}
        \{\pm 1\} & \text{if } e\in B_1 \cap B_2, \\
        \emptyset & \text{if } e\notin B_1 \cup B_2, \\
        -\text{sign}(\vO(e))& \text{if } e\in B_1\setminus B_2, \\
        \text{sign}(\vO(e)) & \text{if } e\in B_2\setminus B_1, \\
    \end{cases}\]
and
\[\vF(e^*) = \begin{cases}
        \emptyset & \text{if } e\in B_1 \cap B_2, \\
        \{\pm 1\} & \text{if } e\notin B_1 \cup B_2, \\
        \text{sign}(\vO(e)) & \text{if } e\in B_1\setminus B_2, \\
        -\text{sign}(\vO(e)) & \text{if } e\in B_2\setminus B_1. \\
    \end{cases}\]
Thus $\vF$ is negative (Definition~\ref{def: positive fourientation}). Note that there is at least one element $e$ with $|\vF(e)| = |\vF(e^*)| = 1$. Hence, by Lemma~\ref{lem: Farkas2}, there is a signed circuit contained in $\vF$, which contradicts the fact that $\beta$ is triangulating.
\end{proof}

\begin{definition}
When $\beta$ is triangulating, we call $\bar{\beta}$ a BBY bijection. 
\end{definition}

\begin{example}\label{ex: map and fourientation}
We continue with the example described in Figure~\ref{fig:torus}. \begin{enumerate}[label=\rm(\roman*)]
\item One can check that the map $\beta$ in Table~\ref{table: triangulating map} is triangulating. 
For example, denote the first and third bases by $B_1$ and $B_3$, respectively. Then the fourientation
\[\vF(B_1,\beta(B_1))\cap (-\vF(B_3,\beta(B_3))) = \begin{pmatrix}
\{\pm 1\} & \{-1\} & \{-1\} & \emptyset   \\
\emptyset & \{-1\} & \{+1\}   & \{\pm 1\}
\end{pmatrix}\]
does not contain any signed circuit $\vC$ (listed in Example~\ref{ex: circuits}) or its opposite $-\vC$. The induced map $\bar{\beta}$ is a bijection by Theorem~\ref{thm: triangulating}.

\begin{table}[ht]
\centering
\begin{tabular}{|cccc|c|cccc|}
\hline
\multicolumn{4}{|c|}{Basis $B$}  &  $2\beta(B)$    & \multicolumn{4}{|c|}{$\vF(B,\beta(B))$} \\
\hline
\color{red} $1$ &   &  &  & \multirow{2}{*}{$(+1,-1,-1,+1)$} & \color{red} $\pm$   & $-$   & $-$   & $+$  \\
&\color{red} $2^*$ &\color{red} $3^*$ &\color{red} $4^*$ &  & $+$ &  \color{red} $\pm$ & \color{red}$\pm$   &  \color{red} $\pm$  \\ \hline
  & \color{red} $2$     &       &       & \multirow{2}{*}{$(-1,-1,-1,-1)$}  &  $-$   & \color{red} $\pm$   &  $-$  & $-$  \\
\color{red}$1^*$  &       &  \color{red} $3^*$    &  \color{red}$4^*$     &     & \color{red} $\pm$    &  $-$   &  \color{red} $\pm$   & \color{red} $\pm$   \\ \hline
 \color{red}$1$ &\color{red} $2$     & \color{red}$3$     &       & \multirow{2}{*}{$(+1,+1,-1,+1)$}   & \color{red} $\pm$  & \color{red} $\pm$  &  \color{red} $\pm$  & $+$   \\
& & & \color{red} $4^*$ &            &   $+$  &  $+$   &   $-$  &  \color{red} $\pm$  \\ \hline
\color{red} $1$ &   \color{red}$2$    &       & \color{red}$4$      & \multirow{2}{*}{$(+1,-1,-1,-1)$}  & \color{red} $\pm$   &  \color{red} $\pm$ & $-$ & \color{red} $\pm$ \\
  &       &      \color{red}$3^*$ &       &                              &  $+$   & $-$  & \color{red} $\pm$  & $-$  \\ \hline
\end{tabular}
\caption{The table shows a map $\beta$ for the example in Figure~\ref{fig:torus}, along with the fourientation $\vF(B,\beta(B))$ for every basis $B$. When the data in a cell are arranged in a $2\times 4$ matrix, the top/bottom positions are indexed by elements in $E$/$E^*$ in order. In fourientations, $\{1\}$ is presented by $+$, etc. By Definition~\ref{def: from map to fourientation}, the positions of the red elements in $B$ match the red elements in $\vF(B,\beta(B))$; the signs of the black elements in $\beta(B)$ match the signs of the black elements in  $\vF(B,\beta(B))$. }
\label{table: triangulating map}
\end{table}
\item  Let $\beta'$ be the map in Table~\ref{table: non-triangulating map}. It is not triangulating, because the fourientation
\[\vF(B_1,\beta'(B_1))\cap (-\vF(B_3,\beta'(B_3)))=
 \begin{pmatrix}
\{\pm 1\} & \{-1\} & \{-1\} & \{+1\}   \\
\emptyset & \{-1\} & \{+1\}   & \{\pm 1\}
\end{pmatrix}
\]contains the signed circuit $\vC=\begin{pmatrix}
0&0&0&1\\0&-1&0&0\\    
\end{pmatrix}$. One can also see that $\beta'$ is not triangulating using Theorem~\ref{thm: triangulating}; indeed, the map $\overline{\beta'}$ is not a bijection since $\overline{\beta'}(B_1)=\overline{\beta'}(B_3)$ (cf. Example~\ref{ex: circuit reversal classes}). More precisely, since $\beta'(B_1)-\beta'(B_3)=(0,-1,0,1)=\pi(\vC)$, %
one can see that $\vC$ is contained in the fourientation. 

\begin{table}[ht]
\centering
\begin{tabular}{|cccc|c|cccc|}
\hline
\multicolumn{4}{|c|}{Basis $B$}  &  $2\beta'(B)$    & \multicolumn{4}{|c|}{$\vF(B,\beta'(B))$} \\
\hline
\color{red} $1$ &   &  &  & \multirow{2}{*}{$(+1,-1,-1,+1)$} & \color{red} $\pm$   & $-$   & $-$   & $+$  \\
&\color{red} $2^*$ &\color{red} $3^*$ &\color{red} $4^*$ &  & $+$ &  \color{red} $\pm$ & \color{red}$\pm$   &  \color{red} $\pm$  \\ \hline
  & \color{red} $2$     &       &       & \multirow{2}{*}{$(-1,-1,-1,-1)$}  &  $-$   & \color{red} $\pm$   &  $-$  & $-$  \\
\color{red}$1^*$  &       &  \color{red} $3^*$    &  \color{red}$4^*$     &     & \color{red} $\pm$    &  $-$   &  \color{red} $\pm$   & \color{red} $\pm$   \\ \hline
 \color{red}$1$ &\color{red} $2$     & \color{red}$3$     &       & \multirow{2}{*}{$(+1,+1,-1,\fbox{-1})$}   & \color{red} $\pm$  & \color{red} $\pm$  &  \color{red} $\pm$  & \fbox {$-$}   \\
& & & \color{red} $4^*$ &            &   $+$  &  $+$   &   $-$  &  \color{red} $\pm$  \\ \hline
\color{red} $1$ &   \color{red}$2$    &       & \color{red}$4$      & \multirow{2}{*}{$(+1,-1,\fbox{+1},-1)$}  & \color{red} $\pm$   &  \color{red} $\pm$ & \fbox{$+$} & \color{red} $\pm$ \\
  &       &      \color{red}$3^*$ &       &                              &  $+$   & $-$  & \color{red} $\pm$  & $-$  \\ \hline
\end{tabular}
\caption{The table shows another map $\beta'$ for the example in Figure~\ref{fig:torus}, along with the fourientation $\vF(B,\beta'(B))$ for every basis $B$. See the caption of Table~\ref{table: triangulating map}. The differences between this table and Table~\ref{table: triangulating map} are highlighted by boxes.}
\label{table: non-triangulating map}
\end{table}

\end{enumerate}

\end{example}

\subsection{Circuit signatures}\label{subsec: signature}
In order to apply Theorem~\ref{thm: triangulating}, one needs a criterion to certify that a given map $\beta\colon\cB\to\cO$ is triangulating.
As in the case of regular matroids \cite{BBY2019}, a convenient tool in this regard is the notion of \emph{acyclic signatures}.

\begin{definition}\label{def: fsigma} 
    A \emph{circuit signature} of $\cC$ is a set of signed circuits which contains exactly one of $\vC$ and $-\vC$ for each signed circuit $\vC$ in $\cC$. 
    A circuit signature~$\sigma$ is said to be \emph{acyclic} if $\sum \lambda_i \vC_i=0$ with $\lambda_i\ge 0$ and $\vC_i\in\sigma$ implies $\lambda_i=0$ for all~$i$.
\end{definition}

\begin{example}\label{ex: acyclic signature}
    Fix a linear ordering of the elements in $E\cup E^*$. Let $\sigma$ be a circuit signature defined as follows.
    For each signed circuit $\vC$, let $e$ be the minimum element in $\mathrm{supp}(\vC)$. If $\vC(e) = 1$, we put $\vC$ in $\sigma$; otherwise, we put $-\vC$ in $\sigma$. Then it is easy to check that $\sigma$ is acyclic.
\end{example}

\begin{definition}\label{def: from signature to map}
For a circuit signature $\sigma$, we define the map \[\beta_\sigma\colon\cB\to\cO\] as follows. 
For every $B \in \cB$ and $e\in E$, let $e'$ be the unique element in $\{e,e^*\}\setminus B$ and let $\vC\in\sigma$ be the unique signed circuit with $\mathrm{supp}(\vC)=\FC(B,e')$. We define
    \begin{align*}
        \beta_\sigma(B)(e) := 
        \begin{cases}
            \frac{1}{2} & \text{if } \vC(e') = 1, \\
            -\frac{1}{2} & \text{if } \vC(e') = -1.
        \end{cases}
    \end{align*}    
\end{definition}

\begin{example}\label{ex: signature to map}
We continue with the example described in Figure~\ref{fig:torus}. Let $\{\ve_i:i\in E\cup E^*\}$ be the standard basis of $\Q^{E\cup E^*}$, where $E=\{1,2,3,4\}$. We will write signed circuits in terms of these basis vectors. \begin{enumerate}[label=\rm(\roman*)]
\item Let \[\sigma=\{\ve_{1^*}-\ve_{2^*}, -\ve_{3}+\ve_{4},\ve_{1^*}-\ve_{3},\ve_{1^*}-\ve_{4},\ve_{2^*}-\ve_{3},-\ve_{2^*}+\ve_{4},-\ve_{1}-\ve_{2}-\ve_{3^*}-\ve_{4^*}\}\]be a circuit signature. One can check that it is acyclic. The map $\beta_\sigma$ is calculated in Table~\ref{table: BBY from acyclic signature}. It equals the map $\beta$ in Example~\ref{ex: map and fourientation}, hence the induced map $\bar{\beta}_\sigma$ is a bijection. 
\begin{table}[ht]
\centering
\begin{tabular}{|c|c|c|c|c|c|}\hline
Basis $B$   & $\vC_1$ & $\vC_2$ & $\vC_3$ & $\vC_4$ & $2\beta_\sigma(B)$  \\ \hline
$12^*3^*4^*$ & ${\color{red}\ve_{1^*}}-\ve_{2^*}$ & $\begin{array}{c}
-\ve_{1}{\color{red}-\ve_{2}}   \\
 -\ve_{3^*}-\ve_{4^*}
\end{array}$   & $\ve_{2^*}{\color{red}-\ve_{3}}$ &  $-\ve_{2^*}{\color{red}+\ve_{4}}$  &  $(+1,-1,-1,+1)$ \\ \hline
$1^*23^*4^*$  & $\begin{array}{c}
{\color{red}-\ve_{1}}-\ve_{2}   \\
 -\ve_{3^*}-\ve_{4^*}
\end{array}$   &  $\ve_{1^*}{\color{red}-\ve_{2^*}}$  & $\ve_{1^*}{\color{red}-\ve_{3}}$   & $\ve_{1^*}{\color{red}-\ve_{4}}$   & $(-1,-1,-1,-1)$\\ \hline
$1234^*$  & ${\color{red}\ve_{1^*}}-\ve_{3}$   & ${\color{red}\ve_{2^*}}-\ve_{3}$   &  $\begin{array}{c}
-\ve_{1}-\ve_{2}   \\
 {\color{red}-\ve_{3^*}}-\ve_{4^*}
\end{array}$   &  $-\ve_{3}{\color{red}+\ve_{4}}$  &  $(+1,+1,-1,+1)$ \\ \hline
$123^*4$&  ${\color{red}\ve_{1^*}}-\ve_{4}$  & ${\color{red}-\ve_{2^*}}+\ve_{4}$   &  ${\color{red}-\ve_{3}}+\ve_{4}$  &  $\begin{array}{c}
-\ve_{1}-\ve_{2}   \\
 -\ve_{3^*}{\color{red}-\ve_{4^*}}
\end{array}$   & $(+1,-1,-1,-1)$ \\ \hline
\end{tabular}
\caption{ The map $\beta_\sigma$ is calculated in the table, where $\sigma$ is defined in Example~\ref{ex: signature to map}. The signed circuit $\vC_e$ is the unique signed circuit in $\sigma$  with $\mathrm{supp}(\vC_e)=\FC(B,e')$, where $e'$ is defined in Definition~\ref{def: from signature to map}. The orientations are determined by the red vectors. }
\label{table: BBY from acyclic signature}
\end{table}
\item Let \[\sigma'=(\sigma\cup\{\ve_{3}-\ve_{4}\})\setminus\{-\ve_{3}+\ve_{4}\}\]be another circuit signature. It is not acyclic because 
\[(\ve_{3}-\ve_{4})+(\ve_{2^*}-\ve_{3})+(-\ve_{2^*}+\ve_{4})=0.\]
The map $\beta_{\sigma'}$ is calculated in Table~\ref{table: map from non-acyclic signature}. It equals the map $\beta'$ in Example~\ref{ex: map and fourientation}. The induced map $\bar{\beta}_{\sigma'}$ is not a bijection. 
\begin{table}[ht]
\centering
\begin{tabular}{|c|c|c|c|c|c|}\hline
Basis $B$   & $\vC_1$ & $\vC_2$ & $\vC_3$ & $\vC_4$ & $2\beta_{\sigma'}(B)$  \\ \hline
$12^*3^*4^*$ & ${\color{red}\ve_{1^*}}-\ve_{2^*}$ & $\begin{array}{c}
-\ve_{1}{\color{red}-\ve_{2}}   \\
 -\ve_{3^*}-\ve_{4^*}
\end{array}$   & $\ve_{2^*}{\color{red}-\ve_{3}}$ &  $-\ve_{2^*}{\color{red}+\ve_{4}}$  &  $(+1,-1,-1,+1)$ \\ \hline
$1^*23^*4^*$  & $\begin{array}{c}
{\color{red}-\ve_{1}}-\ve_{2}   \\
 -\ve_{3^*}-\ve_{4^*}
\end{array}$   &  $\ve_{1^*}{\color{red}-\ve_{2^*}}$  & $\ve_{1^*}{\color{red}-\ve_{3}}$   & $\ve_{1^*}{\color{red}-\ve_{4}}$   & $(-1,-1,-1,-1)$\\ \hline
$1234^*$  & ${\color{red}\ve_{1^*}}-\ve_{3}$   & ${\color{red}\ve_{2^*}}-\ve_{3}$   &  $\begin{array}{c}
-\ve_{1}-\ve_{2}   \\
 {\color{red}-\ve_{3^*}}-\ve_{4^*}
\end{array}$   &  \fbox{$\ve_{3}{\color{red}-\ve_{4}}$}  &  $(+1,+1,-1,\fbox{-1})$ \\ \hline
$123^*4$&  ${\color{red}\ve_{1^*}}-\ve_{4}$  & ${\color{red}-\ve_{2^*}}+\ve_{4}$   &  \fbox{${\color{red}\ve_{3}}-\ve_{4}$}  &  $\begin{array}{c}
-\ve_{1}-\ve_{2}   \\
 -\ve_{3^*}{\color{red}-\ve_{4^*}}
\end{array}$   & $(+1,-1,\fbox{+1},-1)$ \\ \hline
\end{tabular}
\caption{ The map $\beta_{\sigma'}$ is calculated in the table, where $\sigma'$ is defined in Example~\ref{ex: signature to map}. The signed circuit $\vC_e$ is the unique signed circuit in $\sigma'$  with $\mathrm{supp}(\vC_e)=\FC(B,e')$, where $e'$ is defined in Definition~\ref{def: from signature to map}. The orientations are determined by the red vectors. The differences between this table and Table~\ref{table: BBY from acyclic signature} are highlighted by boxes.}
\label{table: map from non-acyclic signature}
\end{table}
\end{enumerate}
\end{example}

The following lemma is an immediate consequence of Definitions~\ref{def: from map to fourientation} and \ref{def: from signature to map}.

\begin{lemma}\label{lem: signature to fourientation}
Let $B$ be a basis and let $\sigma$ be a circuit signature. If a signed circuit $\vC$ is contained in $\vF(B,\beta_\sigma(B))$ and $\mathrm{supp}(\vC)$ is a fundamental circuit with respect to $B$, then $\vC\in\sigma$.    
\end{lemma}

We will use the following property of acyclic circuit signatures to prove that $\beta_\sigma$ is triangulating whenever $\sigma$ is acyclic. 

\begin{lemma}\label{lem: acyclic signature1}
    Let $\sigma$ be an acyclic circuit signature.
    If a signed circuit $\vC$ is contained in $\vF(B,\beta_\sigma(B))$, then $\vC\in\sigma$.
\end{lemma}
\begin{proof}
By Lemma~\ref{lem: sum of fund circuit}, we have
\[
- \vC + \sum_{e\in \mathrm{supp}(\vC)\setminus B} \vC_e = 0,
\]
where $\vC_e$ is the signed circuit such that $\mathrm{supp}(\vC_e)=\FC(B,e)$ and $\vC_e(e) = \vC(e)$. 
Moreover, each summand $\vC_e$ is contained in $\vF(B,\beta_\sigma(B))$, since the elements in $\mathrm{supp}(\vC_e)\setminus\{e\}$ are in $B$ and hence are bioriented in $\vF(B,\beta_\sigma(B))$. 
By Lemma~\ref{lem: signature to fourientation}, every summand $\vC_e $ is in $\sigma$. Because $\sigma$ is acyclic, we have $-\vC \notin \sigma$ and hence $\vC \in \sigma$. 
\end{proof}

\begin{lemma}\label{lem: acyclic signature2}
If $\sigma$ is acyclic, then $\beta_\sigma$ is triangulating.   
\end{lemma}
\begin{proof}
    Suppose, for the sake of contradiction, that there are two distinct bases $B_1$ and $B_2$ such that $\vF(B_1,\beta_\sigma(B_1))$ and $-\vF(B_2,\beta_\sigma(B_2))$ contain a common signed circuit, say $\vC$.
    By Lemma~\ref{lem: acyclic signature1}, we have both $\vC\in \sigma$ and $-\vC\in \sigma$, a contradiction.
\end{proof}

From Theorem~\ref{thm: triangulating} and Lemma~\ref{lem: acyclic signature2}, we obtain: 
\begin{corollary}\label{cor: acyclic to BBY}
If $\sigma$ is acyclic, then the map $\bar{\beta}_\sigma$ is a BBY bijection.     
\end{corollary}

\begin{remark}
For any regular orthogonal representation $\cC$, Example~\ref{ex: acyclic signature} shows that the set of acyclic signatures for $\cC$, and hence the set of BBY bijections for $\cC$, is non-empty.
\end{remark}

\begin{remark}
    For regular matroids, acyclic circuit signatures have an equivalent definition \cite[Lemma~3.1.1]{BBY2019}. There is a similar result for regular orthogonal representations $\cC$: a circuit signature $\sigma$ of $\cC$ is acyclic if and only if there exists a vector $\vw\in\mathbb{R}^{E\cup E^*}$ such that $\langle\vC,\vw\rangle>0$ for all $\cC\in\sigma$, where $\langle,\rangle$ denotes the standard inner product. 
    We omit the proof, which is similar to the regular matroid case, since this result will not be used in the rest of this paper. 

\end{remark}

\subsection{The BBY torsors}\label{subsec: BBY torsor}

In Section~\ref{sec: cycle reversal}, we have shown that the natural group action of $\J(\cC)$ on $\cR(\cC)$ is simply transitive. 
And in Section~\ref{subsec: BBY bijection},
we have shown that the map $\bar{\beta}$ from $\cB$ to $\cR(\cC)$ is a bijection whenever $\beta$ is triangulating. 

By composing the action and the bijection, we obtain, for every triangulating map $\beta : \cB \to \cO$, a simply transitive action $\Gamma_\beta$ of $\J(\cC)$ on $\cB$ defined by\begin{align*}
\Gamma_\beta\colon\J(\cC) \times \cB &\to \cB \\
(\bJ{\vv},B) &\mapsto \bar{\beta}^{-1}(\bJ{\vv} \cdot \bar{\beta}(B)).
\end{align*}

This makes the set $\cB$ of bases of $M$ into a torsor for the group $\J(\cC)$.

\begin{definition}\label{def: BBY torsor}    
The torsor described above is called a \emph{BBY torsor}. 
\end{definition}

Using Definition~\ref{def: action} and Definition~\ref{def: BBY torsor}, we have the following equivalent ways to describe a BBY torsor. 

\begin{lemma}\label{lem: from basis-action to o-action} Suppose $\beta$ is a triangulating map. For $B_1, B_2\in\cB$ and $\bJ{\vv}\in\J(\cC)$, we have
\begin{align*} 
& \Gamma_\beta(\bJ{\vv}, B_1) = B_2  \\ 
\Leftrightarrow\; & \bJ{\vv} \cdot\bar{\beta}(B_1) = \bar{\beta}(B_2)\\
\Leftrightarrow\; & \bJ{\vv+\beta(B_1)} = \bJ{\beta(B_2)}.
\end{align*}
   
\end{lemma}

Our next goal is to show that the BBY bijections and BBY torsors are (in a suitable sense) invariant under reorientation.

\begin{definition}
Let $\beta\colon\cB\to\cO$ be a map and let $\sigma$ be a circuit signature for $\cC$. Let $X\subseteq E$. We define 
\begin{align*}
  \beta^X \colon \cB &\to \cO\\
  B &\mapsto \beta(B)^X
\end{align*}
and
\[\sigma^X:=\{\vC^X:\vC\in\sigma\}.\]
\end{definition}
It is straightforward to prove the following results. 
\begin{lemma}\label{lem: BBY cut-switching}
Let $\beta\colon\cB\to\cO$ be a map and let $\sigma$ be a circuit signature. Then:
\begin{enumerate}[label=\rm(\roman*)]
\item $\beta$ is triangulating if and only if $\beta^X$ is triangulating. 
\item $\sigma$ is acyclic if and only if $\sigma^X$ is acyclic. 
\item $\bar{\beta}(B)=\bR{\vO}\Leftrightarrow\overline{\beta^X}(B)=\bR{\vO}^X$. 
\item $\beta_{\sigma^X}=\beta_{\sigma}^X$ (i.e., $\beta_\sigma(B)=\vO\Leftrightarrow\beta_{\sigma^X}(B)=\vO^X$).
\item $\bar{\beta}_\sigma(B)=\bR{\vO}\Leftrightarrow\bar{\beta}_{\sigma^X}(B)=\bR{\vO}^X$. 
\end{enumerate}
\end{lemma}

The BBY torsor does not depend on reorientation in the following sense. 
\begin{theorem}\label{thm: BBY up to cut-switching}
Let $B_1, B_2\in\cB$ and let $\bJ{\vv}\in\J(\cC)$. 
\begin{enumerate}[label=\rm(\roman*)]
\item If $\beta$ is a triangulating map, then
\[\Gamma_\beta(\bJ{\vv}, B_1) = B_2 \Leftrightarrow\Gamma_{\beta^X}(\bJ{\vv}^X, B_1) = B_2.\]
\item If $\sigma$ is an acyclic circuit signature, then
\[\Gamma_{\beta_\sigma}(\bJ{\vv}, B_1) = B_2 \Leftrightarrow\Gamma_{\beta_{\sigma^X}}(\bJ{\vv}^X, B_1) = B_2.\]
\end{enumerate}
\end{theorem}
\begin{proof}
By Lemma~\ref{lem: from basis-action to o-action} and Lemma~\ref{lem: BBY cut-switching}, we have\begin{align*} 
& \Gamma_\beta(\bJ{\vv}, B_1) = B_2  \\ 
\Leftrightarrow\; & \bJ{\vv} \cdot\bar{\beta}(B_1) = \bar{\beta}(B_2)\\
\Leftrightarrow\; & \bJ{\vv}^X \cdot\overline{\beta^X}(B_1) = \overline{\beta^X}(B_2)\\
\Leftrightarrow\; &\Gamma_{\beta^X}(\bJ{\vv}^X, B_1) = B_2. 
\end{align*}
This proves (i), and (ii) is a direct consequence of (i). 
\end{proof}

\section{The canonical torsor for a ribbon graph}
\label{sec:ribbon torsor}

In the previous section, we have defined BBY torsors for regular orthogonal representations of regular orthogonal matroids. In this section, we apply the theory to ribbon graphs $\vG$. We obtain a canonical BBY torsor associated to $\vG$ by defining a natural class $\{ \sigma_p \}$ of acyclic signatures and showing that they all give rise to the same torsor. 
We then define Bernardi-style bijections and Bernardi torsors for $\vG$, and we show that they coincide with the BBY bijections and BBY torsors for the acyclic signatures $\sigma_p$. In particular, the Bernardi torsor is also canonical. 

\subsection{The Jacobian group and ribbon cycle reversal system of a ribbon graph}
Let $\vG$ be a ribbon graph. From $\vG$, we obtain the orthogonal matroid $M(\vG)$ whose bases are in bijection with the spanning quasi-trees of $\vG$ (cf. Lemma~\ref{lem:ribbon basis}). Now choose a reference orientation $\vec{G}$ of the set $E$ of edges of $\vG$, and assign the dual reference orientation to $E^*$ by the ``right-hand rule''. From these data, we obtain the regular orthogonal representation $\cC(\vG,\vec{G})$ of $M(\vG)$ (cf. Definition~\ref{def: ribbon to om}).

\begin{definition}\label{def: ribbon Jacobian}
Let $\vG$ be a ribbon graph with a reference orientation $\vec{G}$. 
\begin{enumerate}[label=\rm(\roman*)]
    \item The \emph{Jacobian group}\footnote{Our notion of the Jacobian group is isomorphic to the {\em critical group} of $\vG$ as defined by Merino, Moffatt, and Noble \cite{MMN2023}; see Appendix~\ref{subsec:bouchet} for a proof. } of $\vG$ is defined to be
\[\J(\vG):=\J(\cC(\vG,\vec{G})).\]
    \item The \emph{ribbon cycle reversal system} of $\vG$ is defined to be
\[\cR(\vG):=\cR(\cC(\vG,\vec{G})).\]
    \item There is a simply transitive group action of $\J(\vG)$ on $\cR(\vG)$, defined as in Definition~\ref{def: action}. 
\end{enumerate}
\end{definition}

The objects defined in Definition~\ref{def: ribbon Jacobian} are independent of the choice of reference orientation $\vec{G}$, in the following sense. To each edge $e$ (resp. coedge $e^*$) of $\vG$ we associate two opposite directed edges (resp. coedges). The standard basis vectors $\ve$ of $\Z^E=\oplus_{e\in E}\Z \ve$ are interpreted as directed edges in $\vec{G}$, and the opposite vectors $-\ve$ are interpreted as the opposite directed edges. The standard basis vectors $\ve^*$ of $\Z^{E^*}=\oplus_{e^*\in E}\Z \ve^*$ are interpreted similarly. A signed ribbon cycle $\vC\in\cC(\vG,\vec{G})$ is interpreted as the set of directed edges and coedges whose sum is $\vC$. Then a different choice of $\vec{G}$ yields a reorientation $\cC(\vG,\vec{G})^X$ of $\cC(\vG,\vec{G})$, where $X$ is the set of edges where the two reference orientations differ (cf. Remark~\ref{rem: reference orientation}). Note that the signed ribbon cycles $\vC\in\cC(\vG,\vec{G})$ and $\vC^X\in\cC(\vG,\vec{G})^X$ are interpreted as the same set of directed edges and coedges. By Lemma~\ref{lem: reorientation and Jac}, we may identify $\bJ{\ve}\in\J(\cC(\vG,\vec{G}))$ with $\bJ{\ve^X}\in\J(\cC(\vG,\vec{G})^X)$, and the two vectors $\ve$ and $\ve^X$ are interpreted as the same directed edge of $\vG$. In this sense, we may identify the Jacobian groups obtained from different reference orientations (and, in particular, these groups are isomorphic to one another).

Similarly, we may identify the ribbon cycle reversal systems $\cR(\cC(\vG,\vec{G}))$ obtained from different reference orientations $\vec{G}$ using Lemma~\ref{lem: reorientation and circuit reversal}. In particular, the circuit reversals $\vO+\pi(\vC)$ and $\vO^X+\pi(\vC^X)$ have the same geometric interpretation: reverse the directions of the edges $e\in\mathrm{supp}(\pi(\vC))$ in $\vO$. By Lemma~\ref{lem: reorientation and action}, the group action of $\J(\cC(\vG,\vec{G}))$ on $\cR(\cC(\vG,\vec{G}))$ is independent of the choice of $\vec{G}$. 

The discussion above leads to a ``geometric'' description of the group action, without referring to the reference orientation $\vec{G}$ (cf. \cite[Section 4.3]{BBY2019}). More precisely, the action of $\J(\vG)$ on $\cR(\vG)$ can be defined by linearly extending the following action of each generator $\bJ{\ve}$ of $\J(\vG)$ on an orientation class $\bR{\vO}$ of $\cR(\vG)$: 
\begin{enumerate}[label=\rm(\roman*)]
\item if the directed edge $\ve$ is not in $\vO$, then reverse the orientation of $e$ in $\vO$ to obtain $\bJ{\ve}\cdot\bR{\vO}$; 
\item otherwise (using Lemma~\ref{lem: sign}), find a signed ribbon cycle compatible with $\vO$, reverse the cycle in $\vO$, and then apply (i).  
\end{enumerate}

It is straightforward to check that this definition is equivalent to Definition~\ref{def: action} for ribbon graphs.

\subsection{The canonical BBY torsor for a ribbon graph}
\label{subsec:ribbon bby}

Recall that the ribbon graph $\vG$ is embedded into a connected closed orientable surface $\Sigma$.

\begin{definition}\label{def: p sign}
Fix a point $p\in \Sigma\setminus(E\cup E^*)$. Let $\sigma_p$ be the circuit signature where every signed circuit $\vC$ is oriented according to the component of $\Sigma\setminus \mathrm{supp}(\vC)$ containing $p$ (cf. Definition~\ref{def: ribbon to om}). 
\end{definition}

\begin{example}\label{ex: point to signature}
Let $p$ be the point shown in Figure~\ref{fig:p on torus}. Then the circuit signature $\sigma_p$ is equal to the acyclic signature $\sigma$ from Example~\ref{ex: signature to map}.

\begin{figure}
    \centering
    \includegraphics[width=0.50\linewidth]{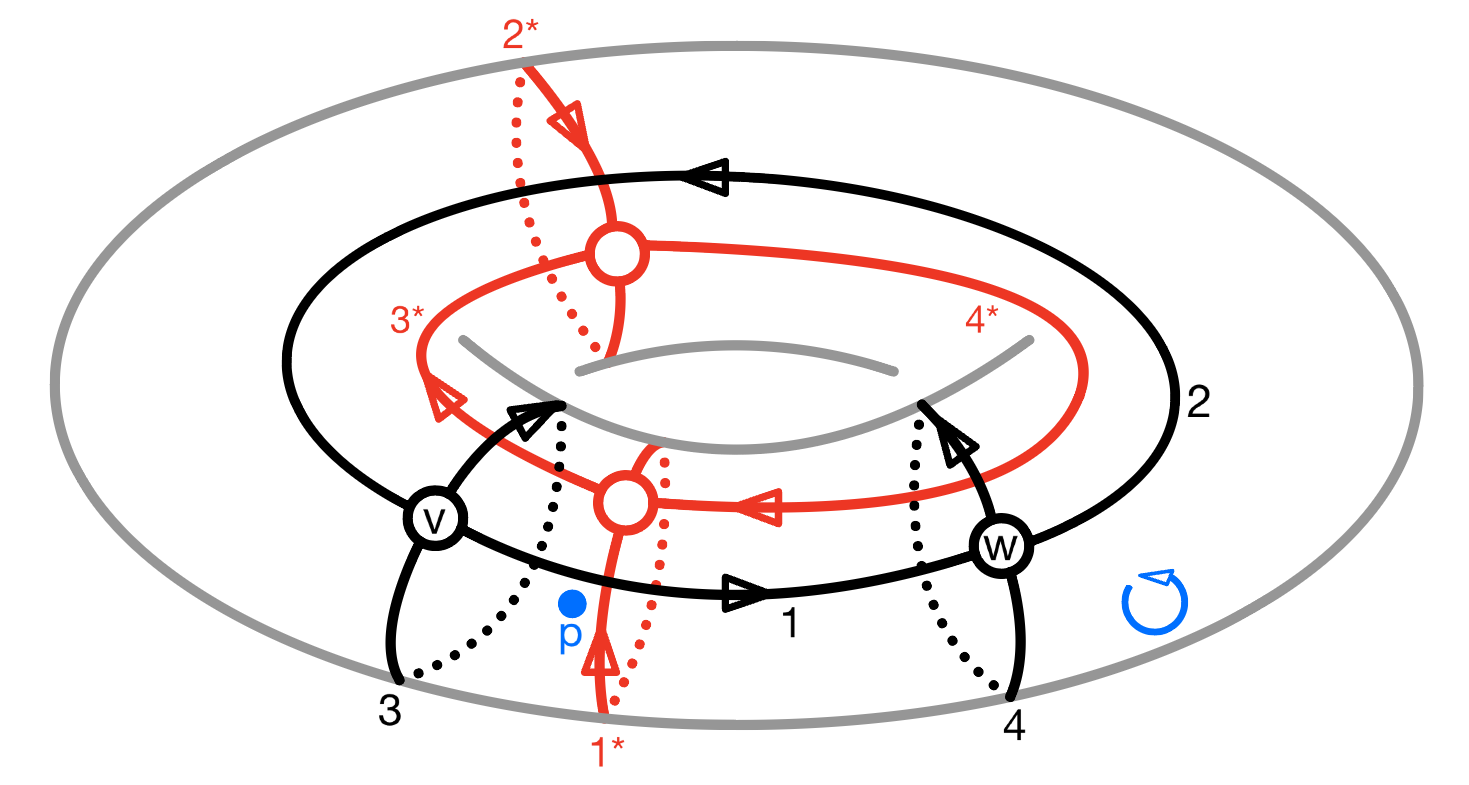}
    \caption{The point $p$ is near the intersection of the edge $1$ and the coedge $1^*$, under the edge $1$ and to the left of the coedge $1^*$. %
    The black and red arrows indicate the reference orientations of the underlying graph and its dual.}
    \label{fig:p on torus}
\end{figure}
    
\end{example}

\begin{proposition}\label{prop: p-signature}
The circuit signature $\sigma_p$ is acyclic.     
\end{proposition}
\begin{proof}
Give $\Sigma$ a smooth structure. Let $\omega$ be a volume form on $\Sigma$, which exists because $\Sigma$ is orientable. Then $\omega$ will also be a volume form on $\Sigma\setminus\{p\}$. It is a well-known fact that the top de Rham cohomology of a noncompact connected orientable smooth manifold vanishes; see \cite[Theorem 17.32]{Lee}. Hence $H^2_{\text{dR}}(\Sigma\setminus\{p\})=0$, which implies that we can write the closed $2$-form $\omega$ on $\Sigma\setminus\{p\}$ as an exact form $\omega=\text{d} \alpha$. For any circuit $C$ of $M(\vG)$, let $A_C$ be the connected component of $\Sigma\setminus C$ \emph{not} containing $p$. By Stokes' theorem, we have
\[\int_{A_C} \omega=\int_{-\vC} \alpha=-\int_{\vC} \alpha.\]
Now suppose \[\sum\lambda_i\vC_i=0,\]
where $\lambda_i\geq 0$ and $\vC_i\in \sigma$ for all $i$. Then we have 
\[\sum\lambda_i\int_{A_{C_i}} \omega=-\sum\lambda_i\int_{\vC_i} \alpha=-\int_{\sum\lambda_i\vC_i} \alpha=0.\]
Because $\int_{A_{C_i}} \omega$ is always positive, all the coefficients $\lambda_i$ vanish. Therefore the circuit signature $\sigma_p$ is acyclic.  

\end{proof}

For simplicity, we denote \[\beta_p:=\beta_{\sigma_p}.\]
By Proposition~\ref{prop: p-signature} and Corollary~\ref{cor: acyclic to BBY}, the map $\bar{\beta}_p$ is a BBY bijection between the set $\cQ(\vG)$ of spanning quasi-trees of $\vG$ and the set $\cR(\vG)$ of ribbon cycle reversal system of $\vG$. Hence we obtain a simply transitive group action $\Gamma_p$ of the Jacobian group $\J(\vG)$ on $\cQ(\vG)$, given by the formula (cf. Section~\ref{subsec: BBY torsor})
\begin{align*}
\Gamma_p\colon\J(\vG) \times \cQ(\vG) &\to \cQ(\vG) \\
(\bJ{\vv},B) &\mapsto \bar{\beta}_p^{-1}(\bJ{\vv} \cdot \bar{\beta}_p(B)),
\end{align*}
where we identify a spanning quasi-tree $Q\in\mathcal{Q}$ with the basis $B=Q\cup (E\setminus Q)^*$. 

\begin{corollary}\label{cor: ribbon action indep of reference orientation}
The action $\Gamma_p$ is independent of the choice of a reference orientation $\vec{G}$. 
\end{corollary}
\begin{proof}
This is a direct consequence of Theorem~\ref{thm: BBY up to cut-switching} and Remark~\ref{rem: reference orientation}. 
\end{proof}

Although the bijection $\beta_p$ depends on the choice of $p$ in general, we will show that the group action $\Gamma_p$ does not. 
\begin{theorem}\label{thm: independent p}
The action $\Gamma_p$ is independent of $p$. 
\end{theorem}

By Corollary~\ref{cor: ribbon action indep of reference orientation} and Theorem~\ref{thm: independent p}, the action $\Gamma := \Gamma_p$ depends only on the ribbon graph $\vG$ and our convention for how to form a dual orientation. In this sense, we say that the action $\Gamma$ is \emph{canonical}.   

To prove the theorem, we start with a lemma. This lemma follows from Bouchet's results on circle graphs~\cite{Bouchet1987c} and multimatroids~\cite{Bouchet2001}, but we give a direct proof here for completeness.

\begin{lemma}\label{lem: 2 fundamental circuits}
Let $B$ be a basis and let $e_1, e_2$ be distinct elements of $E \cup E^*\setminus B$. For $i\in\{1,2\}$, let $\vC_i$ be a signed circuit such that $\mathrm{supp}(\vC_i)=\FC(B, e_i)$. Then we have
\[\vC_1(e_1)\vC_2(e_1^*)+\vC_1(e_2^*)\vC_2(e_2)=0.\]
In particular, $e_1^*\notin\mathrm{supp}(\vC_2)$ implies $\vC_1(e_2)+\vC_1(e_2^*)=0$. 
\end{lemma}
\begin{proof}
Denote $C_i=\mathrm{supp}(\vC_i)$. 
Consider\[\langle \vC_1, \vC_2\rangle = \sum_{e\in E\cup E^*} \vC_1(e)\vC_2(e^*) = 0.\] 
By the definition of fundamental circuits, we have $e_i\in C_i$, $e_i^*\notin C_i$, and $C_i\setminus \{e_i\}\subseteq B$. Thus
\begin{itemize}
    \item when $e\notin\{e_1, e_1^*, e_2, e_2^*\}$, at least one of $C_1$ and $C_2^*$ does not contain $e$, and hence $\vC_1(e)\vC_2(e^*)=0$;
    \item $\vC_1(e_1^*)=\vC_1(e_2)=0$;
    \item $\vC_2(e_2)\neq 0$. 
\end{itemize}
Therefore the only possibly non-zero terms in the sum are $\vC_1(e_1)\vC_2(e_1^*)$ and $\vC_1(e_2^*)\vC_2(e_2)$. Then we get the first desired equality. 

When $e_1^*\notin C_2$, we have $\vC_1(e_1)\vC_2(e_1^*)=0$ and hence $\vC_1(e_2^*)\vC_2(e_2)=0$. Since $\vC_2(e_2)\neq 0$, we get $\vC_1(e_2^*)=0$. Then we get $\vC_1(e_2)+\vC_1(e_2^*)=0$.
\end{proof}

For a given edge $e\in E$, let \[\delta_e:E\to\mathbb{Z}\] be the function that sends $e$ to $1$ and the other edges to $0$.

\begin{lemma}\label{lem: independent p}
    Let $p$ and $q$ be points in different regions of $\Sigma \setminus (E\cup E^*)$ such that $p$ is in the third quadrant and $q$ is in the second quadrant of the $e-e^*$ plane, where $e\in E$; see Figure~\ref{fig: pqrs}. 
    Then 
    \begin{align*}
        \beta_p(B) - \beta_q(B)
        =
        \begin{cases}
            -\delta_e & \text{if } e\notin B, \\
            \pi(\vC)-\delta_e & \text{if } e\in B,
        \end{cases}
    \end{align*}
    where $\vC$ is the signed circuit such that $\mathrm{supp}(\vC) = \FC(B,e^*)$ and $\vC(e^*) = 1$.
\end{lemma}
\begin{proof}
    We compare $\beta_p(B)(x)$ with $\beta_q(B)(x)$ for every $x\in E$. Let $x'$ be the unique element in $\{x,x^*\}\setminus B$.  Recall that the value of $\beta_p(B)(x)$
    is determined by the component of $\Sigma\setminus \FC(B,x')$ containing $p$.  

    We first consider the case $e\notin B$. When $x=e$, we have $x'=e$. Hence we have $\beta_p(B)(e) = -\frac{1}{2}$ and $\beta_q(B)(e) = \frac{1}{2}$. When $x\ne e$, the points $p$ and $q$ are in the same component of $\Sigma\setminus \FC(B,x')$ because one can move from $p$ to $q$ crossing the edge $e\notin \FC(B,x')$. Thus $\beta_p(B)(x) = \beta_q(B)(x)$. Therefore, we have $\beta_p(B) - \beta_q(B) = -\delta_e$ in this case. 

    Next we consider the case where $e\in B$. 
    \begin{itemize}
        \item When $x=e$, we have $x'=e^*$. Because the points $p$ and $q$ are in the same component of $\Sigma\setminus \FC(B,e^*)$, we get $\beta_p(B)(e) = \beta_q(B)(e)$. Since $\pi(\vC)(e)=\vC(e)+\vC(e^*)=1$, the formula $\beta_p(B) - \beta_q(B)=\pi(\vC)-\delta_e$ holds at $x=e$.
        \item When $x\ne e$, there are two cases: \begin{itemize}
            \item If $e\notin \FC(B,x')$, then $p$ and $q$ are in the same component of $\Sigma\setminus \FC(B,x')$ and thus $\beta_p(B)(x) = \beta_q(B)(x)$. By applying Lemma~\ref{lem: 2 fundamental circuits} to $\vC_1=\vC$, $\vC_2=$ a signed circuit with support $\FC(B,x')$, $e_1=e^*$, and $e_2=x'$, we get $\pi(\vC)(x)=0$. Thus the formula $\beta_p(B) - \beta_q(B)=\pi(\vC)-\delta_e$ holds in this case. 
            \item If $e\in \FC(B,x')$,
            then $p$ and $q$ are in different components of $\Sigma\setminus \FC(B,x')$, which implies $\beta_p(B)(x) \ne \beta_q(B)(x)$. 
            Let $\vD$ be the signed circuit such that $\mathrm{supp}(\vD)=\FC(B,x')$ and $\vD(e)=-1$.
            Then $\vD\in\sigma_p$, and we have \[\beta_p(B)(x) = \frac{\vD(x')}{2} \text{ and }\beta_q(B)(x) = -\frac{\vD(x')}{2}.\] Hence we get
            \[\beta_p(B)(x)- \beta_q(B)(x) = \vD(x').\]
            By applying Lemma~\ref{lem: 2 fundamental circuits} to $\vC_1=\vC$ and $\vC_2=\vD$, we have\[\vC(e^*)\vD(e)+\vC((x')^*)\vD(x')=0.\] Thus we have $\vC((x')^*)\vD(x')=1$, which implies $\vC((x')^*) = \vD(x')$. It follows that\[\pi(\vC)(x)=\vD(x').\]
            Then we get \[\beta_p(B)(x) - \beta_q(B)(x) = \pi(\vC)(x) =\pi(\vC)(x) -\delta_e(x).\]
            \end{itemize}
        Therefore, $\beta_p(B) - \beta_q(B) = \pi(\vC) - \delta_e$ also holds at $x\neq e$. \qedhere
    \end{itemize}

\end{proof}

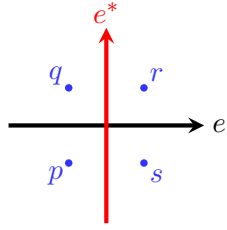
\begin{figure}
    \centering
    \begin{tikzpicture}
        \draw[ultra thick, -stealth] (-1.3,0) -- (1.3,0);
        \draw[ultra thick, -stealth, color=red] (0,-1.3) -- (0,1.3);
        \node at (1.5,0) {$e$};
        \node[color=red] at (0,1.5) {$e^*$};
        
        \node (p) at (-0.5,-0.5)[circle,fill,color=blue!80,inner sep=1pt] {};
        \node[color=blue!80] at (-0.67,-0.67) {$p$};
        
        \node (q) at (-0.5,0.5)[circle,fill,color=blue!80,inner sep=1pt] {};
        \node[color=blue!80] at (-0.67,0.67) {$q$};
        
        \node (r) at (0.5,0.5)[circle,fill,color=blue!80,inner sep=1pt] {};
        \node[color=blue!80] at (0.67,0.67) {$r$};
        
        \node (s) at (0.5,-0.5)[circle,fill,color=blue!80,inner sep=1pt] {};
        \node[color=blue!80] at (0.67,-0.67) {$s$};
    \end{tikzpicture}
    \caption{The points $p$, $q$, $r$, and $s$ defined in Lemma~\ref{lem: independent p} and Lemma~\ref{lem: independent p full version}. }
    \label{fig: pqrs}
\end{figure}

Using the same method, we may prove the following result. 

\begin{lemma}\label{lem: independent p full version}
    Let $p$, $q$, $r$, and $s$ be four points in different regions of $\Sigma \setminus (E\cup E^*)$ such that $p$ is in the third quadrant, $q$ is in the second quadrant, $r$ is in the first quadrant, and $s$ is in the fourth quadrant of the $e-e^*$ plane; see Figure~\ref{fig: pqrs}. Let $\vC$ be the signed circuit such that $\mathrm{supp}(\vC) = \FC(B,e^*)$ and $\vC(e^*) = 1$. Let $\vD$ be the signed circuit such that $\mathrm{supp}(\vD) = \FC(B,e)$ and $\vD(e) = 1$.
    Then 
    \begin{align*}
        \beta_s(B) - \beta_r(B)=\beta_p(B) - \beta_q(B)
        =
        \begin{cases}
            -\delta_e & \text{if } e\notin B, \\
            \pi(\vC)-\delta_e & \text{if } e\in B,
        \end{cases}
    \end{align*}
    \begin{align*}
        \beta_s(B) - \beta_p(B)=\beta_r(B) - \beta_q(B)
        =
        \begin{cases}
            -\delta_e & \text{if } e\in B, \\
            \pi(\vD)-\delta_e & \text{if } e\notin B.
        \end{cases}
    \end{align*}

\end{lemma}

\begin{corollary}\label{cor: equal actions}
Let $p$, $q$, $r$, and $s$ be the four points in Lemma~\ref{lem: independent p full version}. Then the four actions $\Gamma_p$, $\Gamma_q$, $\Gamma_r$, and $\Gamma_s$ are equal.     
\end{corollary}
\begin{proof}
By Lemma~\ref{lem: from basis-action to o-action}, we have\[\Gamma_p(\bJ{\vv}, B_1) = B_2 \Leftrightarrow \bJ{\vv+\beta_p(B_1)} = \bJ{\beta_p(B_2)},\]and\[\Gamma_q(\bJ{\vv}, B_1) = B_2 \Leftrightarrow \bJ{\vv+\beta_q(B_1)} = \bJ{\beta_q(B_2)}.\]
By Lemma~\ref{lem: independent p full version}, we have \[\bJ{\vv+\beta_p(B_1)}-\bJ{\vv+\beta_q(B_1)}=\bJ{-\delta_e}= \bJ{\beta_p(B_2)}-\bJ{\beta_q(B_2)},\]
and hence
\[\Gamma_p(\bJ{\vv}, B_1) = B_2 \Leftrightarrow \Gamma_q(\bJ{\vv}, B_1) = B_2.\]
This means $\Gamma_p=\Gamma_q$. Similarly, we have $\Gamma_s=\Gamma_r$ and $\Gamma_s=\Gamma_p$.
\end{proof}

\begin{proof}[Proof of Theorem~\ref{thm: independent p}]
For any two points $p_1$ and $p_2$ in $\Sigma\setminus (E\cup E^*)$, one can connect them using a path $\gamma$ avoiding $V\cup V^*$. Each time the path crosses an edge or a coedge, the action remains the same by Corollary~\ref{cor: equal actions}. 
\end{proof}

\subsection{The Bernardi bijections for ribbon graphs}
\label{subsec:ribbon bernardi}

We extend the Bernardi bijections on graphs \cite{BW2018,Bernardi2008} to ribbon graphs, and we show that they can be identified with a special subclass of the BBY bijections introduced in Section~\ref{subsec:ribbon bby}.

Let $\vG = (G,\Sigma)$ be a ribbon graph.
In this subsection, we do not assume a fixed reference orientation of $G$ {\em a priori}.
This approach ensures that our result is as general as possible while maintaining compatibility with the original result by Bernardi~\cite{Bernardi2008}.

We denote by $v_e$ the unique point in the intersection of an edge $e$ and the corresponding coedge $e^*$,
so that $v_e$ divides $e$ (resp. $e^*$) into two half-edges (resp. half-coedges).
Given a half-edge $h$ associated with an edge $e$, let $h^*$ be the half-coedge that immediately follows $h$ in the counter-clockwise ordering among the four half-(co)edges incident with $v_e$; see Figure~\ref{fig: half-edge}. 
We also define $(h^*)^* := h$, so that $*$ is an involution on the set of half-edges and half-coedges. 
In particular, the operation $*$ rotates half-coedges clockwise.

Let $N_\epsilon$ be a small $\epsilon$-neighborhood of a spanning quasi-tree $Q \subseteq E$. 
For $e\in E\cup E^*$, the boundary $\partial N_\epsilon$ intersects $e$ if and only if $e \in Q^* \cup (E\setminus Q)$.
Whenever $\partial N_\epsilon$ intersects $e$, it traverses each of the corresponding half-(co)edges exactly once.
Consequently, we obtain a cyclic ordering, denoted $Q^\circ$, on the $2|E|$ half-(co)edges that intersect $\partial N_\epsilon$ by enumerating them according to the induced orientation on $\partial N_\epsilon$.
See Example~\ref{ex: bernardi}.

We consider a pair $(h_1,h_2)$ of half-edges associated with an edge $e$ as a directed edge $\vec{e}$ with tail $h_1$ and head $h_2$.

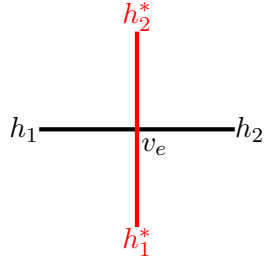
\begin{figure}
    \centering
    \begin{tikzpicture}
        \draw[ultra thick] (-1.3,0) -- (1.3,0);
        \draw[ultra thick, color=red] (0,-1.3) -- (0,1.3);
        \node at (0.23,-0.23) {$v_e$};
        \node at (-1.5,0) {$h_1$};
        \node at (1.5,0) {$h_2$};
        \node[color=red] at (0,-1.5) {$h_1^*$};
        \node[color=red] at (0,1.5) {$h_2^*$};
    \end{tikzpicture}
    \caption{Four half-edges and half-coedges associated with an edge $e$}
    \label{fig: half-edge}
\end{figure}

\begin{definition}\label{def: Bernardi tour}
    Let $h$ be a half-edge and let $Q \subseteq E$ be a spanning quasi-tree.
    Then we obtain a linear ordering $h_0 < h_1 < \ldots < h_{2|E|-1}$ from the cyclic ordering $Q^\circ$ by fixing the least element $h_0$ to be either $h$ or $h^*$ (whichever one appears in $Q^\circ$).
    Let $\gamma_{h}(Q)$ be the orientation of $G$ defined as follows.
    \begin{itemize}
        \item If $e\in Q$, then we orient $e$ by
    \begin{align*}
        \vec{e} = (h_i^*, h_j^*),
    \end{align*}
    where $h_i < h_j$ are the two half-coedges of $e$ appearing in $Q^\circ$.
    \item If $e\notin Q$, then we orient $e$ by
    \begin{align*}
        \vec{e} = (h_j, h_i),
    \end{align*}
     where $h_i < h_j$ are the two half-edges of $e$ appearing in $Q^\circ$.
    \end{itemize}

\end{definition}

\begin{remark}
    If $T$ is a spanning tree of $G$, then $\gamma_h(T)$ is equal to the orientation $\cO_T$ in~\cite{Bernardi2008}.
\end{remark}

\begin{remark}
    For a less formal but more intuitive defintion of $\gamma_h(Q)$, see \S\ref{subsec:Bernardi-intro}.
\end{remark}

\begin{example}\label{ex: bernardi}
    We demonstrate the above process applied to Figure~\ref{fig:torus}.
    Let $h$ be one of the half-edges of $1$ incident with $v$, and let $Q = \{1\}$ be a spanning quasi-tree.
    Then the boundary of a small $\epsilon$-neighborhood of $Q$ intersects $h^*$.
    Fixing the least element $h_0=h^*$, the linear ordering associated with $Q^\circ$ is 
    \begin{align*}
        h^* = h_0 < h_1 < h_2 < h_3 < h_4 < h_5 < h_6 < h_7,
    \end{align*}
    as depicted in Figure~\ref{fig: bernardi}(left).
    Then the orientation $\gamma_h(Q)$ of $G$ is given by $\vec{1} = (h_0^*, h_4^*)$, $\vec{2} = (h_6, h_2)$, $\vec{3} = (h_7, h_5)$, and $\vec{4} = (h_3, h_1)$.
    See Figure~\ref{fig: bernardi}(right). 
\end{example}

\begin{figure}[ht]
    \centering    %
    \begin{tikzpicture}
        \begin{scope}
            \draw[line width=0.6mm, color=gray] (-2.4,-2.4) rectangle (2.4,2.4);
            \draw[line width=0.6mm, color=gray, ->>] (-2.4,0) -- (-2.4,1.8);
            \draw[line width=0.6mm, color=gray, ->>] (2.4,0) -- (2.4,1.8);
            \draw[line width=0.6mm, color=gray, ->] (0,-2.4) -- (0.4,-2.4);
            \draw[line width=0.6mm, color=gray, ->] (0,2.4) -- (0.4,2.4);

            \draw[line width=0.6mm, color=black] 
                (-2.4,0) -- (2.4,0)
                (-0.8,-2.4) -- (-0.8,2.4)
                (0.8,-2.4) -- (0.8,2.4);

            \draw[line width=0.6mm, color=cyan!80!black]
                (-0.8,0) -- (0,0);
            \node[color=cyan!80!black] at (-0.3,0.15) {\tiny $h$};

            \fill[color=cyan!80!black] (-0.3,-0.18) circle (0.05);
            \node[color=cyan!80!black] at (-0.16,-0.23) {\tiny $p$};

            \draw[line width=0.6mm, color=red]
                (-2.4,1.0) -- (2.4,1.0)
                (-1.6,-2.4) -- (-1.6,2.4)
                (0,-2.4) -- (0,2.4);

            \draw[line width=0.6mm, color=black, fill=white]
                (-0.8,0) circle (0.23)
                (0.8,0) circle (0.23);
            \node at (-0.8,0) {$v$};
            \node at (0.8,0) {$w$};

            \draw[line width=0.6mm, color=red, fill=white]
                (-1.6,1.0) circle (0.23)
                (0,1.0) circle (0.23);

            \draw[line width=0.6mm, color=green!80!black]
                (-0.95,-0.45) -- (0.95,-0.45)
                (-0.95,0.45) -- (0.95,0.45)
                (-0.95,-0.45) arc (270:90:0.45)
                (0.95,0.45) arc (90:-90:0.45);
            \draw[line width=0.6mm, color=green!80!black, ->]
                (-0.95,-0.45) -- (0.58,-0.45);

            \node[color=green!80!black] at (0.25,-0.7) {\small $h_0$};
            \node[color=green!80!black] at (1.05,-0.7) {\small $h_1$};
            \node[color=green!80!black] at (1.65,0.21) {\small $h_2$};
            \node[color=green!80!black] at (1.05,0.65) {\small $h_3$};
            \node[color=green!80!black] at (0.25,0.65) {\small $h_4$};
            \node[color=green!80!black] at (-0.55,0.65) {\small $h_5$};
            \node[color=green!80!black] at (-1.6,0.21) {\small $h_6$};
            \node[color=green!80!black] at (-0.55,-0.7) {\small $h_7$};

        \end{scope}
        \begin{scope}[xshift=7cm]
            \draw[line width=0.6mm, color=gray] (-2.4,-2.4) rectangle (2.4,2.4);
            \draw[line width=0.6mm, color=gray, ->>] (-2.4,0) -- (-2.4,1.8);
            \draw[line width=0.6mm, color=gray, ->>] (2.4,0) -- (2.4,1.8);
            \draw[line width=0.6mm, color=gray, ->] (0,-2.4) -- (0.4,-2.4);
            \draw[line width=0.6mm, color=gray, ->] (0,2.4) -- (0.4,2.4);

            \draw[line width=0.6mm, color=black] 
                (-2.4,0) -- (2.4,0)
                (-0.8,-2.4) -- (-0.8,2.4)
                (0.8,-2.4) -- (0.8,2.4);
            \draw[line width=0.6mm, color=black, -stealth] (0.2,0) -- (0.4,0); 
            \draw[line width=0.6mm, color=black, -stealth] (1.8,0) -- (1.57,0); 
            \draw[line width=0.6mm, color=black, -stealth] (-0.8,1.75) -- (-0.8,1.55);
            \draw[line width=0.6mm, color=black, -stealth] (0.8,1.55) -- (0.8,1.75);

            \draw[line width=0.6mm, color=cyan!80!black]
                (-0.8,0) -- (0,0);
            \node[color=cyan!80!black] at (-0.3,0.15) {\tiny $h$};

            \fill[color=cyan!80!black] (-0.3,-0.18) circle (0.05);
            \node[color=cyan!80!black] at (-0.16,-0.23) {\tiny $p$};

            \draw[line width=0.6mm, color=red]
                (-2.4,1.0) -- (2.4,1.0)
                (-1.6,-2.4) -- (-1.6,2.4)
                (0,-2.4) -- (0,2.4);

            \draw[line width=0.6mm, color=black, fill=white]
                (-0.8,0) circle (0.23)
                (0.8,0) circle (0.23);
            \node at (-0.8,0) {$v$};
            \node at (0.8,0) {$w$};

            \draw[line width=0.6mm, color=red, fill=white]
                (-1.6,1.0) circle (0.23)
                (0,1.0) circle (0.23);

            \draw[line width=0.6mm, color=green!80!black]
                (-0.95,-0.45) -- (0.95,-0.45)
                (-0.95,0.45) -- (0.95,0.45)
                (-0.95,-0.45) arc (270:90:0.45)
                (0.95,0.45) arc (90:-90:0.45);
            \draw[line width=0.6mm, color=green!80!black, ->]
                (-0.95,-0.45) -- (0.58,-0.45);
                
            \node[color=black] at (0.33,-0.23) {$1$};
            \node[color=black] at (1.9,-0.23) {$2$};
            \node[color=black] at (-2.1,-0.23) {$2$};
            \node[color=black] at (-0.6,-1.4) {$3$};
            \node[color=black] at (-0.6,1.45) {$3$};
            \node[color=black] at (1.0,-1.4) {$4$};
            \node[color=black] at (1.0,1.45) {$4$};
            
            \node[color=red] at (0.27,-1.4) {$1^*$};
            \node[color=red] at (-1.6+0.27,-1.4) {$2^*$};
            \node[color=red] at (-0.43,1.0-0.23) {$3^*$};
            \node[color=red] at (1.9,1.0-0.23) {$4^*$};
        \end{scope}
    \end{tikzpicture}
    \caption{The orientation $\gamma_h(Q)$ of $G$ in Examples~\ref{ex: bernardi} and~\ref{ex: bernardi2}. The green curve represents the boundary of a small $\epsilon$-neighborhood of $Q=\{1\}$.}
    \label{fig: bernardi}
\end{figure}
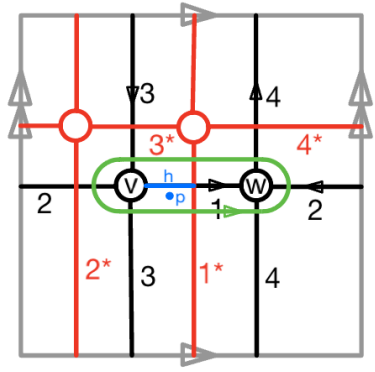

Now fix a reference orientation $\vec{G}$,
and for a half-edge $h$, let $\beta_h'(-,\vec{G}) : \cB(\vG) \to \cO = \{\pm\frac{1}{2}\}^E$ be the map defined by
\begin{align*}
    \beta_h'(B,\vec{G})(e) = 
    \begin{cases}
        +\frac{1}{2} & \text{if the directions of $e$ are the same in $\vec{G}$ and $\gamma_h(E\cap B)$}, \\
        -\frac{1}{2} & \text{otherwise}.
    \end{cases}
\end{align*}

Let $p$ be the point of $\Sigma \setminus (G\cup G^*)$ which is located at distance $\epsilon/2$ from the midpoint of $h$ in the direction of $h^*$; 
see Figure~\ref{fig: bernardi}. %
We denote by $\beta_p(-,\vec{G})$ the corresponding BBY bijection.

\begin{theorem}\label{thm:bernardi and bby}
    The two maps $\beta_h'(-,\vec{G})$ and $\beta_p(-,\vec{G})$ are identical.
\end{theorem}

\begin{example}\label{ex: bernardi2}
    We continue Example~\ref{ex: bernardi}.
    Let $\vec{G}$ be the reference orientation given in Figure~\ref{fig:orientation}.
    The edges that have different directions in $\gamma_h(Q)$ and $\vec{G}$ are $2$ and $3$.
    Therefore, $\beta'_h(12^*3^*4^*,\vec{G}) = \left( \frac{1}{2}, -\frac{1}{2}, -\frac{1}{2}, \frac{1}{2} \right)$, which is equal to $\beta_p(12^*3^*4^*,\vec{G})$ by Example~\ref{ex: signature to map}.
\end{example}

\begin{lemma}\label{lem:bernardi all positive}
    All entries in $\beta_p(B , \gamma_h(Q))$ are $1/2$, where $Q = E\cap B$.
\end{lemma}
\begin{proof}
    Let $N$ be an $\epsilon/2$-neighborhood of $Q$ in $\Sigma$, so that $p$ is in $\partial N$. %
    For each $e\in E$, we claim that $\beta_p(B , \gamma_h(Q))(e) = 1/2$.
    We denote by $h_1$ and $h_2$ the half-edges of $e$ such that the direction of $e$ in $\gamma_h(Q)$ is $(h_1,h_2)$.

    \textbf{Case I.}
    $e\in Q$.
    Then, as we move along the boundary $\partial N$ (with its induced orientation) away from $p$, we cross $h_1^*$ first and $h_2^*$ second.
    Let $p'$ be a point in $\partial N$ just before we cross $h_1^*$. 
    We denote by $\Sigma'$ the component of $\Sigma \setminus \FC(B, e^*)$ with $p\in \Sigma'$.
    Then $p'\in \Sigma'$ as well, and hence the directed coedge $(h_1^*,h_2^*)$ agrees with the induced orientation on $\partial \Sigma'$.
    Therefore, $\beta_p(B , \gamma_h(Q))(e) = \beta_{p'}(B , \gamma_h(Q))(e) = 1/2$.
    
    \textbf{Case II.}
    $e\in E\setminus Q$.
    Then, as we move along the boundary $\partial N$ away from $p$, we cross $h_2$ first and $h_1$ second.
    Let $p'$ be a point in $\partial N_\epsilon$ just before we cross $h_2$.
    We denote by $\Sigma'$ the component of $\Sigma \setminus \FC(B, e)$ with $p\in \Sigma'$.
    Then $p'\in \Sigma'$ as well, and hence the directed edge $(h_1,h_2)$ 
    agrees with the induced orientation on $\partial \Sigma'$.
    Therefore, $\beta_p(B , \gamma_h(Q))(e) = \beta_{p'}(B , \gamma_h(Q))(e) = 1/2$. 
\end{proof}

\begin{proof}[Proof of Theorem~\ref{thm:bernardi and bby}]
    Let $B$ be a basis of $\vG$ and set $Q:= E\cap B$.
    By definition, $\beta_p(-,\vec{G})(e) = \beta_p(-,\gamma_h(Q))(e)$ if and only if the directions of the edge $e$ in $\vec{G}$ and $\gamma_h(Q)$ are the same.
    The latter condition is equivalent to  $\beta_h'(-,\vec{G})(e) = \beta_h'(-,\gamma_h(Q))(e)$.
    By definition and Lemma~\ref{lem:bernardi all positive}, $\beta_h'(B,\gamma_h(Q)) = (1/2,\ldots,1/2) = \beta_p(B, \gamma_h(Q))$.
    Therefore, $\beta_h'(B,\vec{G}) = \beta_p(B, \vec{G})$.
\end{proof}

\section{Complements}\label{sec: complement}
\subsection{Orientation conventions}\label{sec: opposite convention}
Our theory for ribbon graphs $\vG=(G,\Sigma)$ has been built on the convention that the orientation of the surface $\Sigma$ is \emph{counter-clockwise} and the dual reference orientation is assigned using the \emph{right-hand rule}. 
If we instead use
the \emph{clockwise} orientation of $\Sigma$ and the \emph{left-hand rule} for the dual orientation, this affects the definition of $\sigma_p$ (cf. Definition~\ref{def: p sign}) and the cyclic ordering used to define the Bernardi map $\beta_h'$ (cf. Definition~\ref{def: Bernardi tour}). With these alternative conventions, Theorem~\ref{thm:bernardi and bby} can be stated as
\[\mathfrak{b}_h'(-,\vec{G})=\mathfrak{b}_p(-,\vec{G}),\]
where the Bernardi map $\mathfrak{b}_h'$ and the BBY map $\mathfrak{b}_p$ are defined using the clockwise orientation of $\Sigma$ and the left-hand rule, and $\vec{G}$ is a reference orientation on $E$ independent of the choice of convention.

\subsection{Duality}
We explain how the canonical BBY torsor defined in Section~\ref{sec: BBY} (or equivalently the Bernardi torsor defined in Section~\ref{sec:ribbon torsor}) is compatible with duality:

Given an orthogonal matroid $M=(E\cup E^*, \cB)$, its \emph{dual orthogonal matroid} is defined to be\[M^*=(E\cup E^*, \cB^*),\] 
where\[\cB^*=\{B^*:B\in\cB\}.\]
This generalizes the notion of dual matroids because for any matroid $N$, we have $\mathrm{lift}(N)^*=\mathrm{lift}(N^*)$. Clearly, $(M^*)^*=M$. 

\begin{definition}
Let $\cC$ be a regular orthogonal representation of $M$. Then its \emph{dual} is the regular orthogonal representation\[\cC^*:=\{\vC^*:\vC\in\cC\}\] of $M^*$. 
\end{definition}
It is straightforward to check that the set $\cC^*$ satisfies Definition~\ref{def: regular representation}. 
Because for any $\vC\in\cC$ we have $\pi(\vC)=\pi(\vC^*)$, the \emph{identity map} induces \begin{itemize}
    \item the group isomorphism $\J(\cC)\cong\J(\cC^*)$,
    \item the bijection $\cR(\cC)\to\cR(\cC^*)$, and
    \item the ``isomorphic actions'' $\J(\cC)\circlearrowright\cR(\cC)$ and $\J(\cC^*)\circlearrowright\cR(\cC^*)$.
\end{itemize}
We will identify the two objects in each pair. 

Now we consider the BBY bijections. Given an acyclic circuit signature $\sigma$ of $M$, the set \[\sigma^*:=\{\vC^*:\vC\in\sigma\}\] is clearly an acyclic circuit signature of $M^*$. Note that for any basis $B$ of $M$ and any $e\notin B$, we have $\FC(B^*,e^*)=\FC(B,e)^*$. Therefore, we have
\begin{equation}\label{eq: dual BBY}
\beta_\sigma(B)=\beta_{\sigma^*}(B^*).    
\end{equation}

Next we study the dual of a ribbon graph $\vG$ with edge set $E$. As defined in Section~\ref{subsection: ribbon}, the dual ribbon graph $\vG^*$ has edge set $E^*$, but in this subsection we interchange $E$ and $E^*$ so that we can consider
spanning quasi-trees of $\vG^*$ as subsets of $E$. Clearly, we have:
\begin{itemize}
    \item $M(\vG)^*=M(\vG^*)$,
    \item a subset $Q$ of $E$ is a spanning quasi-tree of $\vG$ if and only if $E\setminus Q$ is a spanning quasi-tree of $\vG^*$, and
    \item a subset $C$ of $E\cup E^*$ is a circuit of $M(\vG)$ if and only if $C^*$ is a circuit of $M(\vG^*)$.
\end{itemize}
We remark that a circuit $C$ of $M(\vG)$ is geometrically the same as the circuit of $C^*$ of $M(\vG^*)$; see Figure~\ref{fig: dual} for example. 

\begin{figure}
    \centering
    \includegraphics[width=0.8\linewidth]{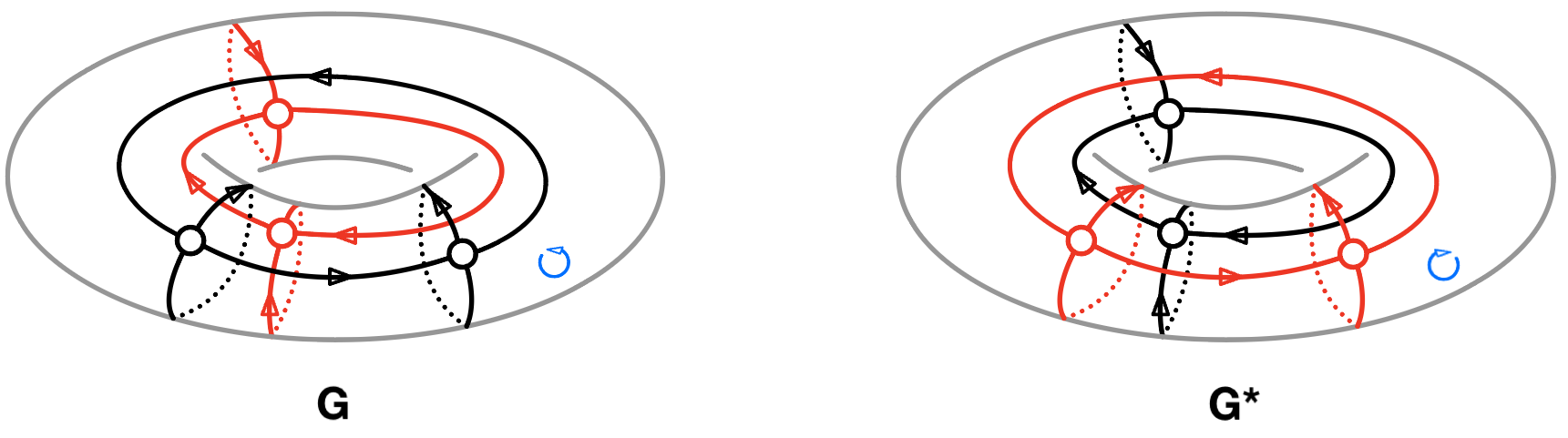}
    \caption{A ribbon graph $\vG$ and its dual $\vG^*$, where the edges are colored black and the coedges are colored red in both $\vG$ and $\vG^*$.
    Each circuit $C$ of $M(\vG)$ in the left figure corresponds to $C^*$ in the right figure, obtained by switching the colors black and red. Moreover, if $\vG$ is given with a reference orientation $\vec{G}$ following the standard convention (the right-hand rule and the counter-clockwise orientation), then this orientation naturally corresponds to the orientation $\vec{G}^*$ for $\vG^*$ with the opposite convention.}
    \label{fig: dual}
\end{figure}

Recall that the regular orthogonal representation $\cC(\vG, \vec{G})$ is obtained from the ribbon graph $\vG$, a reference orientation~$\vec{G}$ of $E$, and the dual reference orientation $\vec{G}^*$ of $E^*$ induced by the right-hand rule. To make our torsors compatible with duality, when we define the regular orthogonal representation $\cC(\vG^*, \vec{G}^*)$ for $M(\vG^*)$, we use the \emph{left-hand} rule (and the clockwise orientation of $\Sigma$), where $\vec{G}^*$ is the reference orientation of $E^*$ for $\vG$ and hence a reference orientation of $E$ for~$\vG^*$; see Figure~\ref{fig: dual}. An immediate consequence is that any signed ribbon cycle $\vC$ of $\vG$ is the same geometric object and the same algebraic object as the signed ribbon cycle $\vC^*$ of $\vG^*$. In particular, we have the identity\[\cC(\vG^*, \vec{G}^*)=\cC(\vG, \vec{G})^*\]and the isomorphism\[\J(\vG)\cong\J(\vG^*)\]
induced by the identity map from $\Z^E$ to $\Z^E$. 

Fix a point $p\in \Sigma\setminus(E\cup E^*)$. By Equation~(\ref{eq: dual BBY}), we get\[
\beta_{\sigma_p}(B)=\beta_{\sigma_p^*}(B^*),\]
where $\sigma_p$ is defined in Definition~\ref{def: p sign} for the ribbon graph $\vG$. Because we use the left-hand rule for $\vG^*$ along with the clockwise surface orientation, we have 
\[
\beta_{\sigma_p^*}(B^*)=-\mathfrak{b}_{p}(B^*,\vec{G}^*).\]
Applying Theorem~\ref{thm:bernardi and bby} to $\vG$ and its left-hand analogue (cf. Section~\ref{sec: opposite convention}) to $\vG^*$, we get
\[\beta_h'(B,\vec{G})=\beta_p(B,\vec{G})(=\beta_{\sigma_p}(B))\]
and
\[\mathfrak{b}_h'(B^*,\vec{G}^*)=\mathfrak{b}_p(B^*,\vec{G}^*).\]The four identities above imply
\begin{equation}\label{eq: dual relation}
\beta_h'(B,\vec{G})=\beta_p(B,\vec{G})=-\mathfrak{b}_{p}(B^*,\vec{G}^*)=-\mathfrak{b}_h'(B^*,\vec{G}^*).    
\end{equation}

Recall that for the ribbon graph $\vG$ (and the right-hand rule), we have the canonical group action
\[\Gamma_p\colon\J(\vG) \times \cQ(\vG) \to \cQ(\vG)\]
determined by the relation
\[\Gamma_p(\bJ{\vv},B_1)=B_2 \Leftrightarrow \bJ{\vv}\cdot\bar{\beta}_p(B_1,\vec{G})=\bar{\beta}_p(B_2,\vec{G}).\]
For the dual ribbon graph $\vG^*$ (and the left-hand rule), the canonical group action\[\Gamma_p^*\colon\J(\vG^*) \times \cQ(\vG^*) \to \cQ(\vG^*)\]is
determined by the relation
\[\Gamma_p^*(\bJ{\vv},B_1^*)=B_2^* \Leftrightarrow \bJ{\vv}\cdot\bar{\mathfrak{b}}_p(B_1^*,\vec{G}^*)=\bar{\mathfrak{b}}_p(B_2^*,\vec{G}^*).\]

By Equation~(\ref{eq: dual relation}), the two group actions satisfy the following property, which means that the canonical torsor $\Gamma_p$ is compatible with duality. 
\begin{theorem}
$\Gamma_p(\bJ{\vv},B_1)=B_2 \Leftrightarrow \Gamma_p^*(-\bJ{\vv},B_1^*)=B_2^*$.
\end{theorem}
This generalizes \cite[Theorem 6.1]{BW2018} which shows that the Bernardi torsor for a plane graph is compatible with duality.

\subsection{Connections to regular matroids}\label{sec: om to m}
We explain how our definitions and results generalize the corresponding ones for regular matroids.

Recall that given a matroid $N = (E,\cB)$, the pair $\mathrm{lift}(N):=(E\cup E^*, \cB')$ is an orthogonal matroid, where $\cB' = \{B\cup (E\setminus B)^* : B\in \cB\}$.
The set of circuits of $\mathrm{lift}(N)$ is the union of $\{C\subseteq E: \text{$C$ is a circuit of $N$}\}$ and $\{D^*\subseteq E^*: \text{$D$ is a cocircuit of $N$}\}$. 

Let $N$ be a regular matroid. Choosing a regular representation of $N$ (i.e., a totally unimodular integer matrix representing $N$), we obtain a set $\cC$ of signed circuits of $N$ and a set $\cD$ of signed cocircuits of $N$, where $\cC,\cD\subseteq\{0,\pm 1\}^E$ and every signed circuit is orthogonal to every signed cocircuit with respect to the standard inner product on $\Z^E$. See \cite[Section 2.1]{BBY2019} for details. The sets $\cC$ and $\cD$ are unique up to reorientations; this result is essentially due to \cite{Camion}.

For every element $\vv\in\cC$ (resp. $\vv\in\cD$), let $\phi(\vv)\in\{0,\pm 1\}^{E\cup E^*}$ be the unique vector such that $\mathrm{supp}(\phi(\vv))\subseteq E$ (resp. $\mathrm{supp}(\phi(\vv))\subseteq E^*$) and $\pi(\phi(\vv))=\vv$. 
Denote\[\cC\cD:=\{\phi(\vv):\vv\in\cC\cup\cD\}\subseteq\{0,\pm 1\}^{E\cup E^*}.\]
The set $\cC\cD$ is a regular orthogonal representation of $\mathrm{lift}(N)$, as one can easily check.

Recall that the \emph{Jacobian group} of $N$ is defined to be \[\J(N):=\Z^E/\langle\cC\cup\cD\rangle.\]      
\begin{proposition}
The Jacobian group of the regular orthogonal representation $\cC\cD$ is naturally isomorphic to the Jacobian group of $N$. 
\end{proposition}
\begin{proof}
Because $\langle\pi(\cC\cD)\rangle=\langle\cC\cup\cD\rangle$, the identity map on $\Z^E$ induces the desired isomorphism.     
\end{proof}

We recall the circuit-cocircuit reversal system for regular matroids, defined in \cite{Gioan08}, using our language. The set of orientations is $\{\pm \frac{1}{2}\}^E$. A \emph{circuit reversal} (resp. a \emph{cocircuit reversal}) of the regular matroid $N$ turns an orientation $\vO$ into the orientation $\vO+\vC$ for some $\vC\in\cC$ (resp. for some $\vC\in\cD$). Two orientations $\vO_1$ and $\vO_2$ of $N$ are said to be \emph{equivalent} if one orientation can be obtained from the other by a sequence of circuit reversals and cocircuit reversals. This defines an equivalence relation on the set of orientations. We call an equivalence class for this relation a \emph{circuit-cocircuit reversal class}. The set of circuit-cocircuit reversal classes of $N$ is denoted by $\cR(N)$. 
Because a circuit reversal or a cocircuit reversal $\vO+\vC$ of the regular matroid $N$ is the same as a circuit reversal $\vO+\pi(\phi(\vC))$ of the regular orthogonal representation $\cC\cD$, we have the following result:

\begin{proposition}
$\cR(\cC\cD)=\cR(N)$. 
\end{proposition}

Similarly, one can check that the following definitions for the regular orthogonal representation $\cC\cD$ are essentially the same as their counterparts for regular matroids. We omit the proofs for brevity. 

\begin{proposition}\;
\begin{enumerate}[label=\rm(\roman*)]
\item The action of $\J(\cC\cD)$ on $\cR(\cC\cD)$ in Definition~\ref{def: action} is the same\footnote{Technically speaking, we use a different convention. The action $\bJ{\vv}\cdot \bR{\vO}$ in our paper is the same as the action $\bJ{-\vv}\cdot \bR{\vO}$ in \cite{BBY2019}.} as the action of $\J(N)$ on $\cR(N)$ in \cite[Section 4.3]{BBY2019}.
\item The map $\beta$ for $\cC\cD$ is triangulating (cf. Definition~\ref{def: triangulating}) if and only if $\beta=f_{\mathcal{A}, \mathcal{A}^*}$ for a pair $(\mathcal{A}, \mathcal{A}^*)$ of triangulating atlases in the sense of \cite[Section 1.4]{Ding2024}. 
\item A circuit signature $\sigma$ of $\cC\cD$ is acyclic if and only if the set \[\sigma_\cC:=\{\vC\in \cC:\phi(\vC)\in\sigma\}\] is an acyclic circuit signature of $N$ and the set \[\sigma_\cD:=\{\vD\in \cD:\phi(\vD)\in\sigma\}\] is an acyclic cocircuit signature of $N$, in the sense of 
\cite[Section 1]{BBY2019}.  \item When the circuit signature $\sigma$ of $\cC\cD$ is acyclic, the BBY bijection $\bar{\beta}_\sigma\colon\cB'(\mathrm{lift}(N))\to\cR(\cC\cD)$ is the same as the BBY bijection induced by a pair $(\sigma_\cC,\sigma_\cD)$ of acyclic signatures defined in \cite[Theorem 1.2.2]{BBY2019}, where we identify each basis $B$ of $N$ with the basis $B\cup (E\setminus B)^*$ of $\mathrm{lift}(N)$. 
\end{enumerate}
\end{proposition}

\section*{Acknowledgements}
Thanks to Nathan Bowler, Matt Larson, Langte Ma, Iain Moffatt, and Wei Wang for helpful discussions, and to Chi Ho Yuen for feedback on the first draft of this manuscript. This work began at the 2024 Workshop on (Mostly) Matroids at the Institute for Basic Science, Daejeon, South Korea, in August 2024, and we would like to thank IBS as well as the organizers of that worshop (S. Backman,
R. Campbell, D. Mayhew, and S. Oum) for providing a stimulating environment. 
The first author was supported by NSF grant DMS-2154224 and a Simons Fellowship in Mathematics.
The third author was supported by the Institute for Basic Science (IBS-R029-C1).

\bibliographystyle{plainurl}
\bibliography{ref.bib}

\appendix

\section{Bouchet's representations of ribbon graphs}\label{subsec:bouchet}

In this appendix, we review Bouchet's construction~\cite{Bouchet1987c} of principally unimodular matrices associated with ribbon graphs,
with the ultimate goal of proving that the Jacobian group $\J(\vG)$, as defined in this paper, coincides with the critical group of $\vG$, as defined in~\cite{MMN2023}. %

Let $\vG = (G,\Sigma)$ be a ribbon graph, and fix a reference orientation of $G$.
We denote by $v_e$ the intersection of an edge $e$ and the corresponding coedge $e^*$.
The point $v_e$ divides $e$ into the two half-edges $e_{+}$ and $e_{-}$, where we adopt the convention that $e_{+}$ corresponds to the head of $e$ and $e_{-}$ corresponds to the tail.
Similarly, we define the two half-coedges $e^*_{+}$ and $e^*_{-}$.

We define a matrix $A(\vG,Q)$ as follows:

\begin{enumerate}[label=\rm(\roman*)]
    \item Select a spanning quasi-tree $(V,Q)$ of $\vG$.
    We denote by $q$ the boundary of a small $\epsilon$-neighborhood of $(V,Q)$ in $\Sigma$.
    Then, as in \S\ref{subsec:ribbon bernardi}, $q$ intersects $2|E|$ half-edges and half-coedges induced by $Q^* \cup (E\setminus Q)$.
    \item As in \S\ref{subsec:ribbon bernardi}, travel along $q$ according to the induced orientation and list the $2|E|$ half-edges and half-coedges you cross, in order; this yields a cyclic ordering which we denote by $q^\circ$.
    \item Form the skew-symmetric matrix $A(\vG,Q)$, whose rows and columns are indexed by the elements in $Q^* \cup (E\setminus Q)$, using the rule: 
    \begin{align*}
        A(\vG,Q)(e,f)
        =
        \begin{cases}
            +1 & \text{if we have $\cdots f_{+} \cdots e_{+} \cdots f_{-} \cdots e_{-} \cdots $ in $q^\circ$}, \\
            -1 & \text{if we have $\cdots e_{+} \cdots f_{+} \cdots e_{-} \cdots f_{-} \cdots $ in $q^\circ$}, \\
            0 & \text{otherwise}. \\
        \end{cases}
    \end{align*}
\end{enumerate}

\begin{theorem}[\cite{Bouchet1987c}; see~\cite{MMN2023}]
    $A(\vG,Q)$ is principally unimodular.
\end{theorem}

The skew-symmetric matrix $A(\vG,Q)$ encodes the signed ribbon cycles $\cC(\vG)$ of $\vG$, in the following sense.
Let $\Lambda(\vG,Q)$ be the $n\times 2n$ matrix $\begin{pmatrix}
    I & A(\vG,Q)
\end{pmatrix}$
such that the columns of the identity matrix $I$ are indexed by elements in $Q\cup (E\setminus Q)^*$ and $I(e,e^*) = 1$ for each $e\in Q^*\cup (E\setminus Q)$.

\begin{proposition}\label{prop:compatibility}
    For each $e\in Q^*\cup (E\setminus Q)$, the $e$-th row of $\Lambda(\vG,Q)$ is equal to $\vC^*_e$, where $\vC_e$ is the signed ribbon cycle such that $\mathrm{supp}(\vC_e) = \FC(Q\cup (E\setminus Q)^*, e)$ and $\vC_e(e) = 1$.
\end{proposition}

\begin{proof}
    Let $c_1$ and $c_2$ be the boundary components of a small $\epsilon$-neighborhood of $(V,Q\triangle \{ \underline{e}\})$, where $\underline{e}\in E\cap \{e,e^*\}$.
    Each $c_i$ inherits its orientation from the induced boundary orientation.
    We may assume that $(c_1,e^*) = 1$.
    Then $(c_2,e^*) = -1$, and $\iota(c_1) = -\iota(c_2) = \vC_e^*$ by Lemma~\ref{lem:vC and fudamental circuit}.
    Let $c_i^\circ$ be the associated cyclic ordering of the half-edges and half-coedges intersecting $c_i$.

    \textbf{Case I.}
    $e^*\in Q$.
    Then $c_1$ intersects $e^*_+$ and $c_2$ intersects $e^*_-$.
    Hence, $c_1^\circ$ and $c_2^\circ$ can be enumerated as $e^*_+ X$ and $e^*_- Y$, respectively, where $X$ and $Y$ are subsequences of the cyclic ordering $q^\circ$ with $q^\circ = e_- X e_+ Y$; see Figure~\ref{fig: bouchet}(left).
    For $e\ne f \in Q^*\cup (E\setminus Q)$, we observe that 
    \begin{itemize}
        \item $(c_1,f)=1$ if and only if $f_+ \in X$ and $f_- \in Y$, and 
        \item $(c_1,f)=-1$ if and only if $f_- \in X$ and $f_+ \in X$.
    \end{itemize}
    Therefore, $(c_1,f) = A(\vG,Q)(e,f)$, implying that $\iota(c_1) = \vC_e^*$ is equal to the $e$-th row of $\Lambda(\vG,Q)$.

    \textbf{Case II.}
    $e^* \in (E\setminus Q)^*$.
    Then $c_1$ intersects $e^*_-$ and $c_2$ intersects $e^*_+$.
    Thus, $c_1^\circ$ and $c_2^\circ$ can be written as $e^*_- X$ and $e^*_+ Y$, respectively, where $X$ and $Y$ are subsequences of the cyclic ordering $q^\circ$ with $q^\circ = e_+ X e_- Y$; see Figure~\ref{fig: bouchet}(right).
    We have:
    \begin{itemize}
        \item $(c_1,f)=1$ if and only if $f_- \in X$ and $f_+ \in Y$, and 
        \item $(c_1,f)=-1$ if and only if $f_+ \in X$ and $f_- \in X$.
    \end{itemize}
    Therefore, $(c_1,f) = A(\vG,Q)(e,f)$, implying that $\vC_e^*$ is equal to the $e$-th row of $\Lambda(\vG,Q)$.
\end{proof}

\begin{figure}
    \centering
    \begin{tikzpicture}
        \node at (0,0) {\includegraphics[scale=0.30]{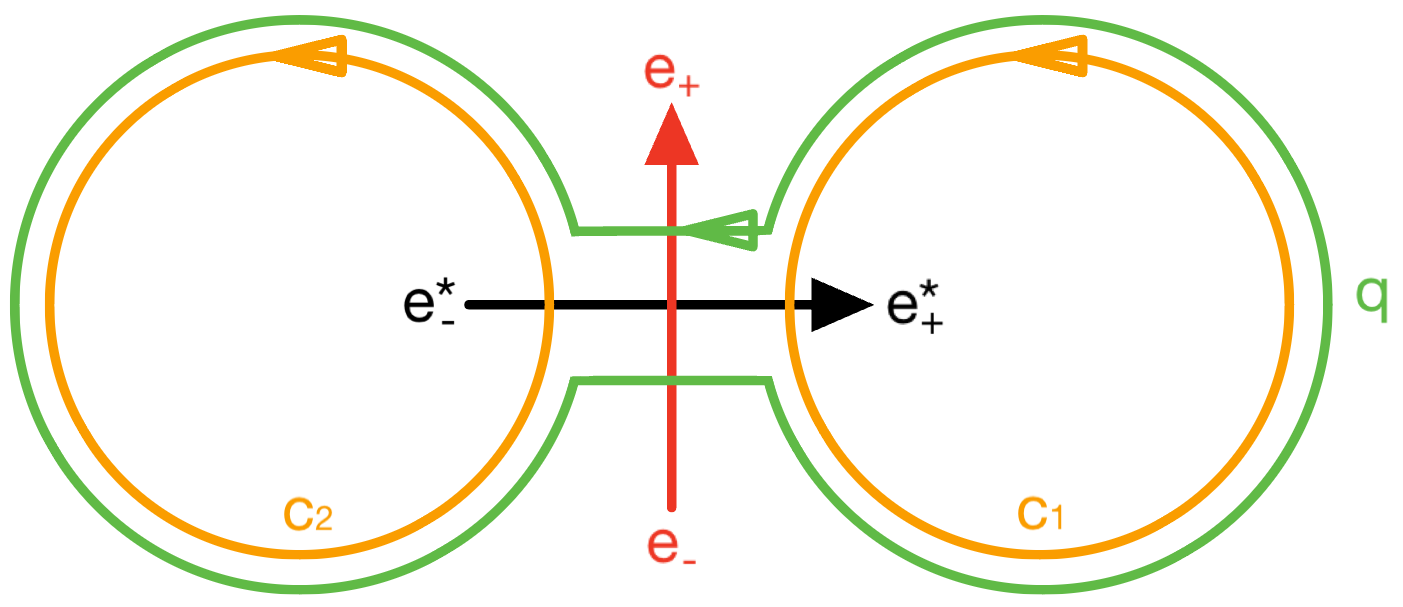}};
        \node at (7,0) {\includegraphics[scale=0.35]{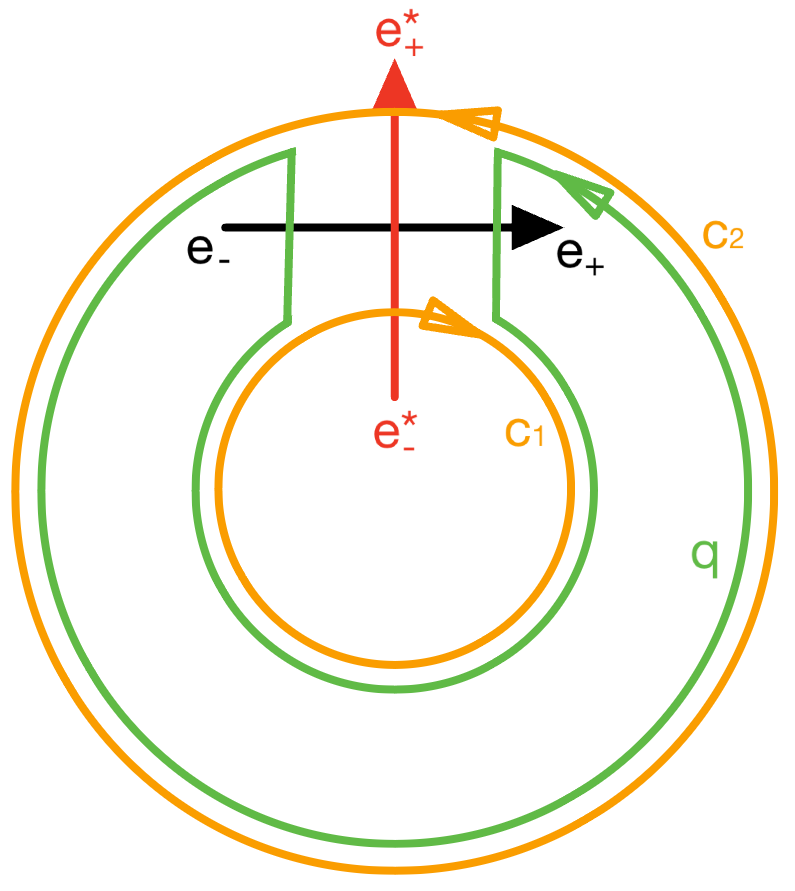}};
    \end{tikzpicture}
    \caption{Descriptions of $q$, $c_1$, and $c_2$ in the proof of Proposition~\ref{prop:compatibility}.
    The left figure illustrates the case $e^* \in Q$, and the right figure illustrates the case $e^* \in (E\setminus Q)^*$.
    In both the left and right illustrations, the region bounded by a green curve~$q$ corresponds to a small $\epsilon$-neighborhood of the spanning quasi-tree $(V,Q)$.
    }
    \label{fig: bouchet}
\end{figure}

Proposition~\ref{prop:compatibility}, combined with Lemma~\ref{lem: equal row span}, shows that 
\begin{align} \label{eq:criticalgroup}
    \J(\vG) = \Z^E / \langle I + A(\vG,Q) \rangle
\end{align}
for any spanning quasi-tree $Q$.
We note that the right-hand side of \eqref{eq:criticalgroup} is the definition of ``critical group'' in~\cite{MMN2023}.

\begin{remark}
    The matrix $A(\vec{\mathbb{G}},Q)$, as defined in~{\cite[Definition~3.1]{MMN2023}}, is the negative of $A(\vG,Q)$.
    However, both matrices yield the same regular orthogonal representation of $M(\vG)$, as the convention for choosing the dual orientation in~\cite{MMN2023} differs from ours.
    A ``combinatorial map'' $\vec{\mathbb{G}} = (\sigma,\alpha,\phi)$, corresponding to our convention, aligns with the definition in~{\cite[Page~4]{MMN2023}}, except that condition~(ii) should be replaced with $\phi = \sigma \alpha$.
\end{remark}

\begin{remark}
    A regular orthogonal matroid $M$ can be represented as $\begin{pmatrix}
        I & A
    \end{pmatrix}$ for some PU skew-symmetric matrix $A$, as we discussed in \S\ref{subsec: regular om}.
    It is easy to see that $\begin{pmatrix}
        I & -A
    \end{pmatrix}$ represents the same regular orthogonal matroid $M$.
    When we denote the corresponding regular representations of $\begin{pmatrix}
        I & A
    \end{pmatrix}$ and $\begin{pmatrix}
        I & -A
    \end{pmatrix}$ by $\cC_1$ and $\cC_2$, respectively, we deduce an isomorphism
    \[
        \J(\cC_1) = \Z^E/\langle I+A \rangle \cong \Z^E/\langle I-A \rangle = \J(\cC_2),
    \]
    because $I-A$ is the transpose of $I+A$.
    However, unlike the reorientation case discussed in \S\ref{subsec: reorientation invariance}, we do not have an automorphism of $\Z^E$ which sends each signed circuit of $\cC_1$ to a signed circuit of $\cC_2$ with the same support, thereby inducing a group isomorphism $\J(\cC_1) \to \J(\cC_2)$ and a bijection $\cR(\cC_1) \to \cR(\cC_2)$. 
    
    For an explicit example, let
    \[
        A = \begin{pmatrix}
            0 & 1 & 1 & 1 \\
            -1 & 0 & 1 & 1 \\
            -1 & -1 & 0 & 0 \\
            -1 & -1 & 0 & 0 \\
        \end{pmatrix}
    \]
    and let $f : \Z^4 \to \Z^4$ be a map sending each $\ve_i + \vr_i$ to $\epsilon_i (\ve_i - \vr_i)$, where $\vr_i$ is the $i$-th row of $A$ and $\epsilon_i \in \{\pm 1\}$.
    Then $f$ is identified with the matrix $(I+A)^{-1} J (I-A)$, where $J = \mathrm{diag}(\epsilon_1,\epsilon_2,\epsilon_3,\epsilon_4)$.
    However, one checks easily that $f$ is not an integral matrix for any choice of $\epsilon_i$'s.
\end{remark}

\end{document}